\documentclass[12pt,reqno]{amsart} 
\usepackage{amsmath, amssymb, mathscinet,mathtools}
\usepackage{amsthm} 
\usepackage[all]{xy}
\usepackage{enumerate}
\usepackage[T1]{fontenc}

\usepackage{hyperref} 
\usepackage{mathrsfs} 

\usepackage{graphicx}
\usepackage{color}
\usepackage{comment}

\usepackage{titlesec}

\titleformat{\section}{\normalfont\bfseries\large}{\thesection.}{0.5em}{}
\titleformat{\subsection}{\normalfont\bfseries}{\thesubsection.}{0.5em}{}
\titleformat{\subsubsection}{\normalfont\bfseries}{\thesubsubsection.}{0.5em}{}

\theoremstyle{plain} 
\theoremstyle{plain}
\newtheorem{theorem}{Theorem}[section]
\newtheorem{lemma}[theorem]{Lemma}
\newtheorem{corollary}[theorem]{Corollary}
\newtheorem{proposition}[theorem]{Proposition}

\theoremstyle{definition}
\newtheorem{definition}[theorem]{Definition}
\newtheorem{remark}[theorem]{Remark}
\newtheorem{example}[theorem]{Example}

\newtheorem{conjecture}[theorem]{Conjecture}
\newtheorem{question}[theorem]{Question}

\newcommand{\N}{\mathbb{N}}
\def\rank{\mathop{\mathrm{rank}}\nolimits}
\def\dim{\mathop{\mathrm{dim}}\nolimits}

\def\det{\mathop{\mathrm{det}}\nolimits}

\newcommand{\supp}{\mathop{\mathrm{supp}}\nolimits}

\makeatletter
\def\address#1#2{\begingroup
	\noindent\parbox[t]{7.8cm}{%
		\small{\scshape\ignorespaces#1}\par\vskip1ex
		\noindent\small{\itshape E-mail address}%
		\/: #2\par\vskip4ex}\hfill%
	\endgroup}%
\makeatother
\title[]{\uppercase{Entropy on Homogeneous Spaces and Classification Results for Subgroups with the Pair Rapid Decay Property
}}
\author{
	\textsc{Jvbin Yao} 
}
\date{} 
\begin{document}
	
	\maketitle

	\begin{abstract}
		
		We study pair rapid decay for homogeneous spaces \(G/H\) and its applications to random walks and subgroup structure. The entropy framework for groups with rapid decay is extended to homogeneous spaces, proving that the asymptotic Shannon entropy on \(G/H\) agrees with a spectral-radius quantity \(c(G,H;\mu)\) for measures with finite entropy and suitable finite moment, and that the lower and upper asymptotic R\'enyi entropy rates converge to the Shannon entropy as \(\alpha\downarrow1\). For finitely supported measures, we also obtain a spectral-radius formula for the asymptotic R\'enyi entropy rates \(h_\alpha(X,\mu)\), \(\alpha\in(1,2]\), and hence continuity at \(\alpha=1\). We further introduce the notion of subexponential Lorentz control for pairs \((G,H)\) and study the associated classification problems for finitely generated subgroups \(H\le G\) for which \((G,H)\) has pair rapid decay or belongs to \(\mathbf{SLC}_{\mathrm{subexp}}\). We obtain a complete criterion in the strongly relatively hyperbolic case and explicit classifications in several hyperbolic settings. We also show that for \(G=\mathrm{SL}_n(\mathbb Z)\), \(n\ge3\), the conditions \((G,H)\in \mathbf{SLC}_{\mathrm{subexp}}\), pair rapid decay, and finite index of \(H\) in \(G\) are equivalent.

	\end{abstract}

	\section{Introduction}
	\label{sec:introduction} 
	
	The rapid decay property (RD) was first proved for free groups by Haagerup \cite{Haagerup1979}. Its systematic study goes back to Jolissaint \cite{Jolissaint1990}, who introduced the property in its modern form and established it for groups of polynomial growth; hyperbolic groups were subsequently shown to satisfy RD by work of de la Harpe and Jolissaint \cite{deLaHarpe1988,Jolissaint1990}. Since then, RD has been proved for many further classes of geometric groups, including relatively hyperbolic groups relative to subgroups with RD \cite{DrutuSapir2005}, mapping class groups \cite{BehrstockMinsky2011}, and various cocompact lattices in semisimple Lie groups \cite{Lafforgue2000}. At the same time, RD is far from ubiquitous: groups containing amenable subgroups of superpolynomial growth do not satisfy RD \cite{Jolissaint1990}. In particular, \(\mathrm{SL}(n,\mathbb Z)\) for \(n\ge3\), and more generally many nonuniform lattices in higher-rank semisimple Lie groups, lie outside the usual RD framework. These examples suggest that the classical formulation of RD is often too rigid to capture interesting quotient and relative settings. Motivated by this, Chatterji and Zarka introduced a pair version of RD by replacing the left-regular representation of \(G\) with the quasi-regular representation on \(G/H\) \cite{ChatterjiZarka2024v1}. This relative formulation is flexible enough to capture situations in which the ambient group itself does not satisfy RD, while the homogeneous space \(G/H\) still satisfies a strong decay estimate.

	 The first goal of this paper is analytic. We extend to homogeneous spaces the entropy framework recently developed for groups with RD in \cite{anderson2024}. The key input is a weighted interpolation estimate for quasi-regular representations on \(\ell^q(G/H)\), obtained under the pair rapid decay assumption of \cite{ChatterjiZarka2024v1}. This yields entropy identities on \(G/H\), shows that the lower and upper asymptotic R\'enyi entropy rates converge to the Shannon entropy as \(\alpha\downarrow1\), and when \(\mu\) is finitely supported, gives spectral-radius formulas for the asymptotic R\'enyi entropy rates and continuity of \(h_\alpha(X,\mu)\) at \(\alpha=1\).

	To state the result, we write \(\underline{h}(X,\mu)\) and \(\overline{h}(X,\mu)\) for the lower and upper asymptotic entropy of the random walk on \(X=G/H\), \(h_\alpha(X,\mu)\) for the asymptotic R\'enyi entropy rate of order \(\alpha>1\), \(r_\alpha(\mu)\) for the spectral radius of the Markov operator on \(\ell^\alpha(X)\), and \(c(G,H;\mu)\) for the corresponding spectral-radius quantity introduced below. Precise definitions are given in Section~\ref{sec:entropy}.

	\begin{theorem}\label{thm:A}
		Let \(X=G/H\). Assume that \((G,H)\) has pair rapid decay, and let \(\mu\) be a probability measure on \(G\) with finite entropy and finite \(\beta\)-moment with respect to \(1+\ell\) for some \(\beta\ge 2s_1\). Then
		\[
		\underline{h}(X,\mu)=\overline{h}(X,\mu)=c(G,H;\mu).
		\]
		Hence the asymptotic Shannon entropy $
		h(X,\mu)\coloneqq \lim_{n\to\infty}\frac1n H(\mu^{*n}*\delta_{eH})$
		exists, and
		\[
		\lim_{\alpha\downarrow1}\underline{h}_{\alpha}(X,\mu)
		=
		\lim_{\alpha\downarrow1}\overline{h}_{\alpha}(X,\mu)
		=
		h(X,\mu).
		\]
		If \(\mu\) is finitely supported, then for every \(\alpha\in(1,2]\) the limit defining \(h_\alpha(X,\mu)\) exists and satisfies
		$h_\alpha(X,\mu)=\frac{\alpha}{1-\alpha}\log r_\alpha(\mu).$
		In particular,
		\[
		\lim_{\alpha\downarrow1} h_\alpha(X,\mu)=h(X,\mu).
		\]
	\end{theorem}

	 In particular, when \(H=\{e\}\), Theorem~\ref{thm:A} gives a partial affirmative answer to the continuity problem raised in \cite{GolubevaPanTamuz2024} for random walks on non-amenable groups.

	The second goal of the paper is structural. The proof of the entropy results shows that several arguments require much less than polynomial rapid decay. This motivates a weaker notion, called \emph{subexponential Lorentz control}, denoted \((G,H)\in \mathbf{SLC}_{\mathrm{subexp}}\), in which the polynomial weight in pair rapid decay is replaced by a radial weight of subexponential growth. This weaker condition is still strong enough for the entropy-theoretic applications developed here, but flexible enough to occur well beyond the usual rapid decay setting.
	
	A natural problem is therefore to classify, for a given finitely generated group \(G\), those finitely generated subgroups \(H\le G\) for which \((G,H)\) has pair rapid decay, or more generally satisfies \((G,H)\in \mathbf{SLC}_{\mathrm{subexp}}\). In the co-amenable case, Chatterji and Zarka obtained a satisfactory answer: if \(H\) is co-amenable in \(G\), then \((G,H)\) has pair rapid decay if and only if the Schreier graph \(G/H\) has polynomial growth \cite[Theorem~1.2]{ChatterjiZarka2024v1}. Outside the co-amenable case, however, no comparable general classification seems to be known. Our approach is therefore case-by-case: we first treat classes for which a clean description is possible, and then turn to broader families where one can still isolate strong necessary conditions or rigidity phenomena.

	Our first structural result provides a more precise formulation of a statement from \cite{ChatterjiZarka2024v1} and gives a complete criterion in the strongly relatively hyperbolic setting.

    \begin{theorem}\label{thm:B}
	Let \(G\) be a finitely generated group, and let \(H<G\) be a proper finitely generated subgroup. Assume that \(G\) is strongly relatively hyperbolic with respect to \(H\). Then the pair \((G,H)\) has pair rapid decay if and only if \(H\) has polynomial growth with respect to the induced length \(\ell_G|_H\).
    \end{theorem}

    We then turn to classification results in the classes considered here. In several important families of groups, the finitely generated subgroups \(H\le G\) for which \((G,H)\in \mathbf{SLC}_{\mathrm{subexp}}\) admit an explicit description.

 \begin{theorem}\label{thm:C}
 	Let \(H\le G\) be a finitely generated subgroup. Then the following hold.
 	
 	\begin{enumerate}
 		\item If \(G=F_n\) for some \(n\ge 2\), or if \(G=\pi_1(\Sigma_g)\), where \(\Sigma_g\) is a closed orientable surface of genus \(g\ge 2\), then \((G,H)\in \mathbf{SLC}_{\mathrm{subexp}}\) if and only if \(H\) has finite index in \(G\), or \(H\cong \mathbb Z\), or \(H=\{e\}\).
 		
 		\item More generally, if \(G=\mathrm{SL}(2,\mathbb Z)\), or if \(G\) is a non-elementary finitely generated locally quasiconvex hyperbolic group, or if \(G\) is a hyperbolic group with \(\beta_1^{(2)}(G)>0\), then \((G,H)\in \mathbf{SLC}_{\mathrm{subexp}}\) if and only if \(H\) has finite index in \(G\) or \(H\) is virtually cyclic.
 		
 		\item If \(M\) is a closed hyperbolic \(3\)-manifold and \(G=\pi_1(M)\), then \((G,H)\in \mathbf{SLC}_{\mathrm{subexp}}\) if and only if \(H\) is virtually cyclic or \(\operatorname{Core}_G(H)\neq \{e\}\).
 		
 		\item If \(G<\mathrm{Isom}(\mathbb H^3)\) is a non-elementary finitely generated discrete subgroup, then \((G,H)\in \mathbf{SLC}_{\mathrm{subexp}}\) if and only if \(H\) is virtually cyclic, or \(H\) is a rank-\(2\) parabolic subgroup, or \(H\) is \(s\)-normal in \(G\).
 	\end{enumerate}
 \end{theorem}

   Theorem~\ref{thm:C} shows that, in the classes considered here, the condition \((G,H)\in \mathbf{SLC}_{\mathrm{subexp}}\) is highly restrictive.
    
    More generally, for relatively hyperbolic groups we obtain strong necessary conditions for \((G,H)\in \mathbf{SLC}_{\mathrm{subexp}}\). In particular, a finitely generated subgroup \(H\) must be \(s\)-normal, parabolic, or virtually cyclic. We also derive concrete obstructions in terms of limit set and ambient growth, and show that for relatively quasiconvex subgroups, after factoring out the peripheral part, what remains is virtually cyclic.

    The final part of the paper turns to higher-rank nonuniform lattices. Our main result there treats the model case \(G=\mathrm{SL}_n(\mathbb Z)\).
    
    \begin{theorem}\label{thm:D}
    	Let \(n\ge3\), let \(G=\mathrm{SL}_n(\mathbb Z)\), and let \(H\le G\) be a subgroup. Then the following are equivalent:
    	\begin{enumerate}
    		\item \((G,H)\in \mathbf{SLC}_{\mathrm{subexp}}\);
    		\item \((G,H)\) has pair rapid decay;
    		\item \(H\) has finite index in \(G\).
    	\end{enumerate}
    \end{theorem}
    
    Thus, in this higher-rank example, the only subgroups \(H\le G\) for which \((G,H)\in \mathbf{SLC}_{\mathrm{subexp}}\), equivalently for which \((G,H)\) has pair rapid decay, are the finite-index subgroups.

    Beyond the concrete results above, the paper points to a natural structural question for non-elementary hyperbolic groups. All presently known finitely generated subgroups \(H\le G\) of infinite index with \((G,H)\in \mathbf{SLC}_{\mathrm{subexp}}\) and \(H\) not virtually cyclic have nontrivial normal core. This suggests that, in the hyperbolic setting, nontrivial normal core may be the only source of such examples.

	On the analytic side, the weighted interpolation inequality is used to convert pair rapid decay or \(\mathbf{SLC}_{\mathrm{subexp}}\) type estimates into entropy identities and spectral radius formulas. On the geometric side, a recurring strategy is to show that \((G,H)\in \mathbf{SLC}_{\mathrm{subexp}}\) forces strong control on the growth of \(H\) through the intersections \(H\cap xHx^{-1}\). In hyperbolic settings, this is combined with limit-set methods, quasiconvexity, and, in some cases, subgroup tameness and virtual fibration, to force strong restrictions on \(H\), including nontrivial normal core. In higher rank, we begin by analyzing the pair \(\bigl(\mathrm{SL}_3(\mathbb Z),\mathrm{UT}_3(\mathbb Z)\bigr)\). The proof in this example suggests that the main obstruction in the general case should come from a suitable parabolic subgroup \(P\). We therefore isolate and study this parabolic model first, and the argument developed there then leads to the proof for \(\mathrm{SL}_n(\mathbb Z)\), \(n\ge3\).

	The paper is organized as follows. In Section~\ref{sec:entropy}, we develop the interpolation machinery for quasi-regular representations and prove the entropy and R\'enyi entropy results on homogeneous spaces. In Subsection~\ref{subsec:hyperbolic}, we introduce \(\mathbf{SLC}_{\mathrm{subexp}}\) and study the subgroup-classification problem for free groups, hyperbolic groups, and relatively hyperbolic groups. In Subsection~\ref{subsec:lattices}, we study the subgroup-classification problem in the model higher-rank case \(G=\mathrm{SL}_n(\mathbb Z)\).

   \bigskip

	\section{Entropy on Homogeneous Spaces}\label{sec:entropy}

	In this section, we develop the analytic framework for random walks on homogeneous spaces \(G/H\). We begin by introducing the basic setup and extending the relevant entropy-theoretic notions from groups to homogeneous spaces. We then establish a weighted interpolation estimate for quasi-regular representations on \(\ell^q(G/H)\), which will serve as a key tool in the sequel, and apply it to the study of asymptotic Shannon and R\'enyi entropy.

	Let $G$ be a countable finitely generated discrete group and let $H \le G$ be a subgroup. We consider the homogeneous space defined by the quotient $X \coloneqq G/H = \{gH : g \in G\}$.
   Fix a finite symmetric generating set $S$ for $G$.  Let $\ell = \ell_S : G \to \N \cup \{0\}$ denote the word length function associated with $S$, defined by $\ell(g) = \min\{n : g = s_1 \dots s_n,\ s_i \in S\}$.
	
	Let $\mu$ be a probability measure on $G$, and consider the random walk on $X$ induced by $\mu$. For part of the discussion below, we shall assume that $\mu$ has finite support, namely that $\supp(\mu) \coloneqq \{g \in G : \mu(g) > 0\}$
	is finite. Later, we shall also consider more general probability measures under suitable entropy and moment assumptions.

	We consider the canonical projection map $\pi: G \to X$ given by $\pi(g) = gH$. For any non-negative function $f: G \to [0, \infty)$, we define the \emph{fiber summation} (or push-forward) operator $\pi_{\#} f : X \to [0, \infty]$ by
	$$
	(\pi_{\#}f)(x) \coloneqq \sum_{g \in \pi^{-1}(x)} f(g) = \sum_{gH=x} f(g), \quad \text{for } x \in X.
	$$
	
	Let $(Z_n)_{n \ge 0}$ be the random walk on $G$ defined by $Z_0 = e$ and
	$$
	Z_n = s_n s_{n-1} \dots s_1,
	$$
	where the increments $s_i$ are independent and identically distributed (i.i.d.) random variables with distribution $\mu$. We examine the trajectory of this random walk projected onto the homogeneous space $X$: $X_n \coloneqq Z_n H \in X$.
	
	Let $\nu_n \in \mathcal{P}(X)$ denote the distribution of $X_n$. Explicitly, with $o=eH$,
	$$
	\nu_n \coloneqq \mathrm{Law}(X_n) = \pi_{\#}(\mu^{*n}) = \mu^{*n} * \delta_o,
	$$
	where the convolution on $\mathcal{P}(X)$ is induced by the left action of $G$ on $X=G/H$.

   For any probability measure $\nu$ on $X$, the Shannon entropy is defined as
   $$
   H(\nu) \coloneqq -\sum_{x \in X} \nu(x) \log \nu(x),
   $$
   with the standard convention that $0 \log 0 = 0$. We define the lower and upper asymptotic entropies of the random walk on the quotient space $G/H$ as
   $$
   \underline{h}(G/H, \mu) \coloneqq \liminf_{n \to \infty} \frac{1}{n} H(\nu_n),
   $$
   and
   $$
   \overline{h}(G/H, \mu) \coloneqq \limsup_{n \to \infty} \frac{1}{n} H(\nu_n).
   $$
   
   For $1 \le q \le \infty$, let $\lambda_{X,q}$ denote the \emph{quasi-regular representation} of $G$ on the Banach space $\ell^q(X)$, defined by
   $$
   (\lambda_{X,q}(g)\xi)(x) = \xi(g^{-1}x), \quad \text{for all } \xi \in \ell^q(X), \, g \in G, \, x \in X.
   $$
   This action extends linearly to the group algebra $\mathbb{C}G$ (the space of formal linear combinations of group elements with complex coefficients). For any $f \in \mathbb{C}G$, the convolution operator is defined as
   $$
   \lambda_{X,q}(f) \coloneqq \sum_{g \in G} f(g)\lambda_{X,q}(g).
   $$
   We define the \emph{$q$-pseudofunction algebra} associated with the pair $(G, H)$, denoted by $PF_q(G,H)$, as the closure of the image of the group algebra under the operator norm of $B(\ell^q(X))$:
   $$
   PF_q(G,H) \coloneqq \overline{\lambda_{X,q}(\mathbb{C}G)}^{\lVert \cdot \rVert_{B(\ell^q(X))}}.
   $$
   
   Let $p \in [2, \infty)$ and let $q = \frac{p}{p-1} \in (1, 2]$ be its conjugate exponent. We focus on the spectral properties of the probability measure $\mu$ when viewed as an operator within this algebra.
   
   We denote the spectral radius of $\lambda_{X,q}(\mu)$ in $PF_q(G,H)$ by $r_q(\mu)$. Explicitly,
   $$
   r_q(\mu) \coloneqq r_{PF_q(G,H)}(\lambda_{X,q}(\mu)) = \lim_{n \to \infty} \lVert \lambda_{X,q}(\mu)^n \rVert_{B(\ell^q(X))}^{1/n}.
   $$
   Note that here $\mu$ is identified with the element $\sum_{g \in G} \mu(g)g \in \mathbb{C}G$, and hence with its image $\lambda_{X,q}(\mu)$ in $PF_q(G,H)$.

   \begin{lemma}
   	Assume that $\mu$ is a probability measure on $G$ such that $r_q(\mu)>0$ for every $q\in(1,2]$. Then the function
   	$$
   	p \mapsto -p \log r_{p/(p-1)}(\mu)
   	$$
   	is monotonically increasing for $p \in [2,\infty)$.
   \end{lemma}

   \begin{proof}
   	Let $1<u<q\le2$, and let $p_u$ and $p_q$ be the conjugate exponents of $u$ and $q$, respectively. Then $p_u>p_q$.
   	
   	Consider the Markov operator
   	$$
   	T\xi(x)\coloneqq \sum_{g\in G}\mu(g)\,\xi(g^{-1}x).
   	$$
   	For each $r\in[1,\infty]$, this defines a contraction on $\ell^r(X)$. In particular,
   $\|T\|_{B(\ell^1(X))}\le1.$
   	For $q\in(1,2]$, we identify $T$ with $\lambda_{X,q}(\mu)$ on $\ell^q(X)$, so its spectral radius is $r_q(\mu)$.
   	By the Riesz--Thorin interpolation theorem, there exists $\theta\in(0,1)$ such that
   	$$
   	\|T^n\|_{B(\ell^u(X))}
   	\le
   	\|T^n\|_{B(\ell^1(X))}^{1-\theta}
   	\|T^n\|_{B(\ell^q(X))}^{\theta}
   	\le
   	\|T^n\|_{B(\ell^q(X))}^{\theta},
   	$$
   	with $1/u=(1-\theta)+\theta/q$.
   	Taking $n$-th roots and letting $n\to\infty$, we obtain
   	$$
   	r_u(\mu)\le r_q(\mu)^\theta.
   	$$
   	Since $1/u=1-1/p_u$ and $1/q=1-1/p_q$, the relation $1/u=(1-\theta)+\theta/q$ yields $\theta=p_q/p_u$. Hence
   	$$
   	r_u(\mu)\le r_q(\mu)^{p_q/p_u},
   	$$
   	and therefore
   	$$
   	-p_u\log r_u(\mu)\ge -p_q\log r_q(\mu).
   	$$
   	This proves that $p\mapsto -p\log r_{p/(p-1)}(\mu)$ is monotonically increasing on $[2,\infty)$.
   \end{proof}



\medskip

By the previous lemma, the limit $\lim_{p \to \infty} \bigl(-p \log r_{p/(p-1)}(\mu)\bigr)$ exists in $[0,\infty]$.

\begin{definition}
	We define
	$$
	c(G,H;\mu) \coloneqq \lim_{p \to \infty} \bigl(-p \log r_{p/(p-1)}(\mu)\bigr).
	$$
\end{definition}

\begin{proposition}\label{prop:rq-positive}
	Let $\mu$ be a probability measure on $G$ with finite entropy. Then for every $q \in (1,2]$, one has $0<r_q(\mu)\le 1.$
	In particular,  $c(G,H;\mu)\ge 0$.
\end{proposition}

\begin{proof}
	Let $T_q \coloneqq \lambda_{X,q}(\mu)$. Since each operator $\lambda_{X,q}(g)$ acts isometrically on $\ell^q(X)$, we have
	$$
	\lVert T_q \rVert_{B(\ell^q(X))}
	\le \sum_{g \in G} \mu(g)\,\lVert \lambda_{X,q}(g) \rVert_{B(\ell^q(X))}
	= \sum_{g \in G} \mu(g)
	= 1.
	$$
	It follows that $r_q(\mu) \le 1$.
	
	To prove positivity, let $o=eH$ and let $\delta_o \in \ell^q(X)$ be the Dirac mass at $o$. Then $\lVert \delta_o \rVert_q = 1$ and
	$$
	T_q^n \delta_o = \nu_n.
	$$
	Hence
	$$
	\lVert T_q^n \rVert_{B(\ell^q(X))} \ge \lVert T_q^n \delta_o \rVert_q = \lVert \nu_n \rVert_q.
	$$
	
	Since $H(\mu)<\infty$, we have
	$$
	H(\nu_n)=H(X_n)\le H(Z_n)\le nH(\mu)<\infty.
	$$
	For a probability measure $\nu$ on $X$, define its Rényi entropy of order $q$ by
	$$
	H_q(\nu) \coloneqq \frac{1}{1-q} \log \left( \sum_{x \in X} \nu(x)^q \right).
	$$
	Since $q = \frac{p}{p-1}$, we have $\frac{q}{1-q} = -p$, and hence
	$$
	H_q(\nu) = \frac{1}{1-q} \log \bigl(\lVert \nu \rVert_q^q\bigr) = -p \log \lVert \nu \rVert_q.
	$$
	We claim that $H(\nu) \ge H_q(\nu)$ for every probability measure $\nu$ on $X$. Indeed, by the concavity of the logarithm,
	$$
	(q-1)\sum_{x \in X} \nu(x)\log \nu(x)
	= \sum_{x \in X} \nu(x)\log\bigl(\nu(x)^{q-1}\bigr)
	\le \log\left(\sum_{x \in X} \nu(x)^q\right).
	$$
	Multiplying by $-1/(q-1)$ yields $H(\nu) \ge H_q(\nu)$.
	
	Applying this to $\nu_n$, and noting that $\frac{q}{1-q}<0$, we obtain
	$$
	\log \|\nu_n\|_q
	\ge
	\frac{1-q}{q} H(\nu_n)
	\ge
	\frac{1-q}{q}\, n H(\mu).
	$$
	Therefore
	$$
	\|\nu_n\|_q^{1/n}
	\ge
	\exp\!\left(\frac{1-q}{q}H(\mu)\right)
	>0.
	$$
	Taking the limit superior, we get
	$$
	r_q(\mu)
	=
	\lim_{n \to \infty} \lVert T_q^n \rVert_{B(\ell^q(X))}^{1/n}
	\ge
	\limsup_{n \to \infty} \lVert \nu_n \rVert_q^{1/n}
	\ge
	\exp\!\left(\frac{1-q}{q}H(\mu)\right)
	>0.
	$$
	Thus $0 < r_q(\mu) \le 1$.
\end{proof}

	\begin{lemma}\label{lem:entropy_inequality}
		Assume that $\mu$ has finite support. Then the lower asymptotic entropy of the random walk on $G/H$ dominates $c(G,H;\mu)$, namely,
		$$
		\underline{h}(G/H,\mu) \ge c(G,H;\mu).
		$$
	\end{lemma}

\begin{proof}
	Fix $p \ge 2$, and set $q = \frac{p}{p-1} \in (1,2]$. Let $T_q \coloneqq \lambda_{X,q}(\mu)$. Since $\nu_n = T_q^n \delta_o$ and $\lVert \delta_o \rVert_q = 1$, we have
	$$
	\lVert \nu_n \rVert_q \le \lVert T_q^n \rVert_{B(\ell^q(X))}.
	$$
	Hence $\limsup_{n \to \infty} \lVert \nu_n \rVert_q^{1/n} \le r_q(\mu)$.
	
	
	
	And we also have
	$$
	\frac{1}{n} H(\nu_n) \ge \frac{1}{n} H_q(\nu_n) = -p \log \bigl(\lVert \nu_n \rVert_q^{1/n}\bigr).
	$$
	Let $L \coloneqq \limsup_{n \to \infty} \lVert \nu_n \rVert_q^{1/n}$. Then for every $\varepsilon > 0$, one has $\lVert \nu_n \rVert_q^{1/n} \le L+\varepsilon$ for all sufficiently large $n$, and therefore
	$$
	\frac{1}{n} H(\nu_n) \ge -p \log(L+\varepsilon)
	$$
	for all sufficiently large $n$. Passing to the limit inferior gives
	$$
	\underline{h}(G/H,\mu) \ge -p \log(L+\varepsilon).
	$$
	Letting $\varepsilon \downarrow 0$, we obtain
	$$
	\underline{h}(G/H,\mu) \ge -p \log L \ge -p \log r_q(\mu),
	$$
	where the last inequality follows from $L \le r_q(\mu)$ and the fact that $x \mapsto -p\log x$ is decreasing on $(0,\infty)$.
	
	Since this holds for every $p \ge 2$, we may let $p \to \infty$ and conclude from the definition of $c(G,H;\mu)$ that
	$$
	\underline{h}(G/H,\mu) \ge c(G,H;\mu).
	$$
\end{proof}


\subsection{Interpolation Estimate}\label{subsec:Interpolation-Estimate}

We now proceed to establish a norm inequality for the quasi-regular representation by interpolating between two endpoint cases. Our goal is to prove an estimate of the form:
\begin{equation} \label{eq:target_estimate}
	\lVert \lambda_{X,q}(f) \rVert_{B(\ell^q(X))} \le C^{1/p} \lVert \pi_{\#}(|f|\omega^{s/p}) \rVert_q,
\end{equation}
where $\omega(g) \coloneqq 1 + \ell(g)$ is the weighted length function, and $s > 0$ is a parameter related to the decay rate.

Let us define the bilinear map $B(f, \xi) \coloneqq \lambda_{X}(f)\xi, \quad \text{for } f \in \mathbb{C}G, \, \xi \in \ell^1(X) \cap \ell^2(X).$ and we refer to the standard theory of interpolation spaces (\cite[Theorems 4.4.1 and 5.1.1]{bergh1976}).


\begin{lemma}\label{lem:interp-norm-est-weighted}
	Let $0<\theta<1$ and let $q$ be determined by
	\[
	\frac1q=\frac{1-\theta}{1}+\frac{\theta}{2}=1-\frac{\theta}{2}.
	\]
	Let $w:G\to(0,\infty)$ be any strictly positive function and let $\beta_0,\beta_1\in\mathbb{R}$.
	Define norms on $\mathbb{C}G$ by
	\[
	\|f\|_{A_0}\coloneqq \bigl\|\pi_\#\bigl(|f|\,w^{\beta_0}\bigr)\bigr\|_{\ell^1(X)},
	\qquad
	\|f\|_{A_1}\coloneqq \bigl\|\pi_\#\bigl(|f|\,w^{\beta_1}\bigr)\bigr\|_{\ell^2(X)},
	\]
	and let $A_0$ and $A_1$ denote the completions of $\mathbb{C}G$ with respect to these norms.
	Set
	\[
	\beta_\theta\coloneqq (1-\theta)\beta_0+\theta\beta_1.
	\]
	Then for every $f\in\mathbb{C}G$ one has
	\begin{equation}\label{eq:interp-norm-est-weighted}
		\|f\|_{(A_0,A_1)_{[\theta]}}
		\le
		\bigl\|\pi_\#\bigl(|f|\,w^{\beta_\theta}\bigr)\bigr\|_{\ell^q(X)}.
	\end{equation}
\end{lemma}

\begin{proof}
	Set
	\[
	c\coloneqq \pi_\#\bigl(|f|\,w^{\beta_\theta}\bigr)\in \ell^{q}(X),
	\qquad
	c(x)=\sum_{g\in x}|f(g)|\,w(g)^{\beta_\theta},
	\]
	and $C\coloneqq \|c\|_{q}=\Bigl(\sum_{x\in X}c(x)^q\Bigr)^{1/q}$.
	If $C=0$, then $c\equiv 0$ and hence $f\equiv 0$, so \eqref{eq:interp-norm-est-weighted} is trivial. Assume $C>0$.
	
	For each $x\in X$, define on the closed strip $S=\{z\in\mathbb{C}:0\le \Re z\le 1\}$,
	\[
	a_x(z)\coloneqq
	\begin{cases}
		\exp\!\Bigl(\bigl(q(1-\tfrac{z}{2})-1\bigr)\log\bigl(\tfrac{c(x)}{C}\bigr)\Bigr), & c(x)>0,\\[1mm]
		1, & c(x)=0,
	\end{cases}
	\]
	where $\log$ denotes the real logarithm on $(0,\infty)$. Define $F:S\to \mathbb{C}G$ by
	\begin{equation}\label{eq:Fz-def-weighted}
		F(z)(g)\coloneqq a_x(z)\,f(g)\,w(g)^{\,\beta_\theta-\beta_0-z(\beta_1-\beta_0)}
		\qquad (g\in x,\ x=\pi(g)\in X).
	\end{equation}
	Equivalently, since $\beta_\theta-\beta_0=\theta(\beta_1-\beta_0)$,
	\[
	F(z)(g)=a_x(z)\,f(g)\,w(g)^{\,(\theta-z)(\beta_1-\beta_0)}.
	\]
	Then $F(z)\in A_0+A_1$ for every $z\in S$, and $F$ is bounded and continuous on $S$ and analytic on the interior strip.
	
	As usual, we enforce boundary decay by a Gaussian. Fix $\varepsilon>0$ and set
	\[
	\psi_\varepsilon(z)\coloneqq \exp\!\bigl(\varepsilon(z-\theta)^2\bigr),
	\qquad
	\widetilde F_\varepsilon(z)\coloneqq \psi_\varepsilon(z)F(z).
	\]
	Then $\widetilde F_\varepsilon\in \mathcal{F}(A_0,A_1)$ and $\widetilde F_\varepsilon(\theta)=F(\theta)$.
	Moreover, since $q(1-\theta/2)-1=0$, we have $a_x(\theta)=1$ for all $x$, and thus
	\[
	F(\theta)(g)=f(g)\,w(g)^{(\theta-\theta)(\beta_1-\beta_0)}=f(g),
	\]
	so $\widetilde F_\varepsilon(\theta)=f$.
	
	\medskip
	\noindent\textbf{Boundary $\Re z=0$.}
	Let $z=it$. Then
	\[
	|a_x(it)|=
	\begin{cases}
		\bigl(\tfrac{c(x)}{C}\bigr)^{q-1}, & c(x)>0,\\
		1, & c(x)=0,
	\end{cases}
	\]
	and since $w(g)>0$ we have $|w(g)^{-it(\beta_1-\beta_0)}|=1$. Therefore, for every $x\in X$,
	\[
	\sum_{g\in x}|F(it)(g)|\,w(g)^{\beta_0}
	=|a_x(it)|\sum_{g\in x}|f(g)|\,w(g)^{\beta_\theta}
	=\Bigl(\tfrac{c(x)}{C}\Bigr)^{q-1}c(x)
	=\frac{c(x)^q}{C^{q-1}}.
	\]
	Hence
	\[
	\|F(it)\|_{A_0}
	=\|\pi_\#(|F(it)|\,w^{\beta_0})\|_{\ell^1(X)}
	=\sum_{x\in X}\sum_{g\in x}|F(it)(g)|\,w(g)^{\beta_0}
	=\frac{1}{C^{q-1}}\sum_{x\in X}c(x)^q
	=C.
	\]
	Consequently,
	\begin{equation}\label{eq:B0-weighted}
		\sup_{t\in\mathbb{R}}\|\widetilde F_\varepsilon(it)\|_{A_0}
		\le \Bigl(\sup_{t\in\mathbb{R}}|\psi_\varepsilon(it)|\Bigr)\Bigl(\sup_{t\in\mathbb{R}}\|F(it)\|_{A_0}\Bigr)
		\le e^{\varepsilon\theta^2}\,C. \tag{$B_0$}
	\end{equation}
	
	\medskip
	\noindent\textbf{Boundary $\Re z=1$.}
	Let $z=1+it$. Then
	\[
	|a_x(1+it)|=
	\begin{cases}
		\bigl(\tfrac{c(x)}{C}\bigr)^{\frac{q}{2}-1}, & c(x)>0,\\
		1, & c(x)=0,
	\end{cases}
	\]
	and again $|w(g)^{-it(\beta_1-\beta_0)}|=1$. Thus, for every $x\in X$,
	\[
	\sum_{g\in x}|F(1+it)(g)|\,w(g)^{\beta_1}
	=|a_x(1+it)|\sum_{g\in x}|f(g)|\,w(g)^{\beta_\theta}
	=\frac{c(x)^{q/2}}{C^{\frac{q}{2}-1}}.
	\]
	Hence
	\[
	\|F(1+it)\|_{A_1}
	=\Biggl(\sum_{x\in X}\Bigl(\sum_{g\in x}|F(1+it)(g)|\,w(g)^{\beta_1}\Bigr)^2\Biggr)^{1/2}
	=\Biggl(\sum_{x\in X}\Bigl(\frac{c(x)^{q/2}}{C^{\frac{q}{2}-1}}\Bigr)^2\Biggr)^{1/2}
	=C.
	\]
	Consequently,
	\begin{equation}\label{eq:B1-weighted}
		\sup_{t\in\mathbb{R}}\|\widetilde F_\varepsilon(1+it)\|_{A_1}
		\le \Bigl(\sup_{t\in\mathbb{R}}|\psi_\varepsilon(1+it)|\Bigr)\Bigl(\sup_{t\in\mathbb{R}}\|F(1+it)\|_{A_1}\Bigr)
		\le e^{\varepsilon(1-\theta)^2}\,C. \tag{$B_1$}
	\end{equation}
	
	By \eqref{eq:B0-weighted}--\eqref{eq:B1-weighted},
	\[
	\|\widetilde F_\varepsilon\|_{\mathcal{F}(A_0,A_1)}
	=\max\Bigl\{\sup_{t\in\mathbb{R}}\|\widetilde F_\varepsilon(it)\|_{A_0},\ 
	\sup_{t\in\mathbb{R}}\|\widetilde F_\varepsilon(1+it)\|_{A_1}\Bigr\}
	\le e^{\varepsilon m}\,C,
	\quad m\coloneqq \max\{\theta^2,(1-\theta)^2\}.
	\]
	By the definition of the complex interpolation norm,
	\[
	\|f\|_{(A_0,A_1)_{[\theta]}}
	=\inf\bigl\{\|F\|_{\mathcal{F}(A_0,A_1)}:\ F\in\mathcal{F}(A_0,A_1),\ F(\theta)=f\bigr\}
	\le \|\widetilde F_\varepsilon\|_{\mathcal{F}(A_0,A_1)}
	\le e^{\varepsilon m}\,C.
	\]
	Letting $\varepsilon\to 0$ yields
	\[
	\|f\|_{(A_0,A_1)_{[\theta]}}
	\le C
	=\bigl\|\pi_\#\bigl(|f|\,w^{\beta_\theta}\bigr)\bigr\|_{\ell^q(X)},
	\]
	which proves \eqref{eq:interp-norm-est-weighted}.
\end{proof}


\begin{corollary}\label{cor:interp-norm-est-one-sided}
	Under the assumptions of Lemma~\ref{lem:interp-norm-est-weighted}, take $\beta_0=0$ and $\beta_1=1$.
	Then for every $f\in\mathbb{C}G$,
	\[
	\|f\|_{(A_0,A_1)_{[\theta]}}
	\le \bigl\|\pi_\#\bigl(|f|\,w^{\theta}\bigr)\bigr\|_{\ell^q(X)},
	\qquad
	\frac1q=1-\frac{\theta}{2}.
	\]
\end{corollary}


\begin{proposition}\label{prop:interp-quasi-regular-weighted-theta}
	Let $0<\theta<1$ and let $q$ be determined by
	\[
	\frac1q=\frac{1-\theta}{1}+\frac{\theta}{2}=1-\frac{\theta}{2}
	\qquad\Bigl(\text{equivalently } q=\frac{2}{2-\theta}\in(1,2)\Bigr).
	\]
	Let $w:G\to(0,\infty)$ be a strictly positive function and let $\beta_0,\beta_1\in\mathbb{R}$.
	Define norms on $\mathbb{C}G$ by
	\[
	\|f\|_{A_0}\coloneqq \bigl\|\pi_\#\bigl(|f|\,w^{\beta_0}\bigr)\bigr\|_{\ell^1(X)},
	\qquad
	\|f\|_{A_1}\coloneqq \bigl\|\pi_\#\bigl(|f|\,w^{\beta_1}\bigr)\bigr\|_{\ell^2(X)},
	\]
	and let $A_0,A_1$ be the corresponding completions.
   Set $\beta_\theta\coloneqq (1-\theta)\beta_0+\theta\beta_1$.
	Assume:
	\begin{enumerate}
		\item[(E$_0$)] $M_0\coloneqq \|w^{-\beta_0}\|_{\ell^\infty(G)}<\infty$
		(e.g.\ if $w\ge 1$ and $\beta_0\ge 0$, then $M_0\le 1$);
		\item[(E$_1$)] there exists $M_1>0$ such that for all $f\in\mathbb{C}G$ and all $\xi\in\ell^2(X)$,
		\[
		\|\lambda_X(f)\xi\|_2\le M_1\,\bigl\|\pi_\#\bigl(|f|\,w^{\beta_1}\bigr)\bigr\|_{\ell^2(X)}\,\|\xi\|_2
		= M_1\,\|f\|_{A_1}\,\|\xi\|_2.
		\]
	\end{enumerate}
	Then for every $f\in\mathbb{C}G$ one has
	\begin{equation}
	\|\lambda_{X,q}(f)\|_{B(\ell^q(X))}
	\le
	M_0^{\,1-\theta}\,M_1^{\,\theta}\,
	\bigl\|\pi_\#\bigl(|f|\,w^{\beta_\theta}\bigr)\bigr\|_{\ell^q(X)}.
	\end{equation}
\end{proposition}

\begin{proof}
	Using finite-support truncation, we identify the completions $A_0$ and $A_1$
	isometrically with the concrete subspaces of $\mathbb C^G$ given by
	\[
	A_j=\Bigl\{f:G\to\mathbb C:\ \bigl\|\pi_\#\bigl(|f|\,w^{\beta_j}\bigr)\bigr\|_{\ell^{p_j}(X)}<\infty\Bigr\},
	\qquad (p_0=1,\ p_1=2).
	\]
	Thus $A_0,A_1\subset \mathbb C^G$ form a compatible Banach couple, and $\mathbb CG$
	is dense in each.
	
	Set $X_0=Y_0=\ell^1(X)$ and $X_1=Y_1=\ell^2(X)$.
	For $f\in A_0\cap A_1$ and $\xi\in \ell^1(X)\cap \ell^2(X)$, 
	$B(f,\xi)= \sum_{g\in G} f(g)\lambda_X(g)\xi $ is well-defined. Indeed, by (E$_0$),
	\[
	\sum_{g\in G}|f(g)|
	\le \|w^{-\beta_0}\|_{\ell^\infty(G)} \sum_{g\in G}|f(g)|\,w(g)^{\beta_0}
	= M_0\,\|f\|_{A_0}<\infty,
	\]
	and  for $j=1,2$,
	\[
	\sum_{g\in G}\bigl\|f(g)\lambda_X(g)\xi\bigr\|_j
	\le \Bigl(\sum_{g\in G}|f(g)|\Bigr)\|\xi\|_j<\infty.
	\]
	Hence the defining series converges absolutely in both $\ell^1(X)$ and $\ell^2(X)$,
	so $B(f,\xi)\in \ell^1(X)\cap \ell^2(X)$.
	
	\medskip
	\noindent\textbf{Endpoint $E_0$ ($\ell^1$ bound).}
	For $f\in A_0\cap A_1$ and $\xi\in\ell^1(X)\cap\ell^2(X)$,
	since $G$ acts on $X$ by permutations, $\|\lambda_X(g)\xi\|_1=\|\xi\|_1$. Hence
	\[
	\|B(f,\xi)\|_1
	\le \sum_{g\in G}|f(g)|\,\|\lambda_X(g)\xi\|_1
	=\Bigl(\sum_{g\in G}|f(g)|\Bigr)\|\xi\|_1.
	\]
	Using $|f(g)|=|f(g)|w(g)^{\beta_0}\,w(g)^{-\beta_0}$ and
	$M_0=\|w^{-\beta_0}\|_\infty$, we get
	\[
	\sum_{g\in G}|f(g)|
	\le M_0\sum_{g\in G}|f(g)|w(g)^{\beta_0}
	= M_0\,\bigl\|\pi_\#\bigl(|f|\,w^{\beta_0}\bigr)\bigr\|_{\ell^1(X)}
	= M_0\,\|f\|_{A_0}.
	\]
	Thus
	\[
	\|B(f,\xi)\|_1\le M_0\,\|f\|_{A_0}\,\|\xi\|_1.
	\]
	We shall also use the analogous estimate in $\ell^2$:
	\begin{equation}\label{eq:aux-A0-to-l2}
		\|B(f,\xi)\|_2
		\le \sum_{g\in G}|f(g)|\,\|\lambda_X(g)\xi\|_2
		\le M_0\,\|f\|_{A_0}\,\|\xi\|_2,
		\qquad
		f\in A_0,\ \xi\in \ell^2(X).
	\end{equation}
	
	\medskip
	\noindent\textbf{Endpoint $E_1$ ($\ell^2$ bound).}
	Let $f\in A_0\cap A_1$ and $\xi\in\ell^1(X)\cap\ell^2(X)$.
	Choose $f_n\in \mathbb CG$ such that
	\[
	\|f_n-f\|_{A_0}+\|f_n-f\|_{A_1}\longrightarrow 0.
	\]
	For each $n$, assumption (E$_1$) gives
	\[
	\|B(f_n,\xi)\|_2
	=\|\lambda_X(f_n)\xi\|_2
	\le M_1\,\|f_n\|_{A_1}\,\|\xi\|_2.
	\]
	On the other hand, by \eqref{eq:aux-A0-to-l2},
	\[
	\|B(f_n,\xi)-B(f,\xi)\|_2
	=\|B(f_n-f,\xi)\|_2
	\le M_0\,\|f_n-f\|_{A_0}\,\|\xi\|_2
	\longrightarrow 0.
	\]
	Hence $B(f_n,\xi)\to B(f,\xi)$ in $\ell^2(X)$, and therefore
	\[
	\|B(f,\xi)\|_2
	\le \limsup_{n\to\infty}\|B(f_n,\xi)\|_2
	\le M_1\Bigl(\lim_{n\to\infty}\|f_n\|_{A_1}\Bigr)\|\xi\|_2
	= M_1\,\|f\|_{A_1}\,\|\xi\|_2.
	\]
	
	\medskip
	\noindent\textbf{Interpolation.}
	We have shown that the bilinear map
	\[
	B:(A_0\cap A_1)\times (X_0\cap X_1)\longrightarrow Y_0\cap Y_1
	\]
	satisfies
	\[
	\|B(f,\xi)\|_{Y_0}\le M_0\,\|f\|_{A_0}\,\|\xi\|_{X_0},
	\qquad
	\|B(f,\xi)\|_{Y_1}\le M_1\,\|f\|_{A_1}\,\|\xi\|_{X_1}.
	\]
	Hence, by Theorem~4.4.1 in \cite{bergh1976}, $B$ extends to a bounded bilinear map
	\[
	\widetilde B:(A_0,A_1)_{[\theta]}\times (X_0,X_1)_{[\theta]}
	\longrightarrow (Y_0,Y_1)_{[\theta]}
	\]
	with operator norm
	\[
	\|\widetilde B\|
	\le M_0^{1-\theta}M_1^\theta.
	\]
	Moreover, by Theorem~5.1.1 in \cite{bergh1976},
	\[
	(X_0,X_1)_{[\theta]}=(Y_0,Y_1)_{[\theta]}=\ell^q(X),
	\qquad
	\frac1q=1-\frac{\theta}{2}.
	\]
	Now let $f\in\mathbb CG$. Then $f\in A_0\cap A_1\subset (A_0,A_1)_{[\theta]}$, and for
	every $\xi\in \ell^1(X)\cap\ell^2(X)$,
	\[
	\widetilde B(f,\xi)=B(f,\xi)=\lambda_X(f)\xi.
	\]
	Since $f$ has finite support, $\lambda_{X,q}(f)$ is a bounded operator on $\ell^q(X)$. As $\ell^1(X)\cap\ell^2(X)$ is dense in
	$\ell^q(X)$, it follows that $\widetilde B(f,\xi)=\lambda_{X,q}(f)\xi$ for $\xi\in\ell^q(X)$.
	Therefore, for every $f\in\mathbb CG$ and every $\xi\in\ell^q(X)$,
	\[
	\|\lambda_{X,q}(f)\xi\|_q
	=\|\widetilde B(f,\xi)\|_q
	\le M_0^{1-\theta}M_1^\theta\,\|f\|_{(A_0,A_1)_{[\theta]}}\,\|\xi\|_q.
	\]
	Taking the supremum over $\|\xi\|_q=1$ gives
	\[
	\|\lambda_{X,q}(f)\|_{B(\ell^q(X))}
	\le M_0^{1-\theta}M_1^\theta\,\|f\|_{(A_0,A_1)_{[\theta]}}.
	\]
	Finally, Lemma~\ref{lem:interp-norm-est-weighted} yields
	\[
	\|f\|_{(A_0,A_1)_{[\theta]}}
	\le
	\bigl\|\pi_\#\bigl(|f|\,w^{\beta_\theta}\bigr)\bigr\|_{\ell^q(X)}.
	\]
	Combining the last two inequalities proves the claim.
\end{proof}

\begin{remark}\label{rem:interp-quasi-regular-weighted-p}
	Under the assumptions of Proposition~\ref{prop:interp-quasi-regular-weighted-theta}, let $2<p<\infty$ and set
	$q=\frac{p}{p-1}$. Applying Proposition~\ref{prop:interp-quasi-regular-weighted-theta} with $\theta=\frac{2}{p}$
	(one has $\theta\in(0,1)$) gives
	\[
	\|\lambda_{X,q}(f)\|_{B(\ell^q(X))}
	\le
	M_0^{\,1-\frac{2}{p}}\,M_1^{\,\frac{2}{p}}\,
	\bigl\|\pi_\#\bigl(|f|\,w^{\beta_p}\bigr)\bigr\|_{\ell^q(X)},
	\qquad
	\beta_p\coloneqq \Bigl(1-\frac{2}{p}\Bigr)\beta_0+\frac{2}{p}\beta_1.
	\]
	Indeed, $\frac1q=1-\frac{\theta}{2}=1-\frac1p$, hence $q=\frac{p}{p-1}$, and
	$\beta_\theta=(1-\theta)\beta_0+\theta\beta_1=\beta_p$.
	
	Moreover, the endpoint case $p=2$ (i.e.\ $q=2$) follows directly from the $\ell^2$ estimate \textup{(E$_1$)}.
\end{remark}

\medskip

We now introduce the rapid decay property for pairs.

\begin{definition}(\cite{ChatterjiZarka2024v1})
	We say that the pair $(G,H)$ has the \emph{rapid decay property}, or simply \emph{pair rapid decay}, if there exist constants $C_h>0$ and $s_1>0$ such that
	$$
	\|f\|_{h} \le C_h\, \|f(1+\ell)^{s_1}\|_{(2,1)}
	$$
	for all $f \in \mathbb{C}G$.
\end{definition}

\begin{corollary}\label{cor:hybrid-rd-implies-q-bound}
	Assume that the pair $(G,H)$ has \emph{pair rapid decay}. Let $0<\theta<1$, and let $q$ be given by $\frac1q = 1 - \frac{\theta}{2}$. Then, for every $f \in \mathbb{C}G$,
	$$
	\|\lambda_{X,q}(f)\|_{B(\ell^q(X))}
	\le
	C_h^{\,\theta}\,
	\bigl\|\pi_\#\bigl(|f|\,(1+\ell)^{s_1\theta}\bigr)\bigr\|_{\ell^q(X)}.
	$$
\end{corollary}


\begin{proof}
	Set $w(g)\coloneqq 1+\ell(g)$, $\beta_0\coloneqq 0$ and $\beta_1\coloneqq s_1$.
	By the identity $\|f(1+\ell)^{s_1}\|_{(2,1)}=\|\pi_\#(|f|w^{\beta_1})\|_2$
	(\cite[Remark~2.4]{ChatterjiZarka2024v1})
	and the estimate $\|\lambda_X(f)\|_{B(\ell^2(X))}\le \|f\|_h$ for the quasi-regular
	representation (\cite[Lemma~2.12(1)]{ChatterjiZarka2024v1}),
	the pair rapid decay assumption yields
	\[
	\|\lambda_X(f)\|_{B(\ell^2(X))}
	\le C_h\,\bigl\|\pi_\#\bigl(|f|\,w^{\beta_1}\bigr)\bigr\|_2.
	\]
	Thus Proposition~\ref{prop:interp-quasi-regular-weighted-theta} applies with
	$M_0=1$ and $M_1=C_h$. Since $\beta_\theta=(1-\theta)\beta_0+\theta\beta_1=\theta s_1$,
	we obtain
	\[
	\|\lambda_{X,q}(f)\|_{B(\ell^q(X))}
	\le C_h^{\theta}\,\bigl\|\pi_\#\bigl(|f|\,w^{\beta_\theta}\bigr)\bigr\|_q
	=
	C_h^{\theta}\,\bigl\|\pi_\#\bigl(|f|\,(1+\ell)^{s_1\theta}\bigr)\bigr\|_q,
	\]
	as claimed.
\end{proof}


\begin{remark}[Ball-supported estimate]
	For later applications it is enough to have the following ball-supported bound.
	Assume  $\supp(f)\subset B(R)$ and that $(G,H)$ has pair rapid decay. Let $0<\theta<1$ and define $q$ by $\frac1q=\frac{1-\theta}{1}+\frac{\theta}{2}=1-\frac{\theta}{2}$.
	One has
	\begin{equation}\label{eq:target}
		\|\lambda_{G/H}(f)\| \le C^{\theta}(R+1)^{D\theta}\,\|f\|_{(q,1)}\,
	\end{equation}
	where for $1\le p<\infty$ we write
	$\|f\|_{(p,1)}:=\big(\sum_{x\in G/H}(\sum_{g\in \pi^{-1}(x)}|f(g)|)^p\big)^{1/p}$.

	\smallskip
	Although \eqref{eq:target} follows from Corollary~\ref{cor:hybrid-rd-implies-q-bound}, it also admits
	a direct proof via standard interpolation.
	By Proposition~2.7(3) in \cite{ChatterjiZarka2024v1}, pair rapid decay is equivalent to the existence of constants
	$C,D\ge 0$ such that $\|u\|_h\le C(R+1)^D\|u\|_{(2,1)}$ whenever $\supp(u)\subset B(R)$.
	Together with $\|\lambda_{G/H}(u)\|_{2\to2}\le \|u\|_h$, this gives the $\ell^2$ endpoint estimate
	$\|\lambda_{G/H}(u)\|_{2\to2}\le C(R+1)^D\|u\|_{(2,1)}$ for all $u$ supported in $B(R)$.
	
	For convenience, fix $R\ge 0$ and set $P_R f:=1_{B(R)}f$, so that $\supp(P_R f)\subset B(R)$.
	Consider the bilinear map $B_R(f,\xi):=\lambda_{G/H}(P_R f)\xi$.
	The $\ell^1$ endpoint is immediate from the triangle inequality and the fact that each $\lambda_{G/H}(g)$ is an
	$\ell^1$-isometry: $\|B_R(f,\xi)\|_1\le \|f\|_1\|\xi\|_1$.
	The $\ell^2$ endpoint follows from the previous paragraph:
	$\|B_R(f,\xi)\|_2\le C(R+1)^D\|f\|_{(2,1)}\|\xi\|_2$.
	
	Thus $B_R$ is bounded both as $X_0\times Y_0\to Z_0$ and as $X_1\times Y_1\to Z_1$, where
	$X_0=\ell^1(G)=\ell^{(1,1)}(G,H)$, $X_1=\ell^{(2,1)}(G,H)$, and $Y_0=Z_0=\ell^1(G/H)$, $Y_1=Z_1=\ell^2(G/H)$.
	Applying the bilinear complex interpolation theorem \cite[Thm.~4.4.1]{bergh1976} yields
	\[
	\|B_R\|_{[X_0,X_1]_\theta\times [Y_0,Y_1]_\theta\to [Z_0,Z_1]_\theta}\le (C(R+1)^D)^\theta .
	\]
	
	To identify the interpolation spaces, we use
	$[Y_0,Y_1]_\theta=[Z_0,Z_1]_\theta=\ell^q(G/H)$ with $\frac1q=\frac{1-\theta}{1}+\frac{\theta}{2}$
	\cite[Thm.~5.1.1]{bergh1976}.
	Moreover, the mixed-norm spaces satisfy the isometric identification
	$\ell^{(p,1)}(G,H)\cong \ell^p(G/H;\ell^1(H))$.
	In particular, the vector-valued complex interpolation theorem \cite[Thm.~5.1.2]{bergh1976} gives
	$[X_0,X_1]_\theta\cong \ell^{(q,1)}(G,H)$.
	
	Combining these facts, we obtain
	$\|\lambda_{G/H}(P_R f)\xi\|_q\le C^\theta(R+1)^{D\theta}\|f\|_{(q,1)}\|\xi\|_q$.
	If $\supp(f)\subset B(R)$ then $P_R f=f$, and \eqref{eq:target} follows.
	\end{remark}

\subsection{Asymptotic Entropy on Homogeneous Spaces}

We now apply the interpolation formula established in Subsection~\ref{subsec:Interpolation-Estimate} to obtain an upper bound for the upper asymptotic entropy \(\overline{h}(G/H,\mu)\).

\begin{lemma}\label{lem:upper-asymptotic-entropy}
	Assume that $(G,H)$ has pair rapid decay and that $\mu$ has finite support. Then the upper asymptotic entropy satisfies
	$$
	\overline{h}(G/H,\mu) \le c(G,H;\mu).
	$$
\end{lemma}

\begin{proof}
	Fix \(p \ge 2\), and set \(q=\frac{p}{p-1}\in(1,2]\). For \(n\ge1\), let \(s=2s_1\), where \(s_1\) is the constant appearing in Corollary~\ref{cor:hybrid-rd-implies-q-bound}, and define
	$$
	a_{n,p}(x) \coloneqq \pi_{\#}\bigl(\mu^{*n}(1+\ell)^{s/p}\bigr)(x)
	= \sum_{gH=x} \mu^{*n}(g)(1+\ell(g))^{s/p}, \qquad x \in X.
	$$
	By the interpolation estimate established in Corollary~\ref{cor:hybrid-rd-implies-q-bound}, we have
	$$
	\|\lambda_{X,q}(\mu^{*n})\|_{B(\ell^q(X))}
	\le
	C^{1/p}\,\|a_{n,p}\|_{\ell^q(X)}.
	$$
	Since $\lambda_{X,q}(\mu^{*n})=\lambda_{X,q}(\mu)^n$ and $r_q(\mu)^n \le \|\lambda_{X,q}(\mu)^n\|_{B(\ell^q(X))}$, it follows that $
	r_q(\mu)^n \le C^{1/p}\,\|a_{n,p}\|_{\ell^q(X)}.$
	Hence
	$$
	-p \log r_q(\mu)
	\ge
	-\frac{p}{n}\log \|a_{n,p}\|_{\ell^q(X)} - \frac{\log C}{n}.
	$$
	
	Let $\nu_n = \mu^{*n} * \delta_o \in \mathcal{P}(X)$. For $x \in X$ with $\nu_n(x)>0$, define $
	\omega_n(x) \coloneqq \mathbb{E}\bigl[(1+\ell(Z_n))^s \mid X_n=x\bigr].$
	For notational convenience, set $\omega_n(x)=1$ when $\nu_n(x)=0$. Then for every \(x\in X\) with \(\nu_n(x)>0\), we have $
	a_{n,p}(x)
	=
	\nu_n(x)\,\mathbb{E}\bigl[(1+\ell(Z_n))^{s/p}\mid X_n=x\bigr].$
	By Jensen's inequality, $
	\mathbb{E}\bigl[(1+\ell(Z_n))^{s/p}\mid X_n=x\bigr]
	\le
	\omega_n(x)^{1/p}$ whenever $\nu_n(x)>0.$
	Therefore, $
	a_{n,p}(x)\le \nu_n(x)\,\omega_n(x)^{1/p}$ for every $x\in X.$
	and hence $
	\|a_{n,p}\|_{\ell^q(X)}
	\le
	\|\nu_n\,\omega_n^{1/p}\|_{\ell^q(X)}.$
	Consequently,
	\begin{equation}\label{eq:logrq}
	-p \log r_q(\mu)
    \ge
    -\frac{p}{n}\log \|\nu_n\,\omega_n^{1/p}\|_{\ell^q(X)} - \frac{\log C}{n}.
	\end{equation}

	Now $\nu_n$ has finite support. Thus for fixed $n$, we may compute the limit directly. Indeed,
	$$
	\|\nu_n\,\omega_n^{1/p}\|_{\ell^q(X)}^q
	=
	\sum_{x \in X} \nu_n(x)^q \omega_n(x)^{q/p}.
	$$
	Using $q=\frac{p}{p-1}$, we obtain $
	\nu_n(x)^q \omega_n(x)^{q/p}
	=
	\nu_n(x)\bigl(\nu_n(x)\omega_n(x)\bigr)^{1/(p-1)},$
	so that
	$$
	\|\nu_n\,\omega_n^{1/p}\|_{\ell^q(X)}^q
	=
	\sum_{x \in \supp(\nu_n)}
	\nu_n(x)\bigl(\nu_n(x)\omega_n(x)\bigr)^{1/(p-1)}.
	$$
	Set $\varepsilon = \frac{1}{p-1}$. Since $\frac{p}{q}=p-1=\varepsilon^{-1}$, we have
	$$
	-p \log \|\nu_n\,\omega_n^{1/p}\|_{\ell^q(X)}
	=
	-\frac{1}{\varepsilon}
	\log\left(
	\sum_{x \in \supp(\nu_n)}
	\nu_n(x)e^{\varepsilon \log(\nu_n(x)\omega_n(x))}
	\right).
	$$
	Define
	$$
	F_n(\varepsilon)
	\coloneqq
	\log\left(
	\sum_{x \in \supp(\nu_n)}
	\nu_n(x)e^{\varepsilon \log(\nu_n(x)\omega_n(x))}
	\right).
	$$
	Since $\supp(\nu_n)$ is finite, $F_n$ is differentiable at $\varepsilon=0$, with $F_n(0)=0$ and
	$$
	F_n'(0)=\sum_{x \in \supp(\nu_n)} \nu_n(x)\log(\nu_n(x)\omega_n(x)).
	$$
	Therefore
	\begin{equation}\label{eq:lim_p}
	\lim_{p \to \infty}
	\bigl(-p \log \|\nu_n\,\omega_n^{1/p}\|_{\ell^q(X)}\bigr)
	=
	H(\nu_n)-\sum_{x \in X}\nu_n(x)\log \omega_n(x).
	\end{equation}

	Let $R \coloneqq \max\{\ell(g):\mu(g)>0\}<\infty.$
	Then $\ell(Z_n)\le nR$ almost surely, and therefore $1 \le \omega_n(x)\le (1+nR)^s$ whenever $\nu_n(x)>0$. It follows that
	\begin{equation}\label{eq:log(1+nR)}
    0 \le \sum_{x \in X}\nu_n(x)\log \omega_n(x)
    \le s\log(1+nR).
	\end{equation}

	Taking the limit as $p \to \infty$ in \eqref{eq:logrq}, and combining \eqref{eq:lim_p} with \eqref{eq:log(1+nR)} we obtain
	$$
	c(G,H;\mu)
	\ge
	\frac{H(\nu_n)}{n}
	-
	\frac{s\log(1+nR)}{n}
	-
	\frac{\log C}{n}
	$$
	for every $n \ge 1$. Since $\frac{s\log(1+nR)}{n} \to 0$ and $\frac{\log C}{n} \to 0$, we conclude that
	$$
	c(G,H;\mu)
	\ge
	\limsup_{n \to \infty}\frac{H(\nu_n)}{n}
	=
	\overline{h}(G/H,\mu).
	$$
\end{proof}

\begin{remark}
	The preceding proof is closely related to the argument in \cite[Theorem 4.4]{anderson2024}. Although that result is stated for probability measures on groups endowed with a weight, its proof only uses weighted $\ell^q$-norms and entropy identities. For this reason, the same argument adapts to the present finite-support probability measure $\nu_n$ on the countable set $X=G/H$.
	
	Indeed, for fixed $n \ge 1$, define
	$$
	\omega_n(x) \coloneqq \mathbb{E}\bigl[(1+\ell(Z_n))^s \mid X_n=x\bigr],
	$$
	with $\omega_n(x)=1$ when $\nu_n(x)=0$. Then $a_{n,p}(x) \le \nu_n(x)\,\omega_n(x)^{1/p}$, and hence
	$\|a_{n,p}\|_{\ell^{q(p)}(X)} \le \|\nu_n\,\omega_n^{1/p}\|_{\ell^{q(p)}(X)}$.
	Arguing exactly as in the proof of \cite[Theorem 4.4]{anderson2024}, one obtains
	$$
	\lim_{p\to\infty}\bigl(-p\log \|\nu_n\,\omega_n^{1/p}\|_{\ell^{q(p)}(X)}\bigr)
	=
	H(\nu_n)-\sum_{x\in X}\nu_n(x)\log \omega_n(x).
	$$
	Since $\omega_n(x)\le (1+nR)^s$, this again yields
	$c(G,H;\mu)\ge \frac{H(\nu_n)}{n} - \frac{s\log(1+nR)}{n} - \frac{\log C}{n}$,
	and therefore $\overline{h}(G/H,\mu)\le c(G,H;\mu)$.
\end{remark}

\begin{proposition}
	Assume that $(G,H)$ has pair rapid decay and that $\mu$ has finite support. Then the lower and upper asymptotic entropies coincide, and their common value is $c(G,H;\mu)$. That is,
	$$
	\underline{h}(G/H,\mu)=\overline{h}(G/H,\mu)=c(G,H;\mu).
	$$
\end{proposition}

\begin{proof}
	By Lemma~\ref{lem:entropy_inequality} and Lemma~\ref{lem:upper-asymptotic-entropy}, we have
	$$
	\overline{h}(G/H,\mu) \le c(G,H;\mu) \le \underline{h}(G/H,\mu).
	$$
	On the other hand, by definition, $\underline{h}(G/H,\mu)\le \overline{h}(G/H,\mu).$
	Hence all three quantities coincide.
\end{proof}

\begin{corollary}
	Assume that there exist constants $C>0$ and $s>0$ such that for every $p \ge 2$, with $q=\frac{p}{p-1}$, one has
	$$
	\|\lambda_{X,q}(f)\|_{B(\ell^q(X))}
	\le
	C^{1/p}\,
	\bigl\|\pi_\#\bigl(|f|(1+\ell)^{s/p}\bigr)\bigr\|_{\ell^q(X)}
	\qquad \text{for all } f \in \mathbb{C}G.
	$$
	Let $\mu$ be a probability measure on $G$ with finite entropy and finite $\alpha$-moment with respect to $1+\ell$ for some $\alpha \ge s$, that is,
    $H(\mu)<\infty$ and $\sum_{g\in G}\mu(g)(1+\ell(g))^\alpha<\infty.$
	Then
	$$
	\underline{h}(G/H,\mu)=\overline{h}(G/H,\mu)=c(G,H;\mu).
	$$
\end{corollary}

\begin{proof}
	The proof is the same as in the finite-support case, except for two points: one must extend the operator estimate from $\mathbb{C}G$ to $\mu^{*n}$, and one must replace the bound $\omega_n(x)\le (1+nR)^s$ by a moment estimate.
	
	First, the proof of $c(G,H;\mu)\le \underline{h}(G/H,\mu)$ is unchanged. Indeed, the argument used earlier only requires that $H(\nu_n)<\infty$ for all $n$.
	
	We therefore only prove $\overline{h}(G/H,\mu)\le c(G,H;\mu)$. Fix $n\ge1$ and $p\ge2$, and set $q=q(p)\coloneqq \frac{p}{p-1}$. Define
	$$
	a_{n,p}(x)\coloneqq \pi_\#\bigl(\mu^{*n}(1+\ell)^{s/p}\bigr)(x),
	\qquad x\in X.
	$$
	Since $\alpha\ge s$, we have $\sum_{g\in G}\mu(g)(1+\ell(g))^s<\infty$, and hence $a_{n,p}\in \ell^q(X)$.
	
	To apply the assumed estimate to $\mu^{*n}$, let
	$$
	f_m\coloneqq \mu^{*n}\mathbf 1_{\{\ell\le m\}}\in \mathbb CG.
	$$
	Then
	$$
	\|\lambda_{X,q}(f_m)\|_{B(\ell^q(X))}
	\le
	C^{1/p}\,
	\bigl\|\pi_\#\bigl(f_m(1+\ell)^{s/p}\bigr)\bigr\|_{\ell^q(X)}.
	$$
	Since $\|\lambda_{X,q}(\mu^{*n})-\lambda_{X,q}(f_m)\| \le \|\mu^{*n}-f_m\|_1 \to 0$ and $\bigl\|\pi_\#\bigl(f_m(1+\ell)^{s/p}\bigr)\bigr\|_{\ell^q(X)} \to \|a_{n,p}\|_{\ell^q(X)}$, we obtain
	$$
	\|\lambda_{X,q}(\mu^{*n})\|_{B(\ell^q(X))}
	\le
	C^{1/p}\|a_{n,p}\|_{\ell^q(X)}.
	$$
	Thus
	$$
	-p\log r_q(\mu)\ge -\frac{p}{n}\log\|a_{n,p}\|_{\ell^q(X)}-\frac{\log C}{n}.
	$$
	
	Now define
	$$
	\omega_n(x)\coloneqq \mathbb E\bigl[(1+\ell(Z_n))^s\mid X_n=x\bigr]
	$$
	for $\nu_n(x)>0$, and set $\omega_n(x)=1$ otherwise. Exactly as in the finite-support case, conditional Jensen gives $a_{n,p}(x)\le \nu_n(x)\omega_n(x)^{1/p}$, hence
	$$
	\|a_{n,p}\|_{\ell^q(X)}\le \|\nu_n\omega_n^{1/p}\|_{\ell^q(X)}.
	$$
	By the countable-set version of \cite[Theorem 4.4]{anderson2024}, it is therefore enough to control $\sum_{x\in X}\nu_n(x)\log\omega_n(x)$.
	
	Write $Y_i\coloneqq 1+\ell(s_i)$. Since $1+\ell(Z_n)\le Y_1+\cdots+Y_n$ and $\mu$ has finite $s$-moment, there exist constants $A>0$ and $\kappa=\max\{1,s\}$ such that
	$$
	\mathbb E\bigl[(1+\ell(Z_n))^s\bigr]\le A n^\kappa
	\qquad\text{for all }n\ge1.
	$$
	Therefore
	$$
	\sum_{x\in X}\nu_n(x)\log\omega_n(x)
	=
	\mathbb E\bigl[\log\omega_n(X_n)\bigr]
	\le
	\log\mathbb E\bigl[(1+\ell(Z_n))^s\bigr]
	\le
	\log A+\kappa\log n.
	$$
	Repeating the weighted-entropy argument from the finite-support case, we obtain
	$$
	c(G,H;\mu)
	\ge
	\frac{H(\nu_n)}{n}
	-
	\frac{\log A+\kappa\log n+\log C}{n}.
	$$
	Taking the limit superior in $n$ yields $c(G,H;\mu)\ge \overline{h}(G/H,\mu)$.
	
\end{proof}

By Corollary~\ref{cor:hybrid-rd-implies-q-bound}, taking $\theta=\frac{2}{p}$, the hypotheses of the previous corollary are satisfied with $C=C_h^2$ and $s=2s_1$. We therefore obtain the following consequence. In the special case $H=\{e\}$, it recovers the group case of \cite[Theorem 1.3]{anderson2024}; more generally, it applies to homogeneous spaces $G/H$.

\begin{corollary}\label{cor:under-over}
	Assume that $(G,H)$ has pair rapid decay and let $\mu$ be a probability measure on $G$ with finite entropy and finite $\beta$-moment with respect to $1+\ell$ for some $\beta\ge 2s_1$. Then
	$$
	\underline{h}(G/H,\mu)=\overline{h}(G/H,\mu)=c(G,H;\mu).
	$$
\end{corollary}

	\subsection{Asymptotic R\'enyi Entropy Rates}
	
	Our goal in this subsection is to study asymptotic R\'enyi entropy rates on homogeneous spaces. We first consider the limit as \(\alpha\downarrow1\). We then study existence and continuity for finitely supported measures.
	
	For $\alpha > 1$, we define the lower and upper Rényi entropy rates as follows:
	$$
	\underline{h}_\alpha \coloneqq \liminf_{n\to\infty}\frac1n H_\alpha(\nu_n),
	\qquad
	\overline{h}_\alpha \coloneqq \limsup_{n\to\infty}\frac1n H_\alpha(\nu_n).
	$$



       \begin{theorem}\label{prop:renyi-to-shannon}
       	Assume that $(G,H)$ has pair rapid decay, and let $\mu$ be a probability measure on $G$ with finite entropy and finite $\beta$-moment with respect to $1+\ell$ for some $\beta \ge 2s_1$. Then
       	$$
       	\lim_{\alpha\downarrow1}\underline{h}_\alpha
       	=
       	\lim_{\alpha\downarrow1}\overline{h}_\alpha
       	=
       	c(G,H;\mu).
       	$$
       \end{theorem}

       \begin{proof}
       	By Corollary~\ref{cor:under-over}, we have
       	$$
       	\underline{h}(G/H,\mu)=\overline{h}(G/H,\mu)=c(G,H;\mu).
       	$$
       	In particular, the Shannon entropy rate exists, and we may denote its common value by
       	$h(X,\mu)=\lim_{n\to\infty}\frac1n H(\nu_n)=c(G,H;\mu)$.
       	
       	Let $\alpha>1$. Since Rényi entropy is monotone in the order parameter, we have $H_\alpha(\nu_n)\le H(\nu_n)$ for every $n\ge1$, and therefore $\overline{h}_\alpha\le h$.
       	
       	On the other hand, let
       	$$
       	r_\alpha \coloneqq \lim_{n\to\infty}
       	\|\lambda_{X,\alpha}(\mu)^n\|_{B(\ell^\alpha(X))}^{1/n}.
       	$$
       	Since $\nu_n=\lambda_{X,\alpha}(\mu)^n\delta_o$ and $\|\delta_o\|_\alpha=1$, we have $\|\nu_n\|_\alpha \le \|\lambda_{X,\alpha}(\mu)^n\|_{B(\ell^\alpha(X))}$, hence
       	$\limsup_{n\to\infty}\frac1n\log\|\nu_n\|_\alpha \le \log r_\alpha$.
       	Using $H_\alpha(\nu_n)=\frac{\alpha}{1-\alpha}\log\|\nu_n\|_\alpha$ and the fact that $\frac{\alpha}{1-\alpha}<0$, it follows that
       	$\underline{h}_\alpha \ge \frac{\alpha}{1-\alpha}\log r_\alpha$.
       	
       	Now let $q=\alpha$ and $p=\frac{\alpha}{\alpha-1}$. Then $q=\frac{p}{p-1}$ and $-p\log r_q=\frac{\alpha}{1-\alpha}\log r_\alpha$. By the definition of $c(G,H;\mu)$,
       	$$
       	c(G,H;\mu)
       	=
       	\lim_{\alpha\downarrow1}\frac{\alpha}{1-\alpha}\log r_\alpha.
       	$$
       	Consequently, $\liminf_{\alpha\downarrow1}\underline{h}_\alpha \ge c(G,H;\mu)$.
       	
       	Since $\underline{h}_\alpha\le \overline{h}_\alpha$ for every $\alpha>1$, we obtain
       	$$
       	c(G,H;\mu)
       	\le
       	\liminf_{\alpha\downarrow1}\underline{h}_\alpha
       	\le
       	\limsup_{\alpha\downarrow1}\underline{h}_\alpha
       	\le
       	\limsup_{\alpha\downarrow1}\overline{h}_\alpha
       	\le
       	h(X,\mu)
       	=
       	c(G,H;\mu).
       	$$
       	Thus $\lim_{\alpha\downarrow1}\underline{h}_\alpha=c(G,H;\mu)$. Moreover,
       	$\liminf_{\alpha\downarrow1}\overline{h}_\alpha
       	\ge
       	\liminf_{\alpha\downarrow1}\underline{h}_\alpha
       	=
       	c(G,H;\mu)$,
       	while $\limsup_{\alpha\downarrow1}\overline{h}_\alpha\le h(X,\mu)=c(G,H;\mu)$. Therefore
       	$$h(X,\mu)=
       	\lim_{\alpha\downarrow1}\underline{h}_\alpha
       	=
       	\lim_{\alpha\downarrow1}\overline{h}_\alpha
       	=
       	c(G,H;\mu).
       	$$
       \end{proof}

\begin{theorem}\label{prop:renyi-rate-pair-rd}
	Assume that the pair $(G,H)$ satisfies pair rapid decay. Let $\mu$ be a finitely supported probability measure on $G$, with $\supp(\mu)\subseteq B(e,R)$. Let $X:=G/H$ and $o:=eH$. For $n\ge1$, set
	\[
	\nu_n:=\mu^{*n} * \delta_o=\pi_\#(\mu^{*n})\in \ell^1(X),
	\qquad
	H_\alpha(\nu_n):=\frac{1}{1-\alpha}\log\Bigl(\sum_{x\in X}\nu_n(x)^\alpha\Bigr).
	\]
	Then for every $\alpha\in(1,2]$ the limit $h_\alpha(X,\mu):=\lim_{n\to\infty}\frac1n H_\alpha(\nu_n)$ exists and
	\[
	h_\alpha(X,\mu)=\frac{\alpha}{1-\alpha}\log r_\alpha(\mu),
	\qquad
	r_\alpha(\mu):=\lim_{n\to\infty}\bigl\|\lambda_{X,\alpha}(\mu)^n\bigr\|_{\alpha\to\alpha}^{1/n}.
	\]
\end{theorem}

\begin{proof}
	Fix $q\in(1,2]$ and choose $\theta\in(0,1]$ such that $\frac1q=1-\frac{\theta}{2}$. Then $\nu_n=\lambda_{X,q}(\mu^{*n})\delta_o\in\ell^q(X).$
	Since $\supp(\mu^{*n})\subseteq B(e,nR)$ and $(G,H)$ satisfies Pair RD, we may apply \eqref{eq:target} to obtain
	$$	\|\lambda_{X,q}(\mu^{*n})\|_{q\to q}\le C^{\theta}(1+nR)^{D\theta}\|\mu^{*n}\|_{(q,1)}.$$

	Since $\mu^{*n}\ge0$, the definition of the $(q,1)$-norm gives
	$\|\mu^{*n}\|_{(q,1)}=\|\pi_\#(\mu^{*n})\|_q=\|\nu_n\|_q.$
	Hence
	\begin{equation}\label{eq:sandwich-compact}
		\|\nu_n\|_q\le \|\lambda_{X,q}(\mu^{*n})\|_{q\to q}\le C^{\theta}(1+nR)^{D\theta}\|\nu_n\|_q,
	\end{equation}
	where the left inequality follows from
	\[
	\|\nu_n\|_q=\|\lambda_{X,q}(\mu^{*n})\delta_o\|_q\le \|\lambda_{X,q}(\mu^{*n})\|_{q\to q}\,\|\delta_o\|_q
	\]
	and $\|\delta_o\|_q=1$.
	Taking logarithms in \eqref{eq:sandwich-compact} and dividing by $n$ (note that $\|\nu_n\|_q>0$), we get
	\[
	\frac1n\log\|\nu_n\|_q
	\le
	\frac1n\log\|\lambda_{X,q}(\mu^{*n})\|_{q\to q}
	\le
	\frac1n\log\|\nu_n\|_q+\frac1n\log\!\bigl(C^{\theta}(1+nR)^{D\theta}\bigr).
	\]
	Since $\frac1n\log\!\bigl(C^{\theta}(1+nR)^{D\theta}\bigr)\to 0$, it follows that
	\[
	\lim_{n\to\infty}\frac1n\log\|\nu_n\|_q
	=
	\lim_{n\to\infty}\frac1n\log\|\lambda_{X,q}(\mu^{*n})\|_{q\to q}.
	\]
	
	Now $\lambda_{X,q}(\mu^{*(m+n)})=\lambda_{X,q}(\mu^{*m})\lambda_{X,q}(\mu^{*n})$ implies that $\|\lambda_{X,q}(\mu^{*n})\|_{q\to q}$ is submultiplicative. Since $\lambda_{X,q}(\mu^{*n})=\lambda_{X,q}(\mu)^n$, the limit defining $r_q(\mu)$ exists, and
	\[
	\lim_{n\to\infty}\frac1n\log\|\lambda_{X,q}(\mu^{*n})\|_{q\to q}
	=
	\log r_q(\mu).
	\]
	By Proposition~\ref{prop:rq-positive} one has $r_q(\mu)>0$, so the logarithm is well defined. Therefore
	\[
	\lim_{n\to\infty}\frac1n\log\|\nu_n\|_q=\log r_q(\mu).
	\]
	
	Finally, setting $q=\alpha\in(1,2]$ and using $H_\alpha(\nu_n)=\frac{\alpha}{1-\alpha}\log\|\nu_n\|_\alpha,$
	we obtain
	\[
	h_\alpha(X,\mu)
	=
	\lim_{n\to\infty}\frac1n H_\alpha(\nu_n)
	=
	\frac{\alpha}{1-\alpha}\lim_{n\to\infty}\frac1n\log\|\nu_n\|_\alpha
	=
	\frac{\alpha}{1-\alpha}\log r_\alpha(\mu).
	\]
\end{proof}

\begin{definition}[subexponential rapid decay for a pair]\label{def:subexp-RD-pair}
	Let \(G\) be a finitely generated group, \(H<G\), and \(\ell\) a word length on \(G\) associated with a finite symmetric generating set.
	We say that \((G,H)\) satisfies \emph{subexp--Lorentz control} or has the \emph{subexponential rapid decay property}, denoted $(G,H)\in \mathbf{SLC}_{\mathrm{subexp}}$,
	if there exist \(C_h>0\) and a weight \(w:G\to[1,\infty)\) such that
	\[
	\|f\|_{h}\le C_h\,\|fw\|_{(2,1)} \qquad \forall f\in\mathbb{C}G,
	\]
	and the radial majorant \(W(t):=\sup\{w(g):\ell(g)\le t\}\) for \(t\ge 0\) 
	has subexponential growth, i.e.
	\[
	\lim_{t\to\infty}\frac{\log W(t)}{t}=0.
	\]
\end{definition}

\bigskip

The preceding proposition yields the following consequence for pairs in $\mathbf{SLC}_{\mathrm{subexp}}$.

\begin{corollary}\label{cor:subexp-weighted-rd-renyi}
	Assume that $(G,H)\in \mathbf{SLC}_{\mathrm{subexp}}$. Let $\mu$ be a finitely supported probability measure on $G$ with $\supp(\mu)\subseteq B(e,R)$, and set
	$\nu_n\coloneqq \mu^{*n}*\delta_o=\pi_\#(\mu^{*n}).$
	Then for every $\alpha\in(1,2]$ the limit
	$$
	h_\alpha(X,\mu)\coloneqq \lim_{n\to\infty}\frac1n H_\alpha(\nu_n)
	$$
	exists and satisfies
	$$
	h_\alpha(X,\mu)=\frac{\alpha}{1-\alpha}\log r_\alpha(\mu),
	$$
	where $	r_\alpha(\mu)\coloneqq \lim_{n\to\infty}\|\lambda_{X,\alpha}(\mu)^n\|_{\alpha\to\alpha}^{1/n}.$
\end{corollary}

\begin{corollary}
	Assume that $(G,H)\in \mathbf{SLC}_{\mathrm{subexp}}$, and let $\mu$ be a finitely supported probability measure on $G$. Then
	$$
	\underline{h}(G/H,\mu)=\overline{h}(G/H,\mu)=c(G,H;\mu).
	$$
\end{corollary}

\medskip

Combining Theorems~\ref{prop:renyi-to-shannon} and~\ref{prop:renyi-rate-pair-rd}, we immediately obtain the following consequence.

\begin{corollary}
	Assume that the pair $(G,H)$ satisfies pair rapid decay. Let $\mu$ be a finitely supported probability measure on $G$, and let $h(X,\mu) \coloneqq \lim_{n\to\infty}\frac1n H(\nu_n).$
	Then
	\[
	\lim_{\alpha \downarrow 1} h_{\alpha}(X,\mu)=h(X,\mu).
	\]
\end{corollary}

\begin{remark}
	The paper \emph{Asymptotic R\'enyi entropies of random walks on groups} \cite{GolubevaPanTamuz2024} raises the problem of determining when the map \(\alpha\mapsto h_\alpha(\mu)\) is continuous at \(\alpha=1\) for random walks on non-amenable groups. In the case \(H=\{e\}\), the preceding corollary shows that property RD together with finite support of \(\mu\) provides a sufficient condition for this continuity. Hence, in the group case, our result gives a partial affirmative answer to that problem; moreover, the present formulation extends to homogeneous spaces \(G/H\).
\end{remark}

\begin{proposition}
	\label{prop:continuity-renyi-rate}
	Assume the hypotheses of Theorem~\ref{prop:renyi-rate-pair-rd}. Then the function
	$$
	\alpha \mapsto h_\alpha(X,\mu), \qquad \alpha\in(1,2],
	$$
	is continuous. In addition, the map $\alpha \mapsto h_\alpha(X,\mu)$ extends continuously to $[1,2]$ by setting $h_1(X,\mu)\coloneqq h(X,\mu)$.
\end{proposition}

\begin{proof}
	Let $X=G/H$. For each $p\in[1,\infty]$, set
	$a_n(p)\coloneqq \|\lambda_{X,p}(\mu)^n\|_{\ell^p(X)\to \ell^p(X)}$.
	Then $(a_n(p))_{n\ge1}$ is submultiplicative, so
	$r_p\coloneqq \lim_{n\to\infty} a_n(p)^{1/n}$
	exists for every $p\in[1,\infty]$.
	For $\alpha\in(1,2]$, this is exactly the quantity denoted by $r_\alpha(\mu)$ in Theorem~\ref{prop:renyi-rate-pair-rd}. By Proposition~\ref{prop:rq-positive}, one has $r_p>0$ for every $p\in(1,2]$, so the function $p\mapsto \log r_p$ is well defined on $(1,2]$.
	
	We claim that $p\mapsto \log r_p$ is continuous on $(1,2]$. Fix $1<p_0,p_1\le 2$, let $0<\theta<1$, and define $p$ by
	$\frac1p=\frac{1-\theta}{p_0}+\frac{\theta}{p_1}$.
	For each $n\ge1$, define
	\[
	(T_n f)(x)\coloneqq \sum_{g\in G}\mu^{*n}(g)f(g^{-1}x), \qquad x\in X.
	\]
	Then $T_n$ acts boundedly on each $\ell^p(X)$, and its realization on $\ell^p(X)$ is exactly $\lambda_{X,p}(\mu)^n$. Applying the Riesz--Thorin interpolation theorem to $T_n$ yields
	\[
	\|\lambda_{X,p}(\mu)^n\|_{\ell^p\to\ell^p}
	\le
	\|\lambda_{X,p_0}(\mu)^n\|_{\ell^{p_0}\to\ell^{p_0}}^{\,1-\theta}
	\|\lambda_{X,p_1}(\mu)^n\|_{\ell^{p_1}\to\ell^{p_1}}^{\,\theta},
	\]
	that is, $a_n(p)\le a_n(p_0)^{1-\theta}a_n(p_1)^\theta$. Taking $n$th roots and letting $n\to\infty$, we obtain
	$r_p\le r_{p_0}^{\,1-\theta}r_{p_1}^{\,\theta}$, equivalently,
	$\log r_p\le (1-\theta)\log r_{p_0}+\theta\log r_{p_1}$.
	Thus the function $u\mapsto \log r_{1/u}$ is convex on $(1/2,1)$, and therefore continuous there. Equivalently, $p\mapsto \log r_p$ is continuous on $(1,2]$.
	
	Now Theorem~\ref{prop:renyi-rate-pair-rd} gives, for every $\alpha\in(1,2]$,
	$h_\alpha(X,\mu)=\frac{\alpha}{1-\alpha}\log r_\alpha(\mu)$.
	Since $\alpha\mapsto \alpha/(1-\alpha)$ is continuous on $(1,2]$, and since $\alpha\mapsto \log r_\alpha(\mu)$ is continuous on the same interval, it follows that $\alpha\mapsto h_\alpha(X,\mu)$ is continuous on $(1,2]$.
	
	To prove continuity at $\alpha=1$, write $p=\frac{\alpha}{\alpha-1}$. Then $\alpha=\frac{p}{p-1}$, and therefore
	\[
	\frac{\alpha}{1-\alpha}\log r_\alpha(\mu)
	=
	-p\log r_{p/(p-1)}(\mu).
	\]
	By Theorem~\ref{prop:renyi-rate-pair-rd}, we have
	$h_\alpha(X,\mu)=\frac{\alpha}{1-\alpha}\log r_\alpha(\mu)$,
	so by the definition of $c(G,H;\mu)$ we obtain
	$\lim_{\alpha\downarrow1} h_\alpha(X,\mu)=c(G,H;\mu)$.
	On the other hand, Corollary~\ref{cor:under-over} gives
	$h(X,\mu)=c(G,H;\mu)$.
	Hence
	$\lim_{\alpha\downarrow1} h_\alpha(X,\mu)=h(X,\mu)$.
	This proves that $\alpha\mapsto h_\alpha(X,\mu)$ extends continuously to $[1,2]$ by setting $h_1(X,\mu)\coloneqq h(X,\mu)$.
\end{proof}

\bigskip

\section{Subgroup Classification for \(\mathbf{SLC}_{\mathrm{subexp}}\)
}

\subsection{Subgroup Classification in Hyperbolic Settings}\label{subsec:hyperbolic}

The goal of this subsection is to study the subgroup-classification problem for \((G,H)\in \mathbf{SLC}_{\mathrm{subexp}}\) in hyperbolic settings. We begin with classes for which explicit classifications are available, including free groups, closed surface groups, hyperbolic \(3\)-manifold groups, finitely generated Kleinian groups, and several other classes of hyperbolic groups. We then turn to relatively hyperbolic groups and derive strong necessary conditions.

\begin{proposition}
	Let \(G\) be a finitely generated group with word length \(\ell\), and let
	\(B(n):=\{g\in G:\ell(g)\le n\}\), \(b_n:=|B(n)|\).
	Define \(w:G\to[1,\infty)\) by \(w(g):=b_{\ell(g)}\).
	Then for every subgroup \(H\le G\) and every \(f\in\mathbb CG\),
	\[
	\|f\|_{h,(G,H)}\le \sqrt2\,\|fw\|_{(2,1),(G,H)}.
	\]
	Moreover, \(W(R):=\sup_{\ell(g)\le R}w(g)=b_{\lfloor R\rfloor}\). In particular, if \(G\) has subexponential growth, then \(W\) is subexponential.
\end{proposition}

\begin{proof}
	Let \(S(n):=\{g\in G:\ell(g)=n\}\), \(a_n:=|S(n)|\). Then
	\[
	\|w^{-1}\|_2^2=\sum_{g\in G}\frac1{w(g)^2}=\sum_{n\ge0}\frac{a_n}{b_n^2}.
	\]
	For \(n\ge1\),
	\[
	\frac{a_n}{b_n^2}=\frac{b_n-b_{n-1}}{b_n^2}
	\le \frac{b_n-b_{n-1}}{b_{n-1}b_n}
	= \frac1{b_{n-1}}-\frac1{b_n},
	\]
	hence \(\sum_{n\ge1} a_n/b_n^2\le 1\), and therefore \(\|w^{-1}\|_2^2\le 2\).
	
	Now for \(f\in\mathbb CG\), Cauchy--Schwarz gives
	\[
	\|f\|_1=\sum_{g}|f(g)|w(g)\,w(g)^{-1}
	\le \|fw\|_2\,\|w^{-1}\|_2
	\le \sqrt2\,\|fw\|_2.
	\]
	Also \(\|fw\|_2\le \|fw\|_{(2,1),(G,H)}\), since on each coset \(C\), \(\|u\|_{\ell^2(C)}\le \|u\|_{\ell^1(C)}\).
	Thus
	\[
	\|f\|_1\le \sqrt2\,\|fw\|_{(2,1),(G,H)}.
	\]
	Finally, \(\|f\|_{h,(G,H)}\le \|f\|_1\), so the claimed estimate follows.
	
	The identity \(W(R)=b_{\lfloor R\rfloor}\) is immediate from \(w(g)=b_{\ell(g)}\), and the last statement follows.
\end{proof}

\medskip

\begin{remark}\label{rem:subexp-analogue-thm12}
	As in the proof of Theorem~1.2 in \cite{ChatterjiZarka2024v1}, one obtains the corresponding subexponential statement.
	If the Schreier graph \(X=S(G,H,S)\) has subexponential growth, then \((G,H)\in \mathbf{SLC}_{\mathrm{subexp}}\).
	Conversely, if \(H\) is co-amenable in \(G\) and \((G,H)\in \mathbf{SLC}_{\mathrm{subexp}}\), then \(X\) has subexponential growth.
	Hence, under co-amenability, \((G,H)\in \mathbf{SLC}_{\mathrm{subexp}}\) if and only if \(X\) has subexponential growth.
	
	For the forward implication, write \(\mu(n)=|B_X(H,n)|\), \(r(c)=d_X(H,c)\),
	choose \(\beta_n:=(n+1)(\mu(n)+1)\), and set \(w(g):=\beta_{r(gH)}\).
	Then, exactly as in Theorem~1.2 of \cite{ChatterjiZarka2024v1}, one gets
	\[
	\|f\|_h\le C_h\|fw\|_{(2,1),(G,H)},
	\qquad
	C_h^2=\sum_{c\in G/H}\beta_{r(c)}^{-2}<\infty,
	\]
	and the radial majorant \(W(t):=\sup_{\ell(g)\le t}w(g)\) is subexponential when \(\mu\) is.
	
	For the reverse implication, co-amenability is used only through Proposition~3.6 of \cite{ChatterjiZarka2024v1}, namely
	\(\|f\|_h=\|f\|_1\) on \(\mathbb R_+G\), which yields
	\(\mu(n)\le C_h^2W(n)^2\) from the test functions \(f_n\) supported on minimal representatives of \(B_X(H,n)\).
\end{remark}

\medskip

\begin{lemma}\label{lem:free-factor-intersection}
	Let $K=H*L$ be a nontrivial free product, and let $k\in L\setminus\{e\}$.
	Then
	\[
	H\cap k^{-1}Hk=\{e\}.
	\]
\end{lemma}

\begin{proof}
	Take $x\in H\cap k^{-1}Hk$. Then $x\in H$, and there exists $y\in H$ such that
	$x=k^{-1}yk$, i.e. $kxk^{-1}=y\in H$.
	If $x\neq e$, then in the free product normal form for $H*L$,
	the word $kxk^{-1}$ is reduced and has the pattern
	\[
	(\text{nontrivial in }L)\cdot(\text{nontrivial in }H)\cdot(\text{nontrivial in }L),
	\]
	hence cannot lie in $H$. Therefore $x=e$.
\end{proof}

\begin{proposition}\label{thm:F2-rank-ge2-bad}
	Let $G=F_n (n\ge 2)$ with a fixed word length $\ell_G$. Let $H\le F_n$ be a finitely generated subgroup such that
	$[F_n:H]=\infty$ and $\operatorname{rank}(H)\ge 2$.
	Then $(F_n,H)\notin \mathbf{SLC}_{\mathrm{subexp}}.$

\end{proposition}

\begin{proof}
	Assume, for contradiction, that $(F_n,H)$ satisfies $\mathbf{SLC}_{\mathrm{subexp}}$.
	By Marshall Hall's theorem for finitely generated subgroups of free groups,
	there exists a finite-index subgroup $K\le F_n$ such that
	$K=H*L.$ Since $[F_n:H]=\infty$, we must have $L\neq\{e\}$ (otherwise $K=H$ would be finite index).
	Choose $k\in L\setminus\{e\}$. By Lemma~\ref{lem:free-factor-intersection}, applied to both $k$ and $k^{-1}$, we have $H\cap k^{-1}Hk=\{e\}$ and 
	$H\cap kHk^{-1}=\{e\}.$
	By Nielsen--Schreier, $H$ is free of rank $r:=\rank(H)\ge 2$.
	Fix a free basis $B=\{u_1,\dots,u_r\}$ for $H$ and let $|\cdot|_H$ be the associated word length.
	Define
	\[
	H_R:=\{h\in H:\ |h|_H\le R\},\qquad N_R:=|H_R|.
	\]
	The standard growth formula for free groups gives, for $R\ge 1$,
     \[
      N_R
      =1+2r\sum_{j=0}^{R-1}(2r-1)^j
      \ge (2r-1)^R.
     \]
     Let$M:=\max_{u\in B}\ell_G(u)<\infty.$
    Then for every $h\in H_R$ one has $\ell_G(h)\le MR.$
    
	Define
   \[
    f_R := \mathbf{1}_{H_R k}, \qquad \varphi_R := \mathbf{1}_{k^{-1} H_R^{-1}}.
    \]
	Since $\supp(\varphi_R)\subset k^{-1}H$ is contained in a single left coset, and $|\supp(\varphi_R)|=N_R$, we get
	\[
	\|\varphi_R\|_{(2,1),(G,H)}=\|\varphi_R|_{k^{-1}H}\|_1=N_R.
	\]
	We claim that the points of $H_Rk$ lie in pairwise distinct left cosets of $H$.
	Indeed, if $h_1kH=h_2kH$, then
	\[
	(h_2k)^{-1}(h_1k)=k^{-1}h_2^{-1}h_1k\in H,
	\]
	so $h_2^{-1}h_1\in kHk^{-1}\cap H=\{e\}$, hence $h_1=h_2$.
	Moreover, for each $h\in H_R$ one has $H_Rk\cap hkH=\{hk\}$, so on each such coset the $\ell^1$-norm equals $1$.
	Therefore
	\[
	\|f_R\|_{(2,1),(G,H)}^2=\sum_{h\in H_R}1^2=N_R,
	\qquad
	\|f_R\|_{(2,1),(G,H)}=\sqrt{N_R}.
	\]
	
	For any $x\in G$,
\begin{equation*}
	(f_R*\varphi_R)(x)=\sum_{y\in G} f_R(y)\,\varphi_R(y^{-1}x)=\#\{(h_1,h_2)\in H_R^2:\ x=h_1h_2^{-1}\}.
\end{equation*}
	So $\operatorname{supp}(f_R*\varphi_R)\subset H$. Therefore
	\[
	\|f_R*\varphi_R\|_{(2,1),(G,H)}
	=\|(f_R*\varphi_R)|_H\|_1
	=\sum_{x\in H}(f_R*\varphi_R)(x)
	=|H_R|^2
	=N_R^2.
	\]
	By definition of $\|\cdot\|_h$,
	$
	\|f_R\|_{h,(G,H)}
	\ge
	\frac{\|f_R*\varphi_R\|_{(2,1),(G,H)}}{\|\varphi_R\|_{(2,1),(G,H)}}
	=
	\frac{N_R^2}{N_R}
	=
	N_R.
	$
	From $\mathbf{SLC}_{\mathrm{subexp}}$,
	\[
	N_R\le \|f_R\|_{h,(G,H)}
	\le C_h\|f_Rw\|_{(2,1),(G,H)}.
	\]
	For $h\in H_R$,
	$
	\ell_G(hk)\le \ell_G(h)+\ell_G(k)\le MR+\ell_G(k)=:T_R,
	$
	hence $w(hk)\le W(T_R)$. Since points $hk$ lie in distinct left cosets,
	\[
	\|f_Rw\|_{(2,1),(G,H)}^2
	=\sum_{h\in H_R}w(hk)^2
	\le N_RW(T_R)^2.
	\]
	Thus
	\[
	N_R\le C_hW(T_R)\sqrt{N_R},
	\qquad
	W(T_R)\ge \frac{\sqrt{N_R}}{C_h}\ge \frac{(2r-1)^{R/2}}{C_h}.
	\]
	Taking logarithms, we obtain
	\[
	\frac{\log W(T_R)}{T_R}
	\ge
	\frac{\frac{R}{2}\log(2r-1)-\log C_h}{MR+\ell_G(k)}.
	\]
	contradicting subexponentiality of $W$. Therefore $(F_n,H)$ cannot satisfy $\mathbf{SLC}_{\mathrm{subexp}}$.
\end{proof}

     \begin{theorem}\label{thm:F2-fg-classification-subexpSLC}
     	Let $H\le F_n (n\ge 2)$ be a finitely generated subgroup. Then
     	$(F_n,H)\in \mathbf{SLC}_{\mathrm{subexp}}$ if and only if
     	either $[F_n:H]<\infty$, or $H\cong \mathbb Z$, or $H=\{e\}$.
     \end{theorem}
     
     \begin{proof}
     	Assume first that $(F_n,H)\in\mathbf{SLC}_{\mathrm{subexp}}$. If $[F_n:H]=m<\infty$, then for $f\in\mathbb C F_n$,
     	\[
     	\|f\|_{h,(F_n,H)}\le \|f\|_1\le \sqrt m\,\|f\|_{(2,1),(F_n,H)}.
     	\]
     	Thus $\mathbf{SLC}_{\mathrm{subexp}}$ holds with $w\equiv 1$. Suppose now $[F_n:H]=\infty$.
     	If $\operatorname{rank}(H)\ge 2$, this contradicts Proposition~\ref{thm:F2-rank-ge2-bad}.
     	Hence $\operatorname{rank}(H)\le 1$.
     	Since finitely generated subgroups of free groups are free,
     	$\operatorname{rank}(H)=0$ gives $H=\{e\}$, and
     	$\operatorname{rank}(H)=1$ gives $H\cong\mathbb Z$.

     	Now assume $H=\{e\}$ or $H\cong\mathbb Z$.
     	Then $H$ is amenable. Since $F_n$ has  rapid decay, Corollary~2.14 in \cite{ChatterjiZarka2024v1} implies that the $(F_n,H)$ has pair rapid decay. Therefore
     	$(F_n,H)\in\mathbf{SLC}_{\mathrm{subexp}}$.
     \end{proof}

  \begin{corollary}
	Let $H\le F_n (n\ge 2)$ be a finitely generated subgroup. Then
		$(F_n,H)$ has pair rapid decay if and only if
		either $[F_n:H]<\infty$, or $H\cong \mathbb Z$, or $H=\{e\}$.
\end{corollary}

The finite generation assumption in Theorem~\ref{thm:F2-fg-classification-subexpSLC} is essential.

\begin{proposition}\label{prop:fg-assumption-essential}
	There exists a subgroup \(H\le F_2\) such that \((F_2,H)\in \mathbf{SLC}_{\mathrm{subexp}}\), while \([F_2:H]=\infty\), \(H\neq\{e\}\), and \(H\not\cong\mathbb Z\).
\end{proposition}

\begin{proof}
	Let
	\[
	Q=H_3(\mathbb Z)=\langle u,v,z\mid [u,v]=z,\ [u,z]=[v,z]=e\rangle,
	\]
    and fix a surjection
    $\pi:F_2=\langle a,b\rangle\twoheadrightarrow Q$ such that
    $\pi(a)=u$ and $\pi(b)=v$.
	Set $K:=\langle u\rangle\le Q$ and $H:=\pi^{-1}(K)\le F_2$.
	
	First, there is a canonical bijection
	\[
	\Theta:F_2/H\longrightarrow Q/K,\qquad \Theta(gH)=\pi(g)K,
	\]
	which is an isomorphism of Schreier graphs with respect to the generating sets $\{a^{\pm1},b^{\pm1}\}$ and $\{u^{\pm1},v^{\pm1}\}$.
	Hence $F_2/H$ and $Q/K$ have the same growth type.
	Since $Q$ has polynomial growth, so does $Q/K$ (indeed
	$|B_{Q/K}(R)|\le |B_Q(R)|$), therefore $F_2/H$ has subexponential growth.
	By Remark~\ref{rem:subexp-analogue-thm12}, the subexponential growth of $F_2/H$ implies that
	$(F_2,H)\in \mathbf{SLC}_{\mathrm{subexp}}$.

	If $v^mK=v^nK$, then $v^{m-n}\in K=\langle u\rangle$. In the abelianization
	$Q_{\mathrm{ab}}\cong \mathbb Z^2$, where $u\mapsto (1,0)$ and $v\mapsto (0,1)$, this forces
	$m=n$, thus $Q/K$ is infinite, so $[F_2:H]=|Q/K|=\infty$.
	Also $a\in H$, hence $H\neq\{e\}$.
	
	To see that $H$ is not cyclic, let $N:=\ker\pi\triangleleft F_2$.
	Since $Q$ is not free, $N\neq\{e\}$.
	A standard fact on free groups says that every nontrivial normal subgroup of
	a nonabelian free group is noncyclic (\cite[Section~1.A, Exercise~7, p.~87]{Hatcher2002}).
	Thus $N$ is noncyclic.
	Because $N\le H$, the subgroup $H$ cannot be cyclic. Moreover, by Theorem~2.10 in Chapter~I, Section~2 of \cite{MagnusKarrassSolitar2004}, it follows that $H$ is not finitely generated.
	
\end{proof}

\begin{corollary}\label{thm:SL2Z-fg-classification-subexpSLC}
	Let $G=\mathrm{SL}(2,\mathbb Z),$ and let $H\le G$ be finitely generated.
	Then $(G,H)\in \mathbf{SLC}_{\mathrm{subexp}}$ if and only if
	either $[G:H]<\infty$ or $H$ is virtually cyclic. 
\end{corollary}

\begin{proof}
	The finite-index case is exactly as in the free-group proof. If $H$ is virtually cyclic, then $H$ is amenable. Since $\mathrm{SL}(2,\mathbb Z)$ has
	RD, Corollary~2.14 in \cite{ChatterjiZarka2024v1} implies that $(G,H)$ has pair rapid decay,
	thus $(G,H)\in\mathbf{SLC}_{\mathrm{subexp}}$ with polynomial weight.
	
	It remains to prove necessity. Assume $[G:H]=\infty$ and  $H$ not virtually cyclic, and suppose, for contradiction, that $(G,H)\in\mathbf{SLC}_{\mathrm{subexp}}$.
	So there exist $C_h>0$ and $w:G\to[1,\infty)$ with
	\[
	\|f\|_{h,(G,H)}\le C_h\|fw\|_{(2,1),(G,H)},
	\qquad
	W(R):=\sup_{\ell(x)\le R}w(x)=e^{o(R)}.
	\]
	
	Let $F:=\Gamma(3)\le \mathrm{SL}(2,\mathbb Z)$, so that $F$ has finite index in $G$,
	is torsion-free, and is a free group of finite rank. Set $K:=H\cap F$. Then
	\[
	[H:K]=[H:H\cap F]\le [G:F]<\infty,
	\]
	and hence $K$ is finitely generated.

	Because \(K\) has finite index in \(H\), virtual cyclicity is equivalent for \(H\) and \(K\).
	The subgroup \(H\) is not virtually cyclic, so \(K\) is not virtually cyclic.

	We claim that \([F:K]=\infty\). Otherwise, if \([F:K]<\infty\), then
	\[
	[G:K]=[G:F]\,[F:K]<\infty.
	\]
	Since \(K\le H\), we have \([G:H]\le [G:K]<\infty\), contradicting \([G:H]=\infty\).
	

	Since $[H:K]<\infty$ and $H$ is not virtually cyclic, $K$ is not virtually cyclic; in particular,
	$K\neq\{e\}$ and $K\not\cong \mathbb Z$.
	
	Let $(C_h,w)$ be as in the definition of $(G,H)\in \mathbf{SLC}_{\mathrm{subexp}}$.
	We next show that $(F,K)\in \mathbf{SLC}_{\mathrm{subexp}}$.	
	Let \(\ell:=\ell_G|_F\), then \(\ell\) is a proper length function on \(F\).
	Moreover, for any $u\in \mathbb CF$ the coset decomposition for $(F,K)$ agrees with
	the decomposition on the $F$-orbit in $G/H$, and hence
	$\|u\|_{(2,1),(F,K)}=\|u\|_{(2,1),(G,H)}$.
	Hence for $f\in \mathbb CF$,
	\[
	\|f\|_{h,(F,K),\ell}
	= \|f\|_{h,(G,H)}
	\le C_h\|fw\|_{(2,1),(G,H)}
	= C_h\|f\,w|_F\|_{(2,1),(F,K)}.
	\]
	So \((F,K)\) satisfies the same type of inequality with weight \(w_F:=w|_F\)
	with respect to the length function \(\ell=\ell_G|_F\). Since \(F\) has finite index in \(G\), \(\ell\) is equivalent to any word length
	\(\ell_F\) on \(F\). Therefore, after changing the constants the same inequality holds with \(\ell_F\) in place of \(\ell\).
	Moreover, for some \(A,B>0\), \(\ell_G(x)\le A\,\ell_F(x)+B\) for all \(x\in F\).
	Hence
	\[
	W_F(R):=\sup_{\ell_F(x)\le R} w_F(x)\le W_G(AR+B)=e^{o(R)}.
	\]
	Therefore \((F,K)\in \mathbf{SLC}_{\mathrm{subexp}}\).

	Finally, $F$ is a nonabelian free group of finite rank, say $F\cong F_m$ with $m\ge2$.
	Applying Theorem~\ref{thm:F2-fg-classification-subexpSLC} to the pair $(F,K)$ we get $[F:K]<\infty$ or $K\cong\mathbb Z$ or $K=\{e\}$. Contradiction.
\end{proof}

\begin{remark}\label{rem:F4-factor-counterexample}
	We give a counterexample to Theorem~1.3 of \cite{ChatterjiZarka2024v1}.
	Consider
	\[
	G=F_2*F_2\cong F_4=\langle a,b,c,d\rangle,\qquad
	H=\langle a,b\rangle,\qquad
	k:=c.
	\]
	Then \(H\) is finitely generated, \([G:H]=\infty\), \(\operatorname{rank}(H)=2\),

	By Theorem \ref{thm:F2-fg-classification-subexpSLC} we get $(F_2*F_2,H)\notin \mathbf{SLC}_{\mathrm{subexp}}$, and a fortiori
	\((F_2*F_2,H)\) does not have pair rapid decay.
	Since \(F_2*F_2\) is strongly relatively hyperbolic with respect to the factor \(H\), this shows that \cite[Theorem~1.3]{ChatterjiZarka2024v1} requires additional hypotheses in the form stated in version~v1.
\end{remark}

\medskip

   Incidentally, we record here a more general statement concerning free products and $\mathbf{SLC}_{\mathrm{subexp}}$. It should be noted that the counterexample given in Remark~\ref{rem:F4-factor-counterexample} shows that the converse of this statement does not hold.

\begin{proposition}\label{prop:amalgam_inheritance_slc}
	
	Let \(A,B\) be finitely generated groups, let \(C\le A,B\), and set
	\(G:=A*_{C}B\) and \(H:=B\). If \((G,H)\in \mathbf{SLC}_{\mathrm{subexp}}\), then \((A,C)\in \mathbf{SLC}_{\mathrm{subexp}}\).
	
\end{proposition}

\begin{proof}
	Choose finite symmetric generating sets \(S_A\) of \(A\) and \(S_B\) of \(B\), and let
	\(\ell_A\) and \(\ell_G\) be the corresponding word lengths on \(A\) and \(G\), where
	\(\ell_G\) is defined by \(S_A\cup S_B\). Let \(\iota:\mathbb CA\hookrightarrow \mathbb CG\)
	denote extension by zero outside \(A\).
	
	For each left coset \(gB\in G/B\), either \(A\cap gB=\varnothing\), or
	\(A\cap gB=xC\) for any \(x\in A\cap gB\); indeed, if \(x,y\in A\cap gB\), then
	\(x^{-1}y\in A\cap B=C\). It follows that for every \(\phi\in\mathbb CA\),
	\[
	\|\iota(\phi)\|_{(2,1),(G,B)}^2
	=\sum_{gB\in G/B}\|\iota(\phi)|_{gB}\|_1^2
	=\sum_{aC\in A/C}\|\phi|_{aC}\|_1^2
	=\|\phi\|_{(2,1),(A,C)}^2.
	\]
	Moreover, for \(f,\phi\in\mathbb CA\), one has
	\(\iota(f*_A\phi)=\iota(f)*_G\iota(\phi)\), since both sides agree on \(A\), and both vanish on \(G\setminus A\). Therefore, for
	\(f\in\mathbb CA\),
	\[
	\|f\|_{h,(A,C)}
	=\sup_{\|\phi\|_{(2,1),(A,C)}=1}\|f*_A\phi\|_{(2,1),(A,C)}
	\le \|\iota(f)\|_{h,(G,B)}.
	\]
	
	Now let \(w\) witness \((G,B)\in\mathbf{SLC}_{\mathrm{subexp}}\), and set \(v:=w|_A\).
	Then for \(f\in\mathbb CA\),
	\[
	\|f\|_{h,(A,C)}
	\le \|\iota(f)\|_{h,(G,B)}
	\le C_h\,\|\iota(f)w\|_{(2,1),(G,B)}
	= C_h\,\|fv\|_{(2,1),(A,C)},
	\]
	where the last equality follows from the norm identity above.
	
	Finally, since \(S_A\subset S_A\cup S_B\), one has \(\ell_G(a)\le \ell_A(a)\) for all
	\(a\in A\). Hence
	\[
	\sup_{\ell_A(a)\le R}v(a)
	\le \sup_{\ell_G(g)\le R}w(g).
	\]
	Since the right-hand side is subexponential in \(R\), so is the left-hand side. Thus
	\((A,C)\in\mathbf{SLC}_{\mathrm{subexp}}\).
\end{proof}

\begin{corollary}
	\label{cor:finite_C_absolute}
	Let \(A\) and \(B\) be finitely generated groups, let \(C\le A,B\) be finite, and set
	\(G:=A*_C B\). If \((G,B)\in \mathbf{SLC}_{\mathrm{subexp}}\), then \(A\) has subexponential rapid decay.
	Consequently, if \(A\) does not have subexponential rapid decay, then
	\((A*_C B,B)\notin \mathbf{SLC}_{\mathrm{subexp}}\).
	In particular, \((A*B,B)\notin \mathbf{SLC}_{\mathrm{subexp}}\).
\end{corollary}

\begin{proof}
	By Proposition~\ref{prop:amalgam_inheritance_slc}, $(A,C)$ satisfies
	$\mathbf{SLC}_{\mathrm{subexp}}$ with weight $v$.
	For $\xi\in\mathbb CA$,
	\[
	\|\xi\|_2\le \|\xi\|_{(2,1),(A,C)}\le \sqrt{|C|}\,\|\xi\|_2
	\]
	(by $\ell^2\le \ell^1$ on each $C$-coset, and Cauchy--Schwarz with $|C|$).
	Hence for $f\in\mathbb CA$,
	\[
	\|\lambda_A(f)\|
	\le \sqrt{|C|}\,\|f\|_{h,(A,C)}
	\le \sqrt{|C|}\,C_h\,\|fv\|_{(2,1),(A,C)}
	\le |C|\,C_h\,\|fv\|_2.
	\]
	Since $v$ is subexponential, this is subexponential RD for $A$.
\end{proof}

\medskip

    It is natural to ask what additional condition is required in order that $(G,H)$ have pair rapid decay when $G$ is strongly relatively hyperbolic with respect to $H$. We provide a necessary and sufficient condition. Here, the term ``strongly relatively hyperbolic'' is used to mean relatively hyperbolic in the sense of Osin and Bowditch(\cite{Osin2006,Bowditch2012}). The modifier ``strongly'' is included solely to distinguish this notion from relative hyperbolicity in the sense of Farb(\cite{Farb1998}).

\begin{theorem}
	\label{thm:hybrid-rd-iff-poly-growth}
	Let $G$ be a finitely generated group, and let $H<G$ be a proper finitely generated subgroup.
	Assume that $G$ is strongly relatively hyperbolic with respect to $H$.
	Then the $(G,H)$ has pair rapid decay if and only if $H$ has polynomial growth with respect to the induced length $\ell_G|_H$.
\end{theorem}

\begin{proof}
    
    	We prove the necessity. Since $G$ is strongly relatively hyperbolic with respect to $H$, the subgroup $H$ is almost malnormal; that is, for each $t\in G\setminus H$, the subgroup $J_t:=H\cap t^{-1}Ht$ is finite (see \cite[Theorems~1.4 and~1.5]{Osin2006}). Fix such a $t$ and write $M:=|J_t|<\infty$.
    	
    	For $R\ge 0$, let
    	\[
    	F_R:=H\cap B_G(R),\qquad
    	f_R:=\sum_{h\in F_R}\delta_{tht^{-1}},\qquad
    	\phi:=\delta_t.
    	\]
    	Then $\operatorname{supp}(f_R)\subseteq B_G(R+2\ell(t))$ and $f_R*\phi=\sum_{h\in F_R}\delta_{th}$. Since $\operatorname{supp}(f_R*\phi)\subseteq tH$, we obtain
    	\[
    	\|f_R*\phi\|_{(2,1)}=|F_R|.
    	\]
    	
    	For each left coset $C=gH$, set $n_C:=|\operatorname{supp}(f_R)\cap C|$. If $th_1t^{-1},th_2t^{-1}\in C$, then $h_1^{-1}h_2\in H\cap t^{-1}Ht=J_t$, and hence $n_C\le M$. Therefore,
    	\[
    	\|f_R\|_{(2,1),H}^2=\sum_C n_C^2\le M\sum_C n_C=M|F_R|.
    	\]
    	
    	Applying the pair rapid decay property to $f_R$ and $\phi$, we obtain
    	\[
    	|F_R|
    	\le C(1+R+2\ell(t))^D\,\|f_R\|_{(2,1),H}\,\|\phi\|_{(2,1),H}
    	\le C(1+R+2\ell(t))^D\sqrt{M|F_R|},
    	\]
    	whence
    	\[
    	|F_R|\le C^2M(1+R+2\ell(t))^{2D}.
    	\]
    	Thus $H$ has polynomial growth with respect to the restricted length function $\ell_G|_H$.

	\smallskip
	
	It remains to prove the sufficiency. Assume that $H$ has polynomial growth with respect to the induced length$\ell_G|_H$.
	Since $G$ is strongly relatively hyperbolic with respect to $H$, and $H$ has
	polynomial growth, it follows from \cite[Thm.~1.1]{DrutuSapir2005} that $G$ has the rapid decay property.  Now, applying Proposition~2.15 in \cite{ChatterjiZarka2024v1},  we get the $(G,H)$ has pair rapid decay.
\end{proof}

\begin{remark}
	The necessity argument uses only the almost malnormality of $H$ in $G$ and therefore does not require $H$ to be finitely generated. In the strongly relatively hyperbolic setting considered here, peripheral subgroups are undistorted in $G$; thus $\ell_G|_H$ is equivalent to any word metric on $H$, and the latter conclusion is equivalent to polynomial growth of $H$ in the usual sense.
\end{remark}

\medskip

From the proof of necessity above, we can abstract a more general result.

\begin{definition}(\cite[\S5.1]{PetersonThom2011} and \cite[Definition~3.1]{Thom2009})\label{def:s-normal}
	Let $G$ be a group and let $H\leq G$ be an infinite subgroup. we say that $H$ is \emph{$s$-normal} in $G$ if $H\cap g^{-1}Hg$ is infinite for every $g\in G$. Equivalently, $gHg^{-1}\cap H$ is infinite for every
	$g\in G$.
\end{definition}

\begin{corollary}\label{cor:non-s-normal}
	Let $G$ be a finitely generated group with word length $\ell_G$, and let $H<G$ be a proper subgroup. Assume that  \((G,H)\in \mathbf{SLC}_{\mathrm{subexp}}\). If $H$ is not $s$-normal in $G$, then $H$ has polynomial growth with respect to the induced length $\ell_G|_H$. 
\end{corollary}

\medskip

\begin{remark}
	Corollary~\ref{cor:non-s-normal} should not be viewed as a characterization.
	
	\begin{enumerate}
		\item The condition that $H$ be not $s$-normal is not necessary for \((G,H)\in \mathbf{SLC}_{\mathrm{subexp}}\). Indeed, let $G:=F_2\times \mathbb Z$ and let $H:=\mathbb Z$ be the second direct factor. Then $H$ is finitely generated and normal in $G$, hence $s$-normal. Moreover, $H$ has polynomial growth with respect to the induced length. Since $G/H\cong F_2$ and $F_2$ has $\mathrm{RD}$, it follows from the normal-subgroup criterion that $(G,H)$ has pair rapid decay.
		
		\item The subgroup-side condition in Corollary~\ref{cor:non-s-normal} is not sufficient to guarantee that \((G,H)\in \mathbf{SLC}_{\mathrm{subexp}}\) in general. Let
		\[
		K:=(\mathbb Z/2\mathbb Z)\wr \mathbb Z,\qquad
		G:=K*\mathbb Z,\qquad
		H:=\mathbb Z,
		\]
		where $H$ is the second free factor. Then $H$ has polynomial growth and is malnormal in $G$, hence in particular $H$ is not $s$-normal. Moreover, $H$ is co-amenable in $G$. Since the Schreier graph $X=G/H$ has exponential growth, it follows from Remark~\ref{rem:subexp-analogue-thm12} that	$(G,H)\notin \mathbf{SLC}_{\mathrm{subexp}}.$
	\end{enumerate}
\end{remark}

\medskip

\begin{proposition}
	\label{prop:surface-subexp-obstruction}
	Let $G=\pi_1(\Sigma_g),  g\ge 2,$
	where $\Sigma_g$ is a closed orientable surface of genus $g$.
   Let \(H \le G\) be a finitely generated subgroup with \([G\!:\!H]=\infty\), \(H\neq\{e\}\), and \(H\not\cong \mathbb{Z}\).
	Then $(G,H)\notin \mathbf{SLC}_{\mathrm{subexp}}.$
	
\end{proposition}

	\begin{proof}

	By Scott's Theorem~3.3 \cite{Scott1978}, after conjugating \(H\) in \(G\) and choosing basepoints appropriately if necessary, there exist a finite-sheeted covering
	$p:\Sigma' \to \Sigma_g$ and a compact incompressible embedded subsurface \(Z \subset \Sigma'\) such that, if \(Y\) denotes the connected component of \(Z\) containing the chosen basepoint of \(\Sigma'\), then
	$H = p_*\bigl(i_*\pi_1(Y)\bigr),$ where \(i:Y \hookrightarrow \Sigma'\) is the inclusion. Since \(Y\) is incompressible, the map
	$i_*:\pi_1(Y)\to \pi_1(\Sigma')$
	is injective. Since \(p\) is a covering map, the map
	$p_*:\pi_1(\Sigma')\to \pi_1(\Sigma_g)$
	is injective as well. Hence \(H \cong \pi_1(Y)\).

	
	If $\partial Y=\varnothing$, then $Y=\Sigma'$, since $\Sigma'$ is connected and closed. Hence $H=p_*\pi_1(\Sigma'),$ so $H$ has finite index in $G$, contrary to the assumption $[G:H]=\infty$. Thus $\partial Y\neq\varnothing$. Since $H\neq\{e\}$ and $H\not\cong \mathbb Z$, it follows that $H$ is a nonabelian free group.
	
	Now regard $G$ as a cocompact Fuchsian group acting on $X=\mathbb H^2$. Since $G$ has no parabolic elements, neither does $H$. In the Fuchsian setting, every finitely generated subgroup without parabolic elements is convex cocompact; in particular, $H$ is quasiconvex in $G$.
	
	Let $\Lambda_H\subset \partial X\simeq S^1$ be the limit set of $H$. Since $H$ has infinite index in $G$ and is quasiconvex, $\Lambda_H$ is a proper closed subset of $S^1$ (\cite[Lemma~3.9]{KapovichShort1996}). Choose disjoint nonempty open sets $U_+,U_-\subset S^1\setminus \Lambda_H$.
	
	For cocompact Fuchsian groups, ordered pairs of attracting and repelling fixed points of hyperbolic elements are dense in $(S^1\times S^1)\setminus \Delta$. Hence there exists a hyperbolic element $a\in G$ such that $a^+\in U_+$ and $a^-\in U_-$. Since $a^-\notin \Lambda_H$ and $\Lambda_H$ is compact, we may choose an open neighborhood $V_-$ of $a^-$ such that $\overline{V_-}\cap \Lambda_H=\varnothing$. By the north--south dynamics of $a$ on $S^1$, there exists $N\in\mathbb N$ such that
	\[
	a^n(S^1\setminus V_-)\subset U_+ \qquad (n\ge N).
	\]
	Since $\Lambda_H\subset S^1\setminus V_-$, we obtain
	\[
	a^N\Lambda_H\subset U_+\subset S^1\setminus \Lambda_H.
	\]
	Set $g:=a^N$. Then $\Lambda_H\cap g\Lambda_H=\varnothing.$
	
	We claim that $H\cap gHg^{-1}=\{e\}$. Indeed, if $x\neq e$ belongs to this intersection, then, since $G$ is torsion-free and cocompact Fuchsian, $x$ is hyperbolic. Its two fixed points on $S^1$ must then lie both in $\Lambda_H$ and in $\Lambda_{gHg^{-1}}=g\Lambda_H$, contradicting the disjointness above. Thus $H$ is not $s$-normal in $G$.
	
	Finally, since $H$ is quasiconvex in the hyperbolic group $G$, it is undistorted in $G$. Equivalently, the restricted length $\ell_G|_H$ is equivalent to a word metric on $H$. As $H$ is a nonabelian free group, it has exponential growth with respect to any word metric, and hence cannot have polynomial growth with respect to $\ell_G|_H$. Corollary~\ref{cor:non-s-normal} therefore implies that $(G,H)\notin \mathbf{SLC}_{\mathrm{subexp}}$.
	
\end{proof}

\begin{theorem}
	Let $G=\pi_1(\Sigma_g)$, where $\Sigma_g$ is a closed orientable surface of genus $g\ge 2$, and let $H\le G$ be a finitely generated subgroup. Then $(G,H)\in \mathbf{SLC}_{\mathrm{subexp}}$ if and only if either $H$ has finite index in $G$, or $H$ is trivial or infinite cyclic.
\end{theorem}

\begin{proof}
	The necessity is a consequence of Proposition~\ref{prop:surface-subexp-obstruction}. The sufficiency is established by the same argument as in the proof of the corresponding implication in Theorem~\ref{thm:F2-fg-classification-subexpSLC}.
\end{proof}

	\medskip

\begin{proposition}\label{cor:quasiconvex}
	Let \(G\) be a non-elementary finitely generated hyperbolic group. Suppose \(H\le G\) is finitely generated, quasiconvex, of infinite index, and not virtually cyclic. Then \((G,H)\notin \mathbf{SLC}_{\mathrm{subexp}}.\)
\end{proposition}

\begin{proof}
	Since \(H\) is quasiconvex of infinite index in the hyperbolic group \(G\), the same limit-set argument as in Proposition~\ref{prop:surface-subexp-obstruction} shows that \(H\) is not \(s\)-normal in \(G\). On the other hand, \(H\) is itself a non-elementary hyperbolic group, and hence has exponential growth with respect to any word metric. Since \(H\) is quasiconvex in \(G\), it is undistorted, so \(\ell_G|_H\) is equivalent to a word metric on \(H\). Therefore \(H\) cannot have polynomial growth with respect to \(\ell_G|_H\), and Corollary~\ref{cor:non-s-normal} implies that \((G,H)\notin \mathbf{SLC}_{\mathrm{subexp}}\).
\end{proof}

\begin{corollary}\label{cor: relatively hyperbolic}
	Let $(G,\mathcal P)$ be a non-elementary relatively hyperbolic group, and let $H\le G$.
	Assume that $\Lambda(H)\subsetneq \partial(G,\mathcal P)$ and 
	$\omega_G^+(H):=\limsup_{R\to\infty}|H\cap B_G(R)|^{1/R}>1.$
	Then $(G,H)\notin \mathbf{SLC}_{\mathrm{subexp}}.$
\end{corollary}

\begin{proof}
Immediate from Corollary~\ref{cor:non-s-normal}, since the standard limit-set argument
shows that $\Lambda(H)\subsetneq \partial(G,\mathcal P)$ implies that $H$ is not $s$-normal in $G$, while the hypothesis $\omega_G^+(H)>1$ shows that $H$ does not have polynomial growth with respect to the ambient word metric on $G$.
\end{proof}

\begin{example}
	Let $M$ be a finite-volume non-compact hyperbolic $3$-manifold, and let $G=\pi_1(M)$.
	If $\mathcal P$ denotes the collection of cusp subgroups, then $(G,\mathcal P)$ is
	non-elementary relatively hyperbolic, but $G$ is not Gromov hyperbolic since it contains
	subgroups isomorphic to $\mathbb Z^2$.
	
	Choose a Schottky subgroup $H\cong F_2\le G$. Then $H$ is finitely generated and has
	exponential growth, so $\omega_G^+(H)>1$. Moreover, $\Lambda(H)$ is a Cantor subset of the
	Bowditch boundary $\partial(G,\mathcal P)\cong S^2$, hence
	$\Lambda(H)\subsetneq \partial(G,\mathcal P)$. Therefore, by the preceding corollary,
	$(G,H)\notin \mathbf{SLC}_{\mathrm{subexp}}$.
	
	This example is not covered by Proposition~\ref{cor:quasiconvex}, since $G$ is not
	hyperbolic.
\end{example}

\begin{remark}
	The Corollary~\ref{cor: relatively hyperbolic} should not be interpreted as suggesting a rigid kernel-like structure. Even under the hypotheses \((G,H)\in \mathbf{SLC}_{\mathrm{subexp}}\) and \(\omega_G^+(H)>1\), it is false in general that, after passing to a finite-index subgroup \(G_0\le G\), the subgroup \(H_0:=H\cap G_0\) must be commensurated in \(G_0\), let alone normal or the kernel of an epimorphism.
	
	Indeed, by \cite{dowdall2017} and together with their discussion recovering the rank-two free examples of \cite{bestvina1997}, there exist \(N\ge 3\) and a free subgroup $\Gamma \le \operatorname{Out}(F_N)$, with \(\Gamma\cong F_2\), such that the associated extension
	\[
	1\longrightarrow F_N\longrightarrow G \xrightarrow{p} \Gamma \longrightarrow 1
	\]
	is a non-elementary hyperbolic group. Fix an infinite cyclic subgroup \(K\le \Gamma\), and set \(H:=p^{-1}(K)\). Then \(H\) is finitely generated and of infinite index in \(G\). Since \(F_N\le H\), the subgroup \(H\) has exponential growth, hence \(\omega_G^+(H)>1\). Moreover, \(F_N\triangleleft G\) is an infinite normal subgroup of the hyperbolic group \(G\), so \(\Lambda(F_N)=\partial G\); therefore \(\Lambda(H)=\partial G\) as well.
	
	On the other hand, since \(F_2\) has property \(\mathrm{RD}\) and \(K\cong \mathbb Z\) is amenable, hence the pair \((\Gamma,K)\) satisfies \(\mathbf{SLC}_{\mathrm{subexp}}\). Pulling this estimate back along the quotient map \(p\) yields \((G,H)\in \mathbf{SLC}_{\mathrm{subexp}}\) (See Lemma~\ref{lem:pullback-pair-rd-slc} for the proof).
	
	We claim that \(H\) does not become commensurated in any finite-index subgroup of \(G\). Suppose otherwise that \(G_0\le G\) has finite index and that \(H_0:=H\cap G_0\) is commensurated in \(G_0\). Let \(\Gamma_0:=p(G_0)\) and \(K_0:=K\cap \Gamma_0\). Then \(\Gamma_0\) is a finite-index subgroup of \(\Gamma\), hence a non-abelian free group and \(K_0\) is infinite cyclic. Since \(p(H_0)=K_0\), the commensurability of \(H_0\) in \(G_0\) forces \(K_0\) to be commensurated in \(\Gamma_0\).
	But this is impossible, Let \(C\) be the maximal cyclic subgroup of \(\Gamma_0\) containing \(K_0\). Since \(\Gamma_0\) is non-abelian free, \(C\) is proper. Choose \(\gamma\in \Gamma_0\setminus N_{\Gamma_0}(C)\). If \(K_0\cap \gamma K_0\gamma^{-1}\) were nontrivial, then the maximal cyclic subgroups containing \(K_0\) and \(\gamma K_0\gamma^{-1}\) would coincide, forcing \(\gamma\in N_{\Gamma_0}(C)\), a contradiction. Thus \(K_0\cap \gamma K_0\gamma^{-1}=\{e\}\), so \(K_0\) is not commensurated in \(\Gamma_0\). Hence \(H_0\) is not commensurated in \(G_0\).
	
	In particular, \(\mathbf{SLC}_{\mathrm{subexp}}\) together with \(\omega_G^+(H)>1\) does not force any virtual kernel-type conclusion.
\end{remark}

\medskip

\begin{corollary}\label{cor:HJH-vc}
	Let \((G,\mathcal P)\) be a non-elementary relatively hyperbolic group, and let
	\(H\le G\) be a finitely generated relatively quasiconvex subgroup with \([G:H]=\infty\).
	Set
	\[
	J_H:=\big\langle\,H\cap gPg^{-1}: P\in\mathcal P,\ \lvert H\cap gPg^{-1}\rvert=\infty\,\big\rangle.
	\]
	If \((G,H)\in \mathbf{SLC}_{\mathrm{subexp}}\), then \(H/J_H\) is virtually cyclic.
\end{corollary}

\begin{proof}
	Assume, for contradiction, that $Q:=H/J_H$	is not virtually cyclic. By Hruska's induced peripheral structure theorem \cite[Theorem~9.1]{Hruska2010},
	there exist finitely many subgroups $O_1,\dots,O_m\le H$ such that
	$(H,\mathcal O)$ is relatively hyperbolic, where $\mathcal O=\{O_1,\dots,O_m\}$, and
	\[
	J_H=\langle\!\langle O_1,\dots,O_m\rangle\!\rangle_H.
	\]
	
	We claim that $(H,\mathcal O)$ is non-elementary. Indeed, if $(H,\mathcal O)$ were
	elementary, then $H$ would be finite, virtually cyclic, or parabolic. In the first two
	cases, the quotient $Q$ would be finite or virtually cyclic. In the parabolic case,
	$H$ is conjugate in $H$ to one of the peripheral subgroups $O_i$, hence
	\[
	J_H=\langle\!\langle O_1,\dots,O_m\rangle\!\rangle_H=H,
	\]
	so $Q$ is trivial. In every case, $Q$ is virtually cyclic, contrary to assumption.
	Thus $(H,\mathcal O)$ is non-elementary.
	
	Since $H$ is finitely generated and relatively hyperbolic with respect to the proper
	subgroups in $\mathcal O$, Xie's theorem \cite[Theorem~1.1]{Xie2007} implies that
	$H$ has uniform exponential growth. Hence, for some finite generating set $T$ of $H$ with associated word length $|\cdot|_T$, one has $\omega(H,T)>1.$
	Let $C:=\max_{t\in T}\ell_G(t).$ Then $\ell_G(h)\le C\,|h|_T$ for $h\in H$, 
	so $B_H^T(n)\subseteq H\cap B_G(Cn).$
	Therefore $\omega_G^+(H)\ge \omega(H,T)^{1/C}>1.$

	On the other hand, since $H$ is relatively quasiconvex in $(G,\mathcal P)$ and
	$[G:H]=\infty$, so $\Lambda(H)\subsetneq \partial(G,\mathcal P).$
	Applying Corollary~\ref{cor: relatively hyperbolic}, we obtain $
	(G,H)\notin \mathbf{SLC}_{\mathrm{subexp}},$
	contrary to the assumption. Hence $H/J_H$ must be virtually cyclic.
\end{proof}

\begin{theorem}\label{thm:SLC-characterization-locally-qc}
   Let \(G\) be a non-elementary finitely generated hyperbolic group, and assume that \(G\) is locally quasiconvex. Let \(H\le G\) be a finitely generated subgroup. Then \((G,H)\in \mathbf{SLC}_{\mathrm{subexp}}\) if and only if either \(H\) has finite index in \(G\) or \(H\) is virtually cyclic.
\end{theorem}
\begin{proof}
    The necessity follows from Proposition~\ref{cor:quasiconvex}. The proof of sufficiency is similar to that of Theorem~\ref{thm:F2-fg-classification-subexpSLC}.
\end{proof}

  \medskip

While $(G,H)\in \mathbf{SLC}_{\mathrm{subexp}}$ clearly holds when $H\le G$ has finite index or is virtually cyclic, quasiconvexity cannot be omitted: the conclusion fails in general for non-quasiconvex subgroups of non-elementary finitely generated hyperbolic groups.

\begin{proposition}\label{prop:ripsconstruction}
	There exist a non-elementary hyperbolic group $G$ and a finitely generated subgroup $H\le G$ such that $H$ is not quasiconvex in $G$ and $(G,H)\notin \mathbf{SLC}_{\mathrm{subexp}}$.
\end{proposition}

\begin{proof}
	Let $Q:=BS(1,2)=\langle a,t\mid tat^{-1}=a^2\rangle.$
	Then \(Q\) is finitely presented, solvable, and of exponential growth.
	By the Rips construction, there exists a short exact sequence
	\[
	1\longrightarrow N\longrightarrow G\xrightarrow{\pi} Q\longrightarrow 1
	\]
	with \(G\) non-elementary hyperbolic and \(N\) finitely generated.
	Set \(H:=N\). 
	
	Since \(Q\) is infinite, we have $[G:N] \;=\; |G/N| \;=\; |Q| \;=\; \infty.$ We show that \(N\) must be infinite. Indeed, if \(N\) were finite, then \(Q \cong G/N\) would be a finite-kernel quotient of a hyperbolic group, hence hyperbolic. contradicting the fact that $BS(1,2)$ is not hyperbolic. Therefore \(N\) is infinite. Because \(N \trianglelefteq G\) is an infinite normal subgroup of a non-elementary hyperbolic group, its limit set equals the whole boundary, that is $\Lambda_N = \Lambda_G = \partial G$ by Theorem~12.2(5) of \cite{KapovichBenakli2002}. Now suppose, for contradiction, that \(N\) is quasiconvex in \(G\). For an infinite quasiconvex subgroup \(L \le G\), one has the standard criterion $\Lambda_L=\partial G$ iff  $[G:L]<\infty$ (\cite[Lemma~3.9 and Lemma~3.3(5)]{KapovichShort1996}). Applying this with \(L=N\) and using \(\Lambda_N=\partial G\), we obtain \([G:N]<\infty\), contradicting \([G:N]=\infty\).
	Hence \(N\) is not quasiconvex. 
	By generalizing Proposition~2.18 in \cite{ChatterjiZarka2024v1}, we obtain if 
	\((G,N)\in \mathbf{SLC}_{\mathrm{subexp}}\), then 
	\((Q,\{e\})\in \mathbf{SLC}_{\mathrm{subexp}}\), 
	which contradicts the choice of \(Q\). Therefore \((G,H)\notin \mathbf{SLC}_{\mathrm{subexp}}\).
\end{proof}

\begin{remark}
	
	The generalization of Proposition~2.18 in \cite{ChatterjiZarka2024v1} is standard. The only point we wish to emphasize is that the converse implication admits a shorter argument, without introducing a section of the quotient map or the associated cocycle.
	 Let \(Q:=G/H\), and assume that \((G,H)\) has the rapid decay property. Given \(F \in \mathbb{C}Q\) with \(\operatorname{supp}(F)\subset B_Q(R)\), choose for each \(q\in \operatorname{supp}(F)\) a lift \(g_q\in G\) obtained from a reduced word for \(q\), so that \(\pi(g_q)=q\) and \(\ell_G(g_q)=\ell_Q(q)\le R\), and set \(u:=\sum_{q\in Q}F(q)\delta_{g_q}\). Then \(\operatorname{supp}(u)\subset B_G(R)\) and \(\|u\|_{(2,1)}=\|F\|_2\), since \(u\) has at most one nonzero coefficient on each left coset of \(H\). Under the canonical identification \(\ell^2(G/H)\simeq \ell^2(Q)\), the quasi-regular representation factors through the left regular representation of \(Q\), so that \(\lambda_{G/H}(u)=\lambda_Q(F)\). Hence Lemma~2.12(1) and Proposition~2.7(3) in \cite{ChatterjiZarka2024v1} imply that
	\[
	\|\lambda_Q(F)\|
	=
	\|\lambda_{G/H}(u)\|
	\le
	\|u\|_h
	\le
	C(1+R)^D\|u\|_{(2,1)}
	=
	C(1+R)^D\|F\|_2,
	\]
	which is exactly the rapid decay estimate for \(Q\).
\end{remark}

\medskip

   The previous counterexample is obtained via a Rips-type construction. This raises the question of whether the additional condition $(\mathrm{NR}_\infty)$, namely that every infinite finitely generated normal subgroup of $G$ has finite index in $G$, restores a clean classification. More precisely, if $G$ is a non-elementary finitely generated hyperbolic group satisfying $(\mathrm{NR}_\infty)$ and $H\le G$ is finitely generated, must the condition $(G,H)\in \mathbf{SLC}_{\mathrm{subexp}}$ force $H$ to have finite index in $G$ or to be virtually cyclic? The following proposition shows that the answer is negative in general.

   \begin{proposition}\label{prop:weeks-counterexample}
   	There exist a non-elementary finitely generated hyperbolic group $G$ satisfying $(\mathrm{NR}_\infty)$ and a finitely generated subgroup $H\le G$ such that $[G:H]=\infty$, $H$ is not virtually cyclic, and $(G,H)\in \mathbf{SLC}_{\mathrm{subexp}}$.
   \end{proposition}
   
\begin{proof}
	Let $M_W$ be the Weeks closed hyperbolic $3$-manifold, and set $G:=\pi_1(M_W)$. Then $G$ is a non-elementary hyperbolic group.
	
	By Agol's virtual fibering theorem \cite{Agol2013}, $M_W$ admits a finite-sheeted cover that fibers over $S^1$. Let $K_0\le G$ be the corresponding finite-index subgroup, and let
	\[
	K:=\bigcap_{g\in G} gK_0g^{-1}\triangleleft G
	\]
	be its normal core. Then $K$ still has finite index in $G$. Since finite covers of fibered $3$-manifolds are fibered, there is a short exact sequence
	\[
	1\longrightarrow H\longrightarrow K\longrightarrow \mathbb Z\longrightarrow 1,
	\]
	where $H\cong \pi_1(\Sigma)$ for a closed hyperbolic surface $\Sigma$. In particular, $H$ is finitely generated and not virtually cyclic, and $[G:H]\ge [K:H]=\infty$.
	
	We claim that $G$ satisfies $(\mathrm{NR}_\infty)$. Let $N\triangleleft G$ be an infinite finitely generated normal subgroup. If $[G:N]=\infty$, then by the normal-subgroup case of Canary's covering theorem, the quotient $G/N$ is virtually cyclic. Hence $G$ admits an epimorphism onto either $\mathbb Z$ or $D_\infty$ (see also Theorem~5.3 and Corollary~6.4 in \cite{boileau2018}). This is impossible, since
	\[
	H_1(G;\mathbb Z)\cong (\mathbb Z/5)^2
	\]
	for the Weeks manifold: in particular, $G$ admits no epimorphism onto $\mathbb Z$, nor onto $D_\infty$, since any epimorphism $G\twoheadrightarrow D_\infty$ would induce a nontrivial homomorphism $G\to \mathbb Z/2$. Therefore every infinite finitely generated normal subgroup of $G$ has finite index, and so $G$ satisfies $(\mathrm{NR}_\infty)$.
	
	Finally, since $H\triangleleft K$ and $K/H\cong \mathbb Z$, the subgroup $H$ is co-amenable in $K$, and the Schreier graph $K/H$ has polynomial growth. Hence, by Remark~\ref{rem:subexp-analogue-thm12}, one has $(K,H)\in \mathbf{SLC}_{\mathrm{subexp}}.$
	Since $[G:K]<\infty$, the finite-index permanence of $\mathbf{SLC}_{\mathrm{subexp}}$ for pairs implies that $(G,H)\in \mathbf{SLC}_{\mathrm{subexp}}.$
	Thus $G$ is a non-elementary hyperbolic group satisfying $(\mathrm{NR}_\infty)$, the subgroup $H$ is finitely generated, of infinite index, and not virtually cyclic, and yet $(G,H)\in \mathbf{SLC}_{\mathrm{subexp}}$.
\end{proof}

 \medskip
 
\begin{lemma}\label{lem:normal-intersection-kleinian}
	Let \(\Gamma<\mathrm{Isom}^+(\mathbb H^3)\) be a torsion-free non-elementary
	Kleinian group. Then any finite intersection of nontrivial normal subgroups of
	\(\Gamma\) is nontrivial.
\end{lemma}

\begin{proof}
	It is enough to treat the intersection of two such subgroups. Let
	\(A,B\triangleleft \Gamma\) be nontrivial. Since \(\Gamma\) is torsion-free,
	neither \(A\) nor \(B\) contains nontrivial finite subgroups. An infinite
	elementary normal subgroup would preserve a finite subset of
	\(\partial\mathbb H^3\), forcing \(\Gamma\) to be elementary; hence every
	nontrivial normal subgroup of \(\Gamma\) is non-elementary.
	
	Choose a loxodromic element \(a\in A\). Since \(\Gamma\) is torsion-free and
	discrete, the centralizer \(C_\Gamma(a)\) is cyclic, in particular elementary.
	As \(B\) is non-elementary, we may choose \(b\in B\setminus C_\Gamma(a)\).
	Then \([a,b]\neq e\), and because \(A\) and \(B\) are normal in \(\Gamma\), we
	have \([a,b]\in A\cap B\). Thus \(A\cap B\neq\{e\}\).
	
	Applying this argument inductively proves the lemma.
\end{proof}

\begin{proposition}\label{prop:slc_core_closed_hyperbolic_3_manifold}
	Let \(M\) be a closed hyperbolic \(3\)-manifold, let \(G=\pi_1(M)\), and let \(H\leq G\) be a finitely generated subgroup. Assume that $(G,H)\in \mathbf{SLC}_{\mathrm{subexp}}$, $[G:H]=\infty$, and $H$ is not virtually cyclic. Then $\operatorname{Core}_G(H)\neq \{e\}$.
\end{proposition}

\begin{proof}
	Fix a word length \(\ell\) on \(G\), and write $B_G(R):=\{g\in G:\ell(g)\le R\}.$
	By the assumption \((G,H)\in \mathbf{SLC}_{\mathrm{subexp}}\), there exist a constant \(C_h>0\) and a function \(w:G\to [1,\infty)\) such that
	\[
	\|f\|_{h,(G,H)}\le C_h\,\|fw\|_{(2,1),(G,H)}
	\qquad (f\in \mathbb C G),
	\]
	and $W(R):=\sup_{\ell(g)\le R} w(g)$ has subexponential growth.
	
	For each \(x\in G\), set $K_x:=H\cap xHx^{-1}.$
	We first prove that, for every \(x\in G\) and every \(R\ge 0\),
	\begin{equation}\label{eq:Hx-growth-estimate}
		|H\cap B_G(R)|\le C_h^2\,W(R+\ell(x))^2\,|K_x\cap B_G(2R)|.
	\end{equation}
	
	Choose a set \(S_R\subset H\cap B_G(R)\) containing exactly one representative from each left \(K_x\)-coset that meets \(H\cap B_G(R)\).  Here we introduce definitions that differ slightly in form from those in Theorem~\ref{thm:hybrid-rd-iff-poly-growth}:
	\[
	f_R := \sum_{h \in S_R} \delta_{hx},
	\qquad
	\varphi := \delta_{x^{-1}}.
	\]
	Similar to the proof in Theorem~\ref{thm:hybrid-rd-iff-poly-growth}, we obtain
    $|S_R|\le C_h^2\,W(R+\ell(x))^2.$

	Now every \(h\in H\cap B_G(R)\) can be written as \(h=sk\) with \(s\in S_R\) and \(k\in K_x\). Since \(\ell(s)\le R\) and \(\ell(h)\le R\), we have $\ell(k)=\ell(s^{-1}h)\le 2R.$
	Therefore
	\[
	H\cap B_G(R)\subseteq \bigcup_{s\in S_R} s\bigl(K_x\cap B_G(2R)\bigr),
	\]
	and \eqref{eq:Hx-growth-estimate} follows.
	
	Next we show that \(H\) has exponential growth with respect to the ambient word metric on \(G\). Since \(H\) is not virtually cyclic, choose independent loxodromic elements \(a,b\in H\), then \(\langle a\rangle\) and \(\langle b\rangle\) are quasiconvex cyclic subgroups of \(G\), and
	\(\langle a\rangle\cap\langle b\rangle\) is finite. Since \(G\) is torsion-free, it follows that
	\(\langle a\rangle\cap\langle b\rangle=\{1\}\).
	By Theorem~1 in \cite{Gitik1999}, for sufficiently large \(N\),
	\[
	F:=\langle a^{N},b^{N}\rangle
	= \langle a^{N}\rangle * \langle b^{N}\rangle \cong F_2,
	\]
	and \(F\) is quasiconvex in \(G\). So \(F\) is quasi-isometrically embedded in \(G\) (Theorem~7.7 in \cite{hullhyperbolic}), its intrinsic word metric is quasi-isometric to the restriction of the ambient word metric on \(G\) to \(F\). 
	Hence there exist constants $a>1$ and $c>0$ such that
	\begin{equation}\label{eq:H-exp-growth}
		|H\cap B_G(R)|
		\;\ge\;
		|F\cap B_G(R)|
		\;\ge\;
		c\,a^{R}
		\qquad (R\gg 1).
	\end{equation}
	We claim that $\Lambda(H)=\partial G.$
	Assume otherwise. Then there exists a nonempty open set \(U\subset \partial G\setminus \Lambda(H)\). Since pairs of fixed points of loxodromic elements are dense in \(\partial^2 G\), we may choose a loxodromic element \(b\in G\) such that
	\[
	b^+\in U,
	\qquad
	b^-\notin \Lambda(H).
	\]
	By north-south dynamics, \(b^n\Lambda(H)\subset U\) for all sufficiently large \(n\). Fix such an \(n\), and set \(x=b^n\). Then
	\[
	x\Lambda(H)\cap \Lambda(H)=\varnothing.
	\]
	Therefore
	\[
	\Lambda(K_x)\subseteq \Lambda(H)\cap \Lambda(xHx^{-1})
	=\Lambda(H)\cap x\Lambda(H)
	=\varnothing,
	\]
	so \(K_x\) is finite. Now \eqref{eq:Hx-growth-estimate} implies that \(|H\cap B_G(R)|\) is bounded above by a subexponential function of \(R\), contradicting \eqref{eq:H-exp-growth}. Thus indeed $\Lambda(H)=\partial G.$
	
	Since \(M\) is a closed hyperbolic \(3\)-manifold, the Subgroup Tameness Theorem
	(a consequence of the Tameness Theorem and Canary's Covering Theorem see Theorem~4.1.2 in \cite{AschenbrennerFriedlWilton2015}) implies
	that every finitely generated subgroup of \(G=\pi_1(M)\) is either geometrically finite
	or a virtual surface fiber subgroup. Since \(H\) has infinite index and
	\(\Lambda(H)=\partial G\), it cannot be geometrically finite: indeed, for a
	geometrically finite Kleinian group, cocompactness on the convex hull of the limit set
	would here imply cocompactness on all of \(\mathbb H^3\), hence finite index in \(G\).
	
	Therefore \(H\) is a virtual surface fiber subgroup. Equivalently, there exist a finite-index subgroup
	\(G_1\le G\), a fiber subgroup \(F\triangleleft G_1\) with \(G_1/F\cong \mathbb Z\), and \(g\in G\) such that
	\(H\) and \(gFg^{-1}\) are commensurable. Replacing \((G_1,F)\) by \((gG_1g^{-1},\,gFg^{-1})\) we may assume
	that \(H\) and \(F\) are commensurable.
	
	Set
	\[
	G_0:=\operatorname{Core}_G(G_1)=\bigcap_{g\in G} gG_1g^{-1},
	\qquad
	F_0:=F\cap G_0,
	\qquad
	L:=H\cap F_0.
	\]
	Then \(G_0\triangleleft G\) has finite index, \(F_0\triangleleft G_0\), and \(L\) has finite index in \(F_0\). Define
	\[
	K_0:=\operatorname{Core}_{G_0}(L)=\bigcap_{g\in G_0} gLg^{-1}.
	\]
	Each conjugate \(gLg^{-1}\) with \(g\in G_0\) has the same finite index in \(F_0\) as \(L\), and \(F_0\) is finitely generated; hence there are only finitely many such conjugates. It follows that \(K_0\) has finite index in \(F_0\). In particular,
	\[
	K_0\neq \{e\},
	\qquad
	K_0\triangleleft G_0,
	\qquad
	K_0\le H.
	\]
	
	Now choose a finite transversal \(T\subset G\) for \(G/G_0\) with \(1\in T\), and define
	\[
	K:=\bigcap_{t\in T} tK_0t^{-1}.
	\]
	Then \(K\triangleleft G\) and \(K\le K_0\le H\). 
	We claim that \(K\neq \{e\}\). Since \(G_0\triangleleft G\), for each \(t\in T\) we have
	\(tK_0t^{-1}\triangleleft G_0\). Moreover, \(K_0\neq \{e\}\) implies
	\(tK_0t^{-1}\neq \{e\}\). As \(G_0\) is a torsion-free non-elementary Kleinian group,
	Lemma~\ref{lem:normal-intersection-kleinian} shows that
	\[
	K=\bigcap_{t\in T} tK_0t^{-1}\neq \{e\}.
	\]

	Finally, since \(K\triangleleft G\) and \(K\le H\), we get $K\le \operatorname{Core}_G(H).$ Therefore $\operatorname{Core}_G(H)\neq \{e\},$ as claimed.
\end{proof}

\begin{corollary}\label{cor:slc_positive_l2betti_hyperbolic}
	Let \(G\) be a hyperbolic group with \(\beta_1^{(2)}(G)>0\), and let \(H\le G\) be a finitely generated subgroup. Then	$(G,H)\in \mathbf{SLC}_{\mathrm{subexp}},$ iff $[G:H]<\infty$ or $H$ is virtually cyclic.
\end{corollary}

\begin{proof}
	We prove the necessity by contradiction. For \(x\in G\), set $K_x:=H\cap xHx^{-1}.$ Exactly as in the proof of Proposition~\ref{prop:slc_core_closed_hyperbolic_3_manifold}, one obtains
	\begin{equation}\label{eq:Hx-growth-estimate-l2betti}
		|H\cap B_G(R)|
		\le
		C_h^2\,W(R+\ell(x))^2\,|K_x\cap B_G(2R)|
		\qquad (R\ge 0).
	\end{equation}

Since \(H\) is a finitely generated non-virtually-cyclic subgroup of the hyperbolic group \(G\), it contains a nonabelian free subgroup. Choose \(u,v\in H\) such that $\langle u,v\rangle\cong F_2,$
and set $L:=\max\{\ell(u),\ell(v)\}.$
Then every reduced word of length \(n\) in \(\{u^{\pm1},v^{\pm1}\}\) has \(G\)-length at most \(Ln\). Hence
\[
|H\cap B_G(Ln)|
\ge
|\langle u,v\rangle\cap B_G(Ln)|
\ge 4\cdot 3^{n-1}
\qquad (n\ge 1),
\]
so \(|H\cap B_G(R)|\) grows exponentially in \(R\).

	If there existed \(x\in G\) such that \(K_x\) were finite, then \eqref{eq:Hx-growth-estimate-l2betti}
	would imply that \(|H\cap B_G(R)|\) is bounded above by a subexponential function of \(R\), a contradiction.
	Hence \(K_x\) is infinite for every \(x\in G\). In other words, \(H\) is \(s\)-normal in \(G\), and therefore \(ws\)-normal.
	
	Now Corollary~5.13 of \cite{PetersonThom2011} applies: an infinite, finitely generated,
	\(ws\)-normal subgroup of a countable group \(G\) with \(\beta_1^{(2)}(G)>0\) must have finite index.
	Therefore \([G:H]<\infty\), contradicting our assumption. This contradiction proves the result.
\end{proof}

\begin{remark}
	If \(M\) is a closed hyperbolic \(d\)-manifold and \(G=\pi_1(M)\), then
	$\beta_i^{(2)}(G)=\beta_i^{(2)}(M),$ and \(\beta_i^{(2)}(M)\) may be nonzero only for \(i=d/2\). Hence all \(L^2\)-Betti numbers vanish when \(d\) is odd; in particular, for \(d=3\), $\beta_1^{(2)}(G)=0.$ Therefore, a hyperbolic group with \(\beta_1^{(2)}(G)>0\) cannot be isomorphic to \(\pi_1(M)\) for any closed hyperbolic \(3\)-manifold \(M\).

	On the other hand, \(\beta_1^{(2)}(G)>0\) does \emph{not} imply local quasiconvexity. Indeed, by \cite{Kapovich1995} there exists a hyperbolic group \(H_0\) containing a finitely presented non-quasiconvex subgroup \(K\). For \(N\ge 1\), define
	\[
	G:=H_0*F_N.
	\]
	Then \(G\) is hyperbolic, and \(H_0\) is quasiconvex in \(G\) (as a free factor). Since \(H_0\le G\) is quasiconvex, any subgroup of \(H_0\) that is quasiconvex in \(G\) is quasiconvex in \(H_0\); thus \(K\) remains non-quasiconvex in \(G\). Consequently, \(G\) is not locally quasiconvex.
	
	Now write
	\[
	H_0=\langle x_1,\dots,x_m\mid R_1,\dots,R_r\rangle,
	\]
	so
	\[
	G=\langle x_1,\dots,x_m,y_1,\dots,y_N\mid R_1,\dots,R_r\rangle.
	\]
	Applying \cite[Thm.~3.2]{PetersonThom2011} gives
	\[
	\beta_1^{(2)}(G)\ge (m+N)-1-r.
	\]
	Hence choosing \(N>r-m+1\) we obtain \(\beta_1^{(2)}(G)>0\). In particular,
   $\beta_1^{(2)}(G) > 0$ does not imply that $G$ is locally quasiconvex.
\end{remark}

\begin{lemma}\label{lem:poly-growth-schreier-commensurable-normal}
	Let \(G\) be a finitely generated group and \(H\le G\). Suppose that there exist a finite-index subgroup
	\(\Gamma\le G\) and a normal subgroup \(F\triangleleft \Gamma\) such that \(\Gamma/F\) has polynomial growth
	and \(H\cap \Gamma\) is commensurable with \(F\). Then, for every finite symmetric generating set \(S\) of \(G\),
	the Schreier graph \(\mathcal S(G,H,S)\) has polynomial growth.
\end{lemma}

\begin{proof}
	Let
	\[
	G_0:=\operatorname{Core}_G(\Gamma)=\bigcap_{g\in G}g\Gamma g^{-1},\qquad
	H_0:=H\cap G_0,\qquad
	F_0:=F\cap G_0.
	\]
	Then $G_0\triangleleft G$ has finite index, $F_0\triangleleft G_0$, and
	\[
	G_0/F_0\cong G_0F/F\le \Gamma/F,
	\]
	so $G_0/F_0$ has polynomial growth. Since $H\cap\Gamma$ is commensurable with $F$, also
	$H_0=(H\cap\Gamma)\cap G_0$ is commensurable with $F_0=F\cap G_0$, hence
	$L:=H_0\cap F_0$ has finite index in both $H_0$ and $F_0$.
	
	Let $\pi:G_0/L\to G_0/F_0$ and $\rho:G_0/L\to G_0/H_0$ be the natural maps.
	Then each fiber of $\pi$ has size $[F_0:L]$, while $\rho$ is surjective. Therefore, for any finite
	symmetric generating set $U$ of $G_0$ and every $R$,
	\[
	|B_{G_0/L,U}(R)|\le [F_0:L]\cdot |B_{G_0/F_0,U}(R)|,\qquad
	|B_{G_0/H_0,U}(R)|\le |B_{G_0/L,U}(R)|.
	\]
	Hence polynomial growth of $G_0/F_0$ implies polynomial growth of $G_0/L$, and therefore of $G_0/H_0$.
	
	Finally, let $T\subset G$ be a finite set of representatives for $G_0\backslash G$, chosen with $1\in T$, and define
	\[
	C:=\{\,tsu^{-1}\in G_0:\ t,u\in T,\ s\in S,\ ts\in G_0u\,\}.
	\]
	By Reidemeister--Schreier, $C$ generates $G_0$; set $\widetilde C:=C\cup C^{-1}$. Moreover, a standard Reidemeister--Schreier rewriting shows that if $\ell_S(g)\le R$, then $g=g_0t$ for some $t\in T$ with $\ell_{\widetilde C}(g_0)\le R$. Hence
	\[
	B_{\mathcal S(G,H,S)}(R)
	\subseteq
	\bigcup_{t\in T}\{\,g_0tH:\ \ell_{\widetilde C}(g_0)\le R\,\}.
	\]
	For each $t\in T$, the set on the right is naturally identified with the radius-$R$ ball in the Schreier graph of
	$G_0/(G_0\cap tHt^{-1})$ with respect to $\widetilde C$. Since	$G_0\cap tHt^{-1}=tH_0t^{-1}$
	is commensurable with $tF_0t^{-1}\triangleleft G_0$, and $G_0/tF_0t^{-1}\cong G_0/F_0$  has polynomial growth, the preceding argument shows that each such Schreier graph has polynomial growth. Since $T$ is finite, the above inclusion yields a polynomial upper bound on $|B_{\mathcal S(G,H,S)}(R)|$, and hence $\mathcal S(G,H,S)$ has polynomial growth.
\end{proof}

\medskip

\begin{theorem}\label{thm:slc-classification-closed-hyperbolic-3-manifold}
	Let \(M\) be a closed hyperbolic \(3\)-manifold, let \(G=\pi_1(M)\), and let \(H\leq G\) be a finitely generated subgroup. Then $(G,H)\in \mathbf{SLC}_{\mathrm{subexp}}$ if and only if either $H$ is virtually cyclic or $\operatorname{Core}_G(H)\neq \{e\}$.
\end{theorem}

\begin{proof}
	For necessity, assume that \((G,H)\in \mathbf{SLC}_{\mathrm{subexp}}\) and that
	\(H\) is not virtually cyclic. If \([G:H]=\infty\), then
	Proposition~\ref{prop:slc_core_closed_hyperbolic_3_manifold} gives
	\(\operatorname{Core}_G(H)\neq\{e\}\). If \([G:H]<\infty\), then
	\(\operatorname{Core}_G(H)\) is the kernel of the action of \(G\) on the finite
	set \(G/H\), so it has finite index in \(G\). Since \(G\) is infinite, it follows
	that \(\operatorname{Core}_G(H)\neq\{e\}\). Thus either \(H\) is virtually cyclic
	or \(\operatorname{Core}_G(H)\neq\{e\}\).
     For the converse, assume that \(N:=\operatorname{Core}_G(H)\neq\{e\}\).
	Since \(G=\pi_1(M)\) is torsion-free and hyperbolic, the nontrivial normal
	subgroup \(N\triangleleft G\) is infinite, hence \(\Lambda(N)\neq\varnothing\).
    By Theorem~12.2(5) in \cite{KapovichBenakli2002}, since \(N\triangleleft G\), we have
	\[
	\Lambda(N)=\Lambda(G)=\partial G.
	\]
	As \(N\le H\), it follows that \(\Lambda(H)=\partial G\).
	Exactly as in the proof of Proposition~\ref{prop:slc_core_closed_hyperbolic_3_manifold},
	once one knows that \(\Lambda(H)=\partial G\), tameness together with Canary's
	covering theorem implies that \(H\) is either finite index in \(G\) or a virtual
	surface fiber subgroup. 
	
	If \([G:H]<\infty\), then \((G,H)\in \mathbf{SLC}_{\mathrm{subexp}}\). Assume therefore that \(H\) is a virtual surface fiber subgroup, After conjugating
	if necessary there exists a finite-index subgroup \(\Gamma\le G\) and a fiber subgroup
	\(F\triangleleft \Gamma\) with \(\Gamma/F\cong \mathbb Z\) such that
	\(H\cap \Gamma\) is commensurable with \(F\). Hence the hypotheses of Lemma~\ref{lem:poly-growth-schreier-commensurable-normal} are satisfied, and therefore \(G/H\) has polynomial growth. So \((G,H)\in \mathbf{SLC}_{\mathrm{subexp}}\) by Lemma~3.3 in \cite{ChatterjiZarka2024v1}.
\end{proof}

\begin{corollary}\label{cor:slc-classification-closed-hyperbolic-3-manifold}
	Let \(M\) be a closed hyperbolic \(3\)-manifold, let \(G=\pi_1(M)\), and let \(H\leq G\) be a finitely generated subgroup. Then $(G,H)\in \mathbf{SLC}_{\mathrm{subexp}}$ if and only if either $H$ is virtually cyclic or $H$ is $s$-normal .
\end{corollary}

\begin{proof}
    If \((G,H)\in \mathbf{SLC}_{\mathrm{subexp}}\) and \(H\) is not virtually cyclic, then Proposition~\ref{prop:slc_core_closed_hyperbolic_3_manifold} yields \(\operatorname{Core}_G(H)\neq\{e\}\). Since \(G=\pi_1(M)\) is torsion-free, \(\operatorname{Core}_G(H)\) is infinite, and hence \(H\) is \(s\)-normal in \(G\).
    
    Conversely, If \(H\) is of infinite index, \(s\)-normal, and not virtually cyclic, then \(H\)
    cannot be quasiconvex. Since \(H\) is finitely generated, tameness together with
    Canary's covering theorem implies that \(H\) is either geometrically finite or a
    virtual fiber subgroup. In the present setting, geometric finiteness is equivalent
    to quasiconvexity by Proposition~4.4.2 of \cite{AschenbrennerFriedlWilton2015}, so
    the former is impossible. Hence \(H\) is a virtual fiber subgroup, and the proof of
    Theorem~\ref{thm:slc-classification-closed-hyperbolic-3-manifold} shows that
    \((G,H)\in \mathbf{SLC}_{\mathrm{subexp}}\). The virtually cyclic case is immediate.
\end{proof}

\medskip

   Although for \(G=\pi_1(M)\) with \(M\) a closed hyperbolic \(3\)-manifold, the conditions that \(H\) be \(s\)-normal in \(G\) and that \(\operatorname{Core}_G(H)\neq \{e\}\) play the same role, these two notions are distinct in general, and the corresponding classification statements need not coincide for other classes of groups.


	
	


\begin{proposition}\label{prop:slc_necessity_kleinian_parabolic}
	Let \(G<\mathrm{Isom}(\mathbb H^3)\) be a non-elementary finitely generated Kleinian group, and let \(H\le G\) be a finitely generated subgroup. If \((G,H)\in \mathbf{SLC}_{\mathrm{subexp}}\), \(H\) is not virtually cyclic, and \(\operatorname{Core}_G(H)=\{e\}\), then \(H\) is a rank-\(2\) parabolic subgroup.
\end{proposition}

\begin{proof}

	Fix a word length \(\ell\) on \(G\), and let \(C_h>0\) and \(w:G\to[1,\infty)\) witness \((G,H)\in \mathbf{SLC}_{\mathrm{subexp}}\). Write
	\[
	W(R):=\sup_{\ell(g)\le R}w(g).
	\]
	Since \(\operatorname{Core}_G(H)=\{e\}\) and \(G\) is infinite, necessarily \([G:H]=\infty\).
	
	For \(x\in G\), set \(K_x:=H\cap xHx^{-1}\). Exactly as in the proof of Proposition~\ref{prop:slc_core_closed_hyperbolic_3_manifold}, one obtains
	\begin{equation}\label{eq:Hx-growth-estimate-kleinian1}
		|H\cap B_G(R)|
		\le
		C_h^2\,W(R+\ell(x))^2\,|K_x\cap B_G(2R)|
		\qquad (x\in G,\ R\ge0).
	\end{equation}
	
	Assume first that \(H\) is non-elementary. Then \(H\) contains a free subgroup
	\(F=\langle u,v\rangle\cong F_2\). Let $M:=\max\{\ell(u),\ell(v)\}.$ Then every reduced word in \(u^{\pm1},v^{\pm1}\) of length at most \(n\) has
	ambient \(G\)-length at most \(Mn\). Hence $|H\cap B_G(Mn)|\ge |F\cap B_F(n)|.$
	Since \(F\cong F_2\), the right-hand side grows exponentially in \(n\). Therefore
	there exist constants \(a>1\) and \(c>0\) such that
	\begin{equation}\label{eq:H-exp-growth-kleinian1}
		|H\cap B_G(R)|\ge c\,a^R
		\qquad (R\gg1).
	\end{equation}
	
	The same limit-set argument as in Proposition~\ref{prop:slc_core_closed_hyperbolic_3_manifold}, with \(\partial G\) replaced by \(\Lambda(G)\), together with \eqref{eq:Hx-growth-estimate-kleinian1} and \eqref{eq:H-exp-growth-kleinian1}, shows that \(\Lambda(H)=\Lambda(G)\).

	Let \(G^+:=G\cap \mathrm{Isom}^+(\mathbb H^3)\), and choose a torsion-free finite-index subgroup \(G_1\le G^+\). Set \(H_1:=H\cap G_1\). Then \(H_1\) is finitely generated, non-elementary, and of infinite index in \(G_1\), and
	\[
	\Lambda(H_1)=\Lambda(H)=\Lambda(G)=\Lambda(G_1).
	\]
	By Theorem~1.5 of \cite{yang2010}, there exist finite-index subgroups \(G_0\le G_1\) and \(H_0\le H_1\) such that
	\[
	H_0\triangleleft G_0,
	\qquad
	G_0/H_0\cong \mathbb Z.
	\]
	
	Set \(N:=\operatorname{Core}_G(G_0)\) and \(F:=H_0\cap N\). Then \(N\triangleleft G\) has finite index, and \(F\triangleleft N\). Since \(N\) has finite index in \(G_0\), the subgroup \(F\) has finite index in \(H_0\), and in particular \(F\neq\{e\}\).
	
	Choose a finite transversal \(T\subset G\) for \(G/N\), with \(1\in T\), and define
	\[
	J:=\bigcap_{t\in T} tFt^{-1}.
	\]
	Since each \(tFt^{-1}\) is a nontrivial normal subgroup of the torsion-free non-elementary Kleinian group \(N\), Lemma~\ref{lem:normal-intersection-kleinian} gives \(J\neq\{e\}\). Moreover, \(J\triangleleft G\), and because \(1\in T\) we have \(J\le F\le H\). Hence $J\le \operatorname{Core}_G(H),$	contradicting the assumption \(\operatorname{Core}_G(H)=\{e\}\).
	
	Therefore \(H\) cannot be non-elementary. Thus \(H\) is elementary. Since \(H\) is not virtually cyclic, the standard classification of finitely generated elementary Kleinian groups implies that \(H\) must be a rank-\(2\) parabolic subgroup.
\end{proof}

\begin{theorem}\label{th:H3}
	Let \(G<\mathrm{Isom}(\mathbb H^3)\) be a non-elementary finitely generated discrete subgroup, and let \(H\le G\) be a finitely generated subgroup. Then \((G,H)\in \mathbf{SLC}_{\mathrm{subexp}}\) if and only if \(H\) is virtually cyclic, or \(H\) is a rank-\(2\) parabolic subgroup, or $H$ is $s$-normal.
\end{theorem}

\begin{proof}
	For the necessity, assume that \((G,H)\in \mathbf{SLC}_{\mathrm{subexp}}\), and suppose that \(H\) is neither virtually cyclic nor a rank-\(2\) parabolic subgroup. Then \(H\) is non-elementary. If \(H\) were not \(s\)-normal, then the argument in the proof of Proposition~\ref{prop:slc_necessity_kleinian_parabolic} would give a contradiction between the \(\mathbf{SLC}_{\mathrm{subexp}}\) estimate and the exponential growth of \(H\). Hence \(H\) is \(s\)-normal.
	
	For the sufficiency, the virtually cyclic case is clear. Assume first that \(H\) is a rank-\(2\) parabolic subgroup. Then \(H\) is virtually \(\mathbb Z^2\), hence amenable. It therefore suffices to show that \(G\) has property \(\mathrm{RD}\).
	
	Set $G^{+}:=G\cap \mathrm{Isom}^{+}(\mathbb H^3).$ Then \(G^{+}\) has index at most \(2\) in \(G\). By Selberg's lemma \cite{alperin1987,selberg1960}, \(G^{+}\)
	contains a torsion-free subgroup \(G_0\) of finite index. Since \(G_0\) is a
	finite-index subgroup of the finitely generated non-elementary Kleinian group
	\(G^{+}\), it is again finitely generated and non-elementary. Let $M:=\mathbb H^3/G_0.$
	Then \(M\) is an orientable complete hyperbolic \(3\)-manifold with finitely generated
	fundamental group. By the Tameness Theorem \cite{calegari2006}, \(M\) is homeomorphic to the interior of a compact \(3\)-manifold \(N\). Since \(M\) is orientable and irreducible, we may choose \(N\) to be orientable and irreducible. Moreover, \(N\) is homotopy equivalent to its interior, and hence $\pi_1(N)\cong \pi_1(M)\cong G_0.$
	
	We distinguish two cases.
	
	If \(\partial N=\varnothing\) or every component of \(\partial N\) is a torus,
	then \(N\not\cong S^1\times D^2\) and \(N\not\cong T^2\times I\), since
	otherwise \(\pi_1(N)\cong \mathbb Z\) or \(\mathbb Z^2\), contrary to the fact
	that \(G_0\) is non-elementary. Since \(N\) is irreducible, it follows that
	every torus component of \(\partial N\) is incompressible; indeed, if some torus
	component of \(\partial N\) were compressible, then \(N\cong S^1\times D^2\),
	contrary to the fact that \(\pi_1(N)\cong G_0\) is non-elementary.
     Hence each boundary component gives rise to a rank-\(2\) cusp, and
	therefore \(M\cong \operatorname{int}(N)\) has finite volume. Thus \(G_0\) is a
	lattice in the rank-one Lie group \(\mathrm{Isom}^{+}(\mathbb H^3)\), and hence
	\(G_0\) has property \(\mathrm{RD}\) by \cite{ChatterjiRuane2005}.
	
	Suppose next that \(\partial N\) has a component of genus at least \(2\). Then
	\(M\cong \operatorname{int}(N)\) has infinite volume. Moreover,
	\(N\not\cong T^2\times I\). Hence by
	\cite[Theorem~4.8.5]{AschenbrennerFriedlWilton2015}, there exists a
	finite-volume hyperbolic \(3\)-manifold \(M'\) such that \(\pi_1(N)\cong G_0\)
	embeds into \(\pi_1(M')\) as a geometrically finite subgroup. The group
	\(\pi_1(M')\) is a lattice in \(\mathrm{Isom}^{+}(\mathbb H^3)\), hence has
	property \(\mathrm{RD}\) by \cite{ChatterjiRuane2005}. Since property
	\(\mathrm{RD}\) is inherited by subgroups \cite{Jolissaint1990}, it follows that
	\(G_0\) has property \(\mathrm{RD}\) in this case as well.
	
	Thus \(G_0\) has property \(\mathrm{RD}\) in all cases. Finally, property
	\(\mathrm{RD}\) is stable under passage to finite-index overgroups
	\cite{Jolissaint1990}, so \(G\) has property \(\mathrm{RD}\).

	It remains to consider the case where \(H\) is \(s\)-normal and neither virtually cyclic nor a rank-\(2\) parabolic subgroup. Then \(H\) is non-elementary. Exactly as in the proof of Corollary~\ref{cor:slc-classification-closed-hyperbolic-3-manifold}, the \(s\)-normality of \(H\) implies that \(\Lambda(H)=\Lambda(G)\). Once \(\Lambda(H)=\Lambda(G)\) is known, the remainder of the argument is identical to the converse direction of Theorem~\ref{thm:slc-classification-closed-hyperbolic-3-manifold}: tameness together with Canary's covering theorem reduces the problem to Lemma~\ref{lem:poly-growth-schreier-commensurable-normal}, and hence yields \((G,H)\in \mathbf{SLC}_{\mathrm{subexp}}\).
\end{proof}

\begin{remark}
	In the setting of finitely generated Kleinian groups \(G<\mathrm{Isom}(\mathbb H^3)\), the condition \(\operatorname{Core}_G(H)\neq\{e\}\) is not sufficient to ensure that \((G,H)\in \mathbf{SLC}_{\mathrm{subexp}}\) in general.

	Let \(\Gamma=\pi_1(\Sigma_g)\) (\(g\ge2\)) be realized as a cocompact Fuchsian group preserving a totally geodesic plane \(P\subset \mathbb H^3\), and let \(r\) be the reflection in \(P\). Set \(G:=\langle \Gamma,r\rangle \cong \Gamma\times C_2\). Then \(G\) is a non-elementary finitely generated discrete subgroup of \(\mathrm{Isom}(\mathbb H^3)\).
	
	Let \(K<\Gamma\) be a finitely generated subgroup of infinite index which is not virtually cyclic, and set \(H:=\langle K,r\rangle=K\times C_2\). Since \(r\) is central in \(G\), one has \(\langle r\rangle\le \operatorname{Core}_G(H)\), and hence \(\operatorname{Core}_G(H)\neq\{e\}\).
	
	On the other hand, if \(Z:=\langle r\rangle\triangleleft G\), then \(Z\le H\). Quotient permanence for \(\mathbf{SLC}_{\mathrm{subexp}}\) therefore shows that, if \((G,H)\in \mathbf{SLC}_{\mathrm{subexp}}\), then \((G/Z,H/Z)\in \mathbf{SLC}_{\mathrm{subexp}}\).
	
	But \((G/Z,H/Z)\cong (\Gamma,K)\), and by the surface-group classification, since \(K\) is finitely generated, of infinite index, and not virtually cyclic, one has \((\Gamma,K)\notin \mathbf{SLC}_{\mathrm{subexp}}\). Hence \((G,H)\notin \mathbf{SLC}_{\mathrm{subexp}}\).
\end{remark}

The following proposition shows that the necessity direction of Theorem~\ref{th:H3} remains valid for non-elementary finitely generated discrete subgroups \(G<\mathrm{Isom}(\mathbb H^n)\) (\(n\ge 3\)), with the corresponding parabolic alternative sharpened to: \(H\) is parabolic and virtually \(\mathbb Z^r\) for some \(2\le r\le n-1\).

\begin{proposition}
	Let $G$ be a finitely generated group hyperbolic relative to
	$\mathcal P=\{P_1,\dots,P_m\}$, and let $H\le G$ be a finitely generated subgroup.
	Assume $(G,H)\in \mathbf{SLC}_{\mathrm{subexp}}$.
	Then either $H$ is $s$-normal in $G$, or $H$ is parabolic
	(i.e.\ conjugate into some $P_i$), or $H$ is
	virtually cyclic. Moreover, if each peripheral subgroup $P_i$ is
	virtually nilpotent, then in the parabolic case $H$ is virtually nilpotent.
\end{proposition}

\begin{proof}
	Assume that $H$ is neither $s$-normal nor parabolic nor virtually cyclic.
	By Corollary~\ref{cor:non-s-normal}, $H$ has polynomial growth with respect to the induced length $\ell_G|_H$.
	On the other hand, since $H$ is finitely generated, non-parabolic, and not virtually cyclic in a relatively hyperbolic group, it contains a free subgroup of rank $2$ \cite{Osin2006}; hence $H$ has exponential growth (for any word metric), and therefore also exponential growth with respect to $\ell_G|_H$.
	This is a contradiction.
\end{proof}

\begin{corollary}\label{prop:slc-subexp-snormal}
	Let \(G\) be a non-elementary finitely generated hyperbolic group, let \(\ell\) be a word length on \(G\), and let \(H\le G\). Assume that $(G,H)\in \mathbf{SLC}_{\mathrm{subexp}}$ and $H$ is not virtually cyclic.
	Then for every \(g\in G\) the subgroup \(H\cap gHg^{-1}\) is not virtually cyclic.
	In particular, \(H\) is \(s\)-normal in \(G\).

\end{corollary}

\begin{proof}
	Fix \(g\in G\) and set \(K_g=H\cap gHg^{-1}\). Apply \eqref{eq:Hx-growth-estimate} with \(x=g\). If \(K_g\) is virtually cyclic, then \(K_g\cap B_G(R)\) has linear growth, since virtually cyclic subgroups of hyperbolic groups are quasiconvex and therefore undistorted. Hence \eqref{eq:Hx-growth-estimate} implies that \(H\cap B_G(R)\) has subexponential growth. This is impossible, because \(H\) is non-virtually-cyclic in a hyperbolic group, so it contains a subgroup isomorphic to \(F_2\), which yields exponential growth for \(H\cap B_G(R)\). Therefore \(H\cap gHg^{-1}\) is not virtually cyclic for every \(g\in G\). In particular, \(H\) is \(s\)-normal in \(G\).
\end{proof}

\begin{remark}\label{rem:slc-subexp-snormal-redundant}
	In Proposition~\ref{prop:slc-subexp-snormal}, the assumptions that \(H\) be infinite, finitely generated, or of infinite index are unnecessary. Indeed, every finite subgroup is virtually cyclic, so the condition that \(H\) be non-virtually-cyclic already forces \(H\) to be infinite. Moreover, if \([G:H]<\infty\), then \(H\) is automatically \(s\)-normal.
\end{remark}

\medskip

However, the $s$-normality of $H$ in $G$ does not by itself imply that $(G,H)\in \mathbf{SLC}_{\mathrm{subexp}}$. Indeed, a slight modification of Proposition~\ref{prop:ripsconstruction} yields a counterexample.
Let $Q=\operatorname{BS}(1,2)=\langle a,t \mid tat^{-1}=a^2\rangle$ and $A:=\langle a\rangle<Q$. Then $A$ is not normal in $Q$, since $tAt^{-1}=\langle a^2\rangle\neq A$. By a Rips construction, there exist a non-elementary hyperbolic group $G$ and a surjective homomorphism $\pi:G\twoheadrightarrow Q$ whose kernel $N:=\ker\pi$ is finitely generated. Set $H:=\pi^{-1}(A)$.

Since $A$ is not normal in $Q$, the subgroup $H$ is not normal in $G$. On the other hand, for every $g\in G$ one has $N\subseteq H\cap gHg^{-1}$. Moreover, $N$ must be infinite, for otherwise $Q\cong G/N$ would be hyperbolic, contradicting the fact that $\operatorname{BS}(1,2)$ is not hyperbolic. It follows that $H$ is $s$-normal in $G$.
Finally, the Schreier graph $G/H$ is naturally identified with $Q/A$. Since $A$ is co-amenable in $Q$ and the Schreier graph $Q/A$ has exponential growth, Remark~\ref{rem:subexp-analogue-thm12} implies that $(G,H)\notin \mathbf{SLC}_{\mathrm{subexp}}$.
It is therefore natural to ask under the above assumptions whether $s$-normality can in fact be promoted to commensuratedness.

\begin{lemma}\label{lem:pullback-pair-rd-slc}
	Let \(\pi:G\twoheadrightarrow Q\) be a surjective homomorphism of finitely generated groups, let \(L\le Q\), and set \(H:=\pi^{-1}(L)\).  
	If \((Q,L)\) has pair rapid decay, respectively belongs to \(\mathbf{SLC}_{\mathrm{subexp}}\), then \((G,H)\) has pair rapid decay, respectively belongs to \(\mathbf{SLC}_{\mathrm{subexp}}\).
\end{lemma}

\begin{proof}
	Let \( \pi_{\#}:\mathbb CG\to\mathbb CQ\) be
	\( \pi_{\#}(f)(q):=\sum_{\pi(g)=q}|f(g)|\).
	Since \(\pi^{-1}(qL)=qH\), we have
	\[
	\|f\|_{(2,1),(G,H)}=\| \pi_{\#}(f)\|_{(2,1),(Q,L)}.
	\]
	Also, for \(f,\varphi\in\mathbb CG\),
	\[
	 \pi_{\#}(f*\varphi)\le  \pi_{\#}(f)* \pi_{\#}(\varphi)
	\]
	pointwise on \(Q\), hence $\|f*\varphi\|_{(2,1),(G,H)}
	\le
	\| \pi_{\#}(f)* \pi_{\#}(\varphi)\|_{(2,1),(Q,L)}.$
	
	If \((Q,L)\) has pair rapid decay, apply its estimate to \( \pi_{\#}(f), \pi_{\#}(\varphi)\), and pull back the weight via \(\pi\). If \(w_Q(q)=(1+\ell_Q(q))^s\), set \(w_G:=w_Q\circ\pi\); since \(\ell_Q(\pi(g))\le \ell_G(g)\), \(w_G\) has polynomial growth. This gives pair rapid decay for \((G,H)\).
	
	The \(\mathbf{SLC}_{\mathrm{subexp}}\) case is identical: if
	\(\|F\|_{h,(Q,L)}\le C_h\|F\omega\|_{(2,1),(Q,L)}\),
	take \(w:=\omega\circ\pi\), and note
	\(\sup_{\ell_G(g)\le R}w(g)\le \sup_{\ell_Q(q)\le R}\omega(q)\).
	Thus subexponential control is preserved.
\end{proof}

\begin{proposition}
	\label{prop:snormal-noncommensurated}
	There exist a torsion-free non-elementary finitely generated hyperbolic group \(G\) and a finitely generated subgroup \(H\le G\) such that \((G,H)\) has pair rapid decay, \(H\) is \(s\)-normal in \(G\), and \(H\) is not commensurated in \(G\). In particular, \((G,H)\in \mathbf{SLC}_{\mathrm{subexp}}\).
\end{proposition}

\begin{proof}
	Let \(S\) be a closed, connected, orientable surface of genus \(g\ge 2\), and choose independent pseudo-Anosov classes \(f,g\in\mathrm{Mod}(S)\).
	By Farb--Mosher \cite[Thms.~1.3, 1.4]{FarbMosher2002}, for all sufficiently large \(N\) the subgroup $Q:=\langle a,b\rangle:=\langle f^N,g^N\rangle$
	is a Schottky subgroup of \(\mathrm{Mod}(S)\), hence free of rank \(2\), and the associated extension
	\[
	1\to \Pi:=\pi_1(S)\to G \xrightarrow{\pi} Q\to 1
	\]
	is word-hyperbolic.
	Since \(\Pi\) and \(Q\cong F_2\) are torsion-free, \(G\) is torsion-free. Moreover, \(G\) is non-elementary, since it surjects onto \(Q\cong F_2\) and therefore cannot be virtually cyclic.

	Set \(H:=\pi^{-1}(\langle a\rangle)\).
	The restriction \(\pi|_H\) gives \(1\to \Pi\to H\to \langle a\rangle\cong\mathbb Z\to 1\), split by a lift of \(a\), hence
	\(H\cong \Pi\rtimes \mathbb Z\).
	Thus \(H\) is finitely generated and not virtually cyclic.
	
	For any \(y\in G\), writing \(q=\pi(y)\), one has
	\[
	yHy^{-1}=\pi^{-1}(q\langle a\rangle q^{-1}),\qquad
	H\cap yHy^{-1}=\pi^{-1}\!\bigl(\langle a\rangle\cap q\langle a\rangle q^{-1}\bigr).
	\]
	Therefore \(H\cap yHy^{-1}\supseteq \ker\pi=\Pi\), so \(H\cap yHy^{-1}\) is infinite for all \(y\), i.e. \(H\) is \(s\)-normal.
	
	Now choose \(u\in G\) with \(\pi(u)=b\). In \(Q=F(a,b)\), the cyclic subgroup \(\langle a\rangle\) is malnormal, so
	\(\langle a\rangle\cap b\langle a\rangle b^{-1}=\{1\}\).
	Hence $H\cap uHu^{-1}=\pi^{-1}(1)=\Pi.$
	Since \(H/\Pi\cong\mathbb Z\), we get \([H:H\cap uHu^{-1}]=[H:\Pi]=\infty\), so \(H\) is not commensurated in \(G\).
	
	Finally, \(Q\cong F_2\) has RD and \(\langle a\rangle\) is amenable, so \((Q,\langle a\rangle)\) has pair RD by Corollary~2.14 in \cite{ChatterjiZarka2024v1}.
	Applying Lemma~\ref{lem:pullback-pair-rd-slc} to \(\pi\) and \(H=\pi^{-1}(\langle a\rangle)\), we obtain that \((G,H)\) has pair RD.
	In particular, \((G,H)\in \mathbf{SLC}_{\mathrm{subexp}}\).
\end{proof}

\begin{remark}\label{rem:core-of-good-not-commensurated-example}
	In Proposition~\ref{prop:snormal-noncommensurated}, one has
	\(\operatorname{Core}_G(H)=\Pi\). Indeed,
	\(\operatorname{Core}_G(H)=\pi^{-1}(\operatorname{Core}_Q(\langle a\rangle))\), and
	\(\operatorname{Core}_Q(\langle a\rangle)=\{1\}\) since \(\langle a\rangle\) is a non-normal maximal cyclic subgroup of the free group \(Q\cong F_2\). Thus the example lies in the nontrivial-core regime.
\end{remark}

\medskip

For certain special classes of non-elementary hyperbolic groups, we have already obtained a satisfactory classification of finitely generated subgroups \(H\le G\) for which \((G,H)\in \mathbf{SLC}_{\mathrm{subexp}}\). A natural question is whether such a classification can be extended to arbitrary non-elementary hyperbolic groups.

More modestly, even if such a satisfactory classification is out of reach, one may still try to understand the \emph{mechanism} by which pairs \((G,H)\) satisfying \(\mathbf{SLC}_{\mathrm{subexp}}\) arise. In most examples constructed so far, the pair \((G,H)\) is obtained from a surjection \(\pi:G\twoheadrightarrow Q\) by pulling back a subgroup \(L\le Q\), that is, \(H=\pi^{-1}(L)\). This suggests the following principle: for a non-elementary hyperbolic group \(G\), every finitely generated infinite-index subgroup \(H\le G\) satisfying \(\mathbf{SLC}_{\mathrm{subexp}}\) should either be virtually cyclic, or arise as the pullback of a subgroup from a nontrivial quotient of \(G\).

Since \(G\) is non-elementary hyperbolic, it is Hopfian. Hence for any surjection \(\pi:G\twoheadrightarrow Q\), one has \(\ker\pi\neq \{e\}\) if and only if \(Q\not\cong G\). Moreover, a subgroup \(H\le G\) is the pullback of a subgroup from a nontrivial quotient of \(G\) if and only if \(\operatorname{Core}_G(H)\neq \{e\}\).

Indeed, if \(H=\pi^{-1}(L)\) for some surjection \(\pi:G\twoheadrightarrow Q\) with nontrivial kernel, then \(\ker\pi\le H\). Since \(\ker\pi\triangleleft G\), it follows that \(\ker\pi\le \operatorname{Core}_G(H)\), and hence \(\operatorname{Core}_G(H)\neq \{e\}\). Conversely, if \(\operatorname{Core}_G(H)\neq \{e\}\), let \(N:=\operatorname{Core}_G(H)\). Then \(N\triangleleft G\), \(N\le H\), and for the quotient map \(q:G\twoheadrightarrow G/N\) one has \(H=q^{-1}(H/N)\).

This leads to the following question.

\begin{question}
	Let \(G\) be a non-elementary finitely generated hyperbolic group, and let
	\(H\le G\) be a finitely generated subgroup of infinite index. Suppose that
	\((G,H)\in \mathbf{SLC}_{\mathrm{subexp}}\) and that \(H\) is not virtually cyclic.
	Must one have \(\operatorname{Core}_G(H)\neq \{e\}\)?
\end{question}

\bigskip


\subsection{Subgroup Classification in Irreducible Nonuniform Higher-Rank Lattices}\label{subsec:lattices}

Throughout this section, we let \(n\ge 3\) and write \(G:=\mathrm{SL}_n(\mathbb Z)\).
Our aim is to determine when \((G,H)\) satisfies \(\mathbf{SLC}_{\mathrm{subexp}}\) for a subgroup \(H\le G\).

Although the final statement is naturally formulated for all \(n\ge 3\), the essential obstruction already appears in rank \(3\). We therefore begin with the concrete subgroup \(\mathrm{UT}_3(\mathbb Z)\). This example deserves to be treated separately, since besides being completely explicit, it also appears as a distinguished test case in the literature on related rigidity phenomena; see, for instance, \cite{Oppenheim2023}. Among the possible proofs of the failure of \(\mathbf{SLC}_{\mathrm{subexp}}\) for $\bigl(\mathrm{SL}_3(\mathbb Z),\mathrm{UT}_3(\mathbb Z)\bigr),$
we emphasize the argument that admits a structural extension.

We then pass to the parabolic subgroup
\[
P:=\mathrm{Stab}_{\mathrm{SL}_3(\mathbb Z)}(\langle e_2,e_3\rangle)
=\{g\in \mathrm{SL}_3(\mathbb Z)\mid g_{12}=g_{13}=0\}.
\]
It is included not only as a generalization of the unipotent case, but also as the main intermediate step in the proof of the general result.  Once the failure of \(\mathbf{SLC}_{\mathrm{subexp}}\) is established for $\bigl(\mathrm{SL}_3(\mathbb Z),P\bigr),$ the extension to the general subgroup theorem is conceptually close.
The discussion of these two rank-\(3\) examples leads to Theorem~\ref{thm:SLn-subexp-classification}, where we prove that for every  subgroup $H\le \mathrm{SL}_n(\mathbb Z)$, the pair \(\bigl(\mathrm{SL}_n(\mathbb Z),H\bigr)\) satisfies \(\mathbf{SLC}_{\mathrm{subexp}}\) if and only if \(H\) has finite index in \(\mathrm{SL}_n(\mathbb Z)\).

\begin{example}\label{ex:sl3-ut3-no-slc}
	The pair $\big(SL_3(\mathbb{Z}),\,UT_3(\mathbb{Z})\big)\notin \mathbf{SLC}_{\mathrm{subexp}}$.
\end{example}

\begin{proof}
	Choose a hyperbolic matrix $A\in SL_2(\mathbb{Z})$, e.g.
	\(
	A=\begin{psmallmatrix}2&1\\1&1\end{psmallmatrix}.
	\)
	Set
	\[
	t=\operatorname{diag}(1,A)\in SL_3(\mathbb{Z}),
	\quad
	U=\left\{
	\begin{pmatrix}
		1&0&0\\
		a&1&0\\
		b&0&1
	\end{pmatrix}:a,b\in\mathbb{Z}
	\right\}\cong\mathbb{Z}^2,
	\quad
	K:=\langle U,t\rangle.
	\]
	Writing $u(v)=\begin{psmallmatrix}1&0\\ v&I_2\end{psmallmatrix}$ for $v\in\mathbb{Z}^2$,
	one checks	$t\,u(v)\,t^{-1}=u(Av),$
	hence $K\cong \mathbb{Z}^2\rtimes_A \mathbb{Z}.$
	Therefore \(K\) is solvable, hence amenable. Since $A$ is hyperbolic,
	$\mathbb{Z}^2\rtimes_A\mathbb{Z}$ is not virtually nilpotent; by Milnor--Wolf,
	$K$ has exponential growth.
	
	It remains to show $K\cap UT_3(\mathbb{Z})=\{e\}$.
	Take $g\in K$ and write $g=u(v)t^n$.
	In block form,
	\[
	g=\begin{pmatrix}1&0\\ v&A^n\end{pmatrix},
	\]
	so its characteristic polynomial is	$\chi_g(X)=(X-1)\chi_{A^n}(X).$
	If $g\in UT_3(\mathbb{Z})$, then $g$ is unipotent, so $\chi_g(X)=(X-1)^3$;
	thus $\chi_{A^n}(X)=(X-1)^2$, impossible for $n\neq0$ because $A$ is hyperbolic.
	Hence $n=0$, so $g=u(v)\in U\cap UT_3(\mathbb{Z})$, and this forces $v=0$.
	Therefore \(g=e\), so \(K\cap \mathrm{UT}_3(\mathbb Z)=\{e\}\).
	
	Assume toward contradiction that $\bigl(\mathrm{SL}_3(\mathbb Z),\mathrm{UT}_3(\mathbb Z)\bigr)\in \mathbf{SLC}_{\mathrm{subexp}}$
	with constants \(C_h>0\) and \(w\). Since each left coset of \(\mathrm{UT}_3(\mathbb Z)\) meets \(K\) in at most one point, the zero-extension argument yields
	\[
	\|\lambda_K(f)\|=\|f\|_{h,(K,\{e\})}\le C_h\,\|f\cdot (w|_K)\|_2
	\qquad (f\in \mathbb CK).
	\]
	Let \(\ell\) be the ambient word length on \(\mathrm{SL}_3(\mathbb Z)\), and let \(\ell_K\) be a word length on \(K\). If \(S_K\) is a finite symmetric generating set for \(K\) and \(L:=\max_{s\in S_K}\ell(s)\), then \(\ell(k)\le L\,\ell_K(k)\) for all \(k\in K\), hence $\sup_{\ell_K(k)\le R} w(k)\le W(LR).$
	Thus \(w|_K\) has subexponential growth with respect to \(\ell_K\). Since \(K\) is amenable, Remark~\ref{rem:subexp-analogue-thm12} applied to \((K,\{e\})\) implies that \(K\) has subexponential growth, contradicting the fact that \(K\) has exponential growth.
\end{proof}

The preceding example is based on a transverse subgroup, and is therefore covered by the special case \(K\cap H=\{e\}\). The parabolic setting is different. In fact, for
\[
H=\operatorname{Stab}_{\mathrm{SL}_n(\mathbb Z)}(\langle e_2,\dots,e_n\rangle),
\]
one can show that every finitely generated amenable subgroup of exponential growth meets some conjugate of \(H\) in an infinite subgroup. Thus the transverse condition \(K\cap H=\{e\}\) is too restrictive here, and one is naturally led to consider the relative quotient \(K/(K\cap H)\).

\begin{lemma}\label{lem:restriction-subgroup}
	Assume that \((G,H)\in \mathbf{SLC}_{\mathrm{subexp}}\). Let \(K<G\), and set
	\(L:=K\cap H\). Then the same convolution estimate holds on \((K,L)\), namely
	\[
	\|f\|_{h,(K,L)}
	\le C_h\,\|f\cdot (w|_K)\|_{(2,1),(K,L)}
	\qquad (f\in \mathbb CK),
	\]
	and \(w|_K\) is subexponential with respect to the restricted length \(\ell|_K\).
	In particular, if \(K\) is finitely generated, then $(K,L)\in \mathbf{SLC}_{\mathrm{subexp}}.$

\end{lemma}

\begin{proof}
	Let \(\iota:\mathbb CK\to \mathbb CG\) be extension by zero. For \(\varphi\in \mathbb CK\),
	if \(gH\cap K\neq\varnothing\) and \(k\in gH\cap K\), then $gH\cap K=kH\cap K=k(H\cap K)=kL.$
	Hence 	$\|\iota\varphi\|_{(2,1),(G,H)}=\|\varphi\|_{(2,1),(K,L)}.$

	Thus \(\iota\) is an isometric embedding of \(\ell^{(2,1)}(K,L)\) into \(\ell^{(2,1)}(G,H)\).
	
	Now let \(f\in\mathbb CK\). Since \(K\) is a subgroup, \(f*\varphi\) is supported in \(K\) for every \(\varphi\in\mathbb CK\), and $\iota(f*\varphi)=(\iota f)*(\iota\varphi).$
	Therefore
	\[
	\|f\|_{h,(K,L)}
	\le \|\iota f\|_{h,(G,H)}
	\le C_h\,\|(\iota f)w\|_{(2,1),(G,H)}.
	\]
	Since \((\iota f)w\) is supported in \(K\), the same coset decomposition gives
	\[
	\|(\iota f)w\|_{(2,1),(G,H)}
	=\|f\cdot (w|_K)\|_{(2,1),(K,L)}.
	\]
	This proves the stated inequality.
	
	 If \(K\) is finitely generated and \(\ell_K\) is the word length associated with a finite symmetric generating set \(S_K\), then with $M:=\max_{s\in S_K}\ell(s)$
	one has \(\ell(k)\le M\,\ell_K(k)\) for all \(k\in K\). Hence $\sup_{\ell_K(k)\le R} w(k)\le W(MR),$ so \(w|_K\) is subexponential with respect to \(\ell_K\). Together with the
	convolution estimate proved above, this shows that $(K,L)\in \mathbf{SLC}_{\mathrm{subexp}}.$
\end{proof}

\begin{proposition}\label{prop:relative-obstruction}
	Assume that \(G\) is finitely generated and \(H<G\). Suppose that \(G\) contains a finitely generated amenable subgroup \(K\), and set \(L:=K\cap H\). If the Schreier graph of the left coset space \(K/L\) has exponential growth, then \((G,H)\notin \mathbf{SLC}_{\mathrm{subexp}}\). In particular, the same conclusion holds if \(K\cap H=\{e\}\) and \(K\) has exponential growth.
\end{proposition}

\begin{proof}
	Assume that \((G,H)\in \mathbf{SLC}_{\mathrm{subexp}}\). By Lemma~\ref{lem:restriction-subgroup}, the pair \((K,L)\) also satisfies \(\mathbf{SLC}_{\mathrm{subexp}}\), with respect to the restricted weight \(w|_K\), which is subexponential with respect to any word length on \(K\). Since \(K\) is amenable, \(L\) is co-amenable in \(K\). It then follows from Remark~\ref{rem:subexp-analogue-thm12} that the Schreier graph of the left coset space \(K/L\) has subexponential growth, a contradiction.
\end{proof}

\begin{remark}
	There is also a direct proof that the pair $\bigl(\mathrm{SL}_3(\mathbb Z),\mathrm{UT}_3(\mathbb Z)\bigr)$
	does not satisfy Pair RD. Consider the subgroup $A=\{a(n)=I+nE_{21}:n\in\mathbb Z\}\cong \mathbb Z.$
	As above, one has \(A\cap \mathrm{UT}_3(\mathbb Z)=\{e\}\), so the map \(A\to \mathrm{SL}_3(\mathbb Z)/\mathrm{UT}_3(\mathbb Z)\), \(a\mapsto a\,\mathrm{UT}_3(\mathbb Z)\) is injective. Let \(N_R:=\lfloor e^{R/C}-1\rfloor\), \(S_R:=\{a(n): |n|\le N_R\}\), and \(f_R:=\mathbf 1_{S_R}\).
	Then \(\operatorname{supp}(f_R)\subset B_{\mathrm{SL}_3(\mathbb Z)}(R)\), and $
	\|f_R\|_{(2,1)}=\sqrt{2N_R+1}.$
	On the other hand, the injectivity above gives	$\|f_R\|_h\ge \|\lambda_A(f_R)\|,$
	and since \(A\cong \mathbb Z\) is abelian, one has \(\|\lambda_A(f_R)\|=2N_R+1\). Hence there exist constants \(c_1,c_2>0\) and \(R_0>0\) such that, for all \(R\ge R_0\),
	\[
	\|f_R\|_h\ge c_1 e^{R/C},
	\qquad
	\|f_R\|_{(2,1)}\le c_2 e^{R/(2C)}.
	\]
	This is incompatible with Pair RD.
\end{remark}

\begin{proposition}\label{prop:sl3-parabolic-no-slc-subexp}
	Let \(G:=\mathrm{SL}_3(\mathbb Z)\), \(W_0:=\langle e_2,e_3\rangle\subset \mathbb Z^3\), and $H:=\mathrm{Stab}_G(W_0)=\{g\in G:\ g_{12}=g_{13}=0\}.$
	Then \((G,H)\notin \mathbf{SLC}_{\mathrm{subexp}}\).
\end{proposition}

\begin{proof}
	By Proposition~\ref{prop:relative-obstruction}, it is enough to construct a finitely generated amenable subgroup \(K<G\) such that, with \(L:=K\cap H\), the Schreier graph of \(K/L\) has exponential growth.
	
	Let
	\[
	A:=\begin{pmatrix}2&1\\1&1\end{pmatrix}\in\mathrm{SL}_2(\mathbb Z),\qquad
	t:=\mathrm{diag}(1,A)\in G,\qquad
	u(a,b):=\begin{pmatrix}1&0&0\\ a&1&0\\ b&0&1\end{pmatrix}.
	\]
	Set
	\[
	K_0:=\langle u(1,0),u(0,1),t\rangle
	=\{u(v)t^n:\ v\in\mathbb Z^2,\ n\in\mathbb Z\}
	\cong \mathbb Z^2\rtimes_A\mathbb Z.
	\]
	Since \(A\) is hyperbolic, then \(K_0\) has exponential growth. Moreover, \(K_0\) is solvable, hence amenable. Let
	\[
	g:=\begin{pmatrix}1&1&0\\0&1&0\\0&0&1\end{pmatrix},
	\qquad
	K:=gK_0g^{-1}.
	\]
	Then \(K\) is again finitely generated, amenable, and of exponential growth.
	
	Write \(A^n=\begin{pmatrix}p_n&q_n\\ q_n&r_n\end{pmatrix}\). For \(k=gu(a,b)t^n g^{-1}\in K\), a direct computation gives \(\mathrm{row}_1(k)=(1+a,\ p_n-1-a,\ q_n)\). Hence \(k\in H\) if and only if \(q_n=0\) and \(a=p_n-1\). If \(q_n=0\), then \(A^n\) is diagonal. Since \(A^n\in \mathrm{SL}_2(\mathbb Z)\), this forces \(A^n=\pm I\), and therefore \(n=0\) because \(A\) is hyperbolic. It follows that \(a=0\), while \(b\) is arbitrary. Thus
	\[
	L:=K\cap H=g\langle u(0,1)\rangle g^{-1}\cong \mathbb Z.
	\]
	
	Let \(L_0:=\langle u(0,1)\rangle<K_0\). For \(m\ge1\) and \(\varepsilon=(\varepsilon_0,\dots,\varepsilon_{m-1})\in\{0,1\}^m\), set \(v_\varepsilon:=\sum_{i=0}^{m-1}\varepsilon_i A^i e_1\), \(k_\varepsilon:=u(v_\varepsilon)t^m\in K_0\), and \(x_\varepsilon:=gk_\varepsilon g^{-1}\in K\). Since \(u(v)t\cdot u(w)t=u(v+Aw)t^2\), one has \(k_\varepsilon=\prod_{i=0}^{m-1}\bigl(u(\varepsilon_i e_1)t\bigr)\). Hence, with respect to the finite generating set
	\[
	S:=\{g u(1,0)g^{-1},\, g u(0,1)g^{-1},\, g t g^{-1}\}^{\pm1}
	\]
	of \(K\), one has \(\ell_S(x_\varepsilon)\le 2m\).

	We claim that the left cosets \(x_\varepsilon L\) are pairwise distinct. Equivalently, the left cosets \(k_\varepsilon L_0\) are pairwise distinct, and
	$k_\varepsilon L_0=k_{\varepsilon'}L_0$ if and only if 
	$v_\varepsilon-v_{\varepsilon'}\in \mathbb Z\,A^m e_2.$
	Set \(u_m:=A^m e_2\). Since \(A^m\in \mathrm{SL}_2(\mathbb Z)\), the vector \(u_m\) is primitive, so \(v\in \mathbb Z\,u_m\) if and only if \(\det(u_m,v)=0\). It therefore suffices to show that \(\det(u_m,v_\varepsilon-v_{\varepsilon'})\neq0\) for \(\varepsilon\neq\varepsilon'\).
	
	Write \(A^j e_2=(x_j,y_j)\) for \(j\ge1\). Then \((x_{j+1},y_{j+1})=(2x_j+y_j,\ x_j+y_j)\) and \((x_1,y_1)=(1,1)\). Hence \(x_j\ge y_j>0\) for all \(j\), and \(x_j>y_j\) for \(j\ge2\). In particular, \(y_{j+1}=x_j+y_j>2y_j\) for \(j\ge2\). Since \(y_2=2>y_1\), it follows inductively that
	\[
	y_j>\sum_{k<j}y_k\qquad (j\ge2).
	\]
	
	On the other hand, since \(A\) is symmetric and \(\det A=1\), one has \(\det(u_m,A^i e_1)=\det(A^{m-i}e_2,e_1)=-y_{m-i}\). Therefore
	\[
	\det(u_m,v_\varepsilon-v_{\varepsilon'})
	=-\sum_{j=1}^m \sigma_j y_j
	\]
	for suitable \(\sigma_j\in\{-1,0,1\}\), not all zero if \(\varepsilon\neq\varepsilon'\). If \(j_0\) is maximal with \(\sigma_{j_0}\neq0\), then
	\[
	\left|\sum_{j=1}^m \sigma_j y_j\right|
	\ge y_{j_0}-\sum_{j<j_0}y_j>0.
	\]
	Thus \(\det(u_m,v_\varepsilon-v_{\varepsilon'})\neq0\), and the cosets \(x_\varepsilon L\) are distinct.
	
	It follows that the ball of radius \(2m\) in the Schreier graph of \(K/L\) contains at least \(2^m\) vertices. Hence this Schreier graph has exponential growth. Proposition~\ref{prop:relative-obstruction} now gives \((G,H)\notin \mathbf{SLC}_{\mathrm{subexp}}\).
\end{proof}

\begin{lemma}\label{lem:full-rank-subgroup-Zd}
	Let \(d\ge 1\), and let \(L\le \mathbb Z^d\). If
	\(\operatorname{rank}_{\mathbb Z}(L)=d\), then \([\mathbb Z^d:L]<\infty\).
	Moreover, if \(m=[\mathbb Z^d:L]\), then $m\mathbb Z^d\subseteq L.$
\end{lemma}

\begin{proof}
	Since \(\operatorname{rank}_{\mathbb Z}(L)=d\), the quotient
	\(\mathbb Z^d/L\) is a finitely generated abelian group of rank
	\(d-d=0\), hence finite. Thus \(m:=[\mathbb Z^d:L]<\infty\).
	
	Let \(v\in \mathbb Z^d\), and let \(\overline{v}\) denote its class in
	\(\mathbb Z^d/L\). Since \(|\mathbb Z^d/L|=m\), the order of
	\(\overline{v}\) divides \(m\). Therefore \(m\overline{v}=0\), i.e.
	\(mv\in L\). Since \(v\) was arbitrary, we obtain
	\(m\mathbb Z^d\subseteq L\).
\end{proof}

\begin{lemma}\label{lem:transitivity-on-lines}
	The group \(\mathrm{SL}_n(\mathbb Z)\) acts transitively on the set of
	primitive rank-one direct summands of \(\mathbb Z^n\). In particular, if
	\(P_i\le \mathrm{SL}_n(\mathbb Z)\) denotes the stabilizer of
	\(\mathbb Z e_i\), then \(P_1,\dots,P_n\) are all conjugate in
	\(\mathrm{SL}_n(\mathbb Z)\).
\end{lemma}

\begin{proof}
    If $n=1$, the statement is trivial. Assume $n\ge 2$. Let $L\le \mathbb Z^n$ be a rank-one direct summand. Then	\(L=\mathbb Z v\) for some primitive vector \(v\in \mathbb Z^n\). Since \(v\) is primitive, it can be extended to a \(\mathbb Z\)-basis $v,v_2,\dots,v_n$
	of \(\mathbb Z^n\). Let \(g\in \mathrm{GL}_n(\mathbb Z)\) be the matrix
	whose columns are \(v,v_2,\dots,v_n\). Then \(g e_1=v\), so
	\(g(\mathbb Z e_1)=L\). If \(\det(g)=-1\), replace \(v_2\) by \(-v_2\),
	this changes the determinant by a factor of \(-1\) and does not change
	the image of \(\mathbb Z e_1\). Hence we may assume that \(g\in
	\mathrm{SL}_n(\mathbb Z)\). This proves transitivity.
	
	The final assertion follows by taking \(L=\mathbb Z e_i\).
\end{proof}

\begin{theorem}[\cite{Meiri2017,Tits1976,bass1967}]\label{thm:congruence-inclusion}

	Let \(n\ge 3\), and for \(N\ge 1\) define \(u_{ij}(r):=I_n+rE_{ij}\) \((r\in \mathbb Z,\ i\neq j)\), and
	\[
	\Lambda_n(N):=\bigl\langle u_{ij}(N)\mid 1\le i\neq j\le n\bigr\rangle
	\le \mathrm{SL}_n(\mathbb Z).
	\]
	Then \(\Gamma_n(N^2)\subseteq \Lambda_n(N)\), where \(\Gamma_n(m):=\ker\!\bigl(\mathrm{SL}_n(\mathbb Z)\to \mathrm{SL}_n(\mathbb Z/m\mathbb Z)\bigr)\) is the principal congruence subgroup of level \(m\). In particular, \(\Lambda_n(N)\) has finite index in \(\mathrm{SL}_n(\mathbb Z)\).
	
\end{theorem}

\begin{theorem}\label{thm:line-parabolic-rank-criterion}
	Let \(G=\mathrm{SL}_n(\mathbb Z)\) with \(n\ge 3\), and let \(P<G\) be a
	maximal parabolic subgroup stabilizing a primitive rank-one direct
	summand of \(\mathbb Z^n\). Let \(U_P\) denote its unipotent radical.
	Then \(U_P\cong \mathbb Z^{n-1}\), and for every subgroup \(H\le G\) the
	following are equivalent:
	\begin{enumerate}
		\item \([G:H]<\infty\).
		\item For every conjugate \(P'=gPg^{-1}\) of \(P\) in \(G\),
		if \(U_{P'}:=gU_Pg^{-1}\) denotes the unipotent radical of \(P'\), then $\operatorname{rank}_{\mathbb Z}\bigl(H\cap U_{P'}\bigr)=n-1.$
	\end{enumerate}
   Equivalently, \([G:H]=\infty\) if and only if there exists a conjugate
   \(P'=gPg^{-1}\) such that
   \(\operatorname{rank}_{\mathbb Z}\bigl(H\cap U_{P'}\bigr)\le n-2\).
\end{theorem}

\begin{proof}
	Let \(P_1:=\operatorname{Stab}_G(\mathbb Z e_1)\).
	By Lemma~\ref{lem:transitivity-on-lines}, \(P\) is conjugate to \(P_1\).
	Choose \(g_0\in G\) such that \(P=g_0P_1g_0^{-1}\).
	
	Let \(U_1:=U_{P_1}\). Explicitly,
	\[
	U_1=
	\left\{
	I_n+\sum_{j\neq 1} a_j E_{1j}\;:\; a_j\in \mathbb Z
	\right\}
	=
	\bigl\langle u_{1j}(1)\mid j\neq 1\bigr\rangle.
	\]
	Since the matrices \(u_{1j}(1)\) (\(j\neq 1\)) commute pairwise, the map
	\[
	\phi_1:\mathbb Z^{n-1}\longrightarrow U_1,\qquad
	(a_j)_{j\neq 1}\longmapsto \prod_{j\neq 1} u_{1j}(a_j),
	\]
	is an isomorphism. Hence \(U_1\cong \mathbb Z^{n-1}\). Therefore
	\(U_P=g_0U_1g_0^{-1}\cong \mathbb Z^{n-1}\). Consequently, for every
	conjugate \(P'=gPg^{-1}\) of \(P\), \(U_{P'}=gU_Pg^{-1}\cong
	\mathbb Z^{n-1}\).
	
	We first prove \((1)\Rightarrow(2)\). Let \(P'=gPg^{-1}\) be any
	conjugate of \(P\), and let \(U_{P'}\) be its unipotent radical. The
	natural map
	\[
	U_{P'}/(U_{P'}\cap H)\longrightarrow G/H,\qquad
	u(U_{P'}\cap H)\longmapsto uH,
	\]
	is injective. Therefore
	\([U_{P'}:U_{P'}\cap H]\le [G:H]<\infty\).
	Since \(U_{P'}\cong \mathbb Z^{n-1}\), every finite-index subgroup of
	\(U_{P'}\) has \(\mathbb Z\)-rank \(n-1\). Hence
	\(\operatorname{rank}_{\mathbb Z}(H\cap U_{P'})=n-1\)
	for every conjugate \(P'\) of \(P\).
	
	We now prove \((2)\Rightarrow(1)\). Since the set of conjugates of \(P\)
	depends only on the conjugacy class of \(P\), we may replace \(P\) by a
	conjugate and assume that \(P=P_1:=\operatorname{Stab}_G(\mathbb Z e_1)\).
	For \(1\le i\le n\), let
	\(P_i:=\operatorname{Stab}_G(\mathbb Z e_i)\).
	By Lemma~\ref{lem:transitivity-on-lines}, the subgroups \(P_1,\dots,P_n\)
	are all conjugate in \(G\), hence each \(P_i\) is a conjugate of \(P\).
	
	For each \(i\), let \(U_i:=U_{P_i}\). Then
	\[
	U_i=
	\left\{
	I_n+\sum_{j\neq i} a_j E_{ij}\;:\; a_j\in \mathbb Z
	\right\}
	=
	\bigl\langle u_{ij}(1)\mid j\neq i\bigr\rangle.
	\]
	Since the matrices \(u_{ij}(1)\) (\(j\neq i\)) commute pairwise, the map
	\[
	\phi_i:\mathbb Z^{n-1}\longrightarrow U_i,\qquad
	(a_j)_{j\neq i}\longmapsto \prod_{j\neq i} u_{ij}(a_j),
	\]
	is an isomorphism. In particular, \(U_i\cong \mathbb Z^{n-1}\).
	
	By hypothesis, \(L_i:=H\cap U_i\) has \(\mathbb Z\)-rank \(n-1\) for
	every \(i=1,\dots,n\). By Lemma~\ref{lem:full-rank-subgroup-Zd}, each
	\(L_i\) has finite index in \(U_i\). Write \(d_i:=[U_i:L_i]\).
	Applying Lemma~\ref{lem:full-rank-subgroup-Zd} to
	\(\phi_i^{-1}(L_i)\le \mathbb Z^{n-1}\), we obtain
	\(d_i\mathbb Z^{n-1}\subseteq \phi_i^{-1}(L_i)\).
	Equivalently, \(u_{ij}(d_i)\in H\) for every \(j\neq i\).
	
	Now set \(N:=\operatorname{lcm}(d_1,\dots,d_n)\).
	Since \(u_{ij}(r)^m=u_{ij}(mr)\) for all \(r,m\in\mathbb Z\), it follows
	that \(u_{ij}(N)\in H\) for all \(i\neq j\). Therefore
	\[
	\Lambda_n(N):=\bigl\langle u_{ij}(N)\mid 1\le i\neq j\le n\bigr\rangle
	\subseteq H.
	\]
	
	By Theorem~\ref{thm:congruence-inclusion},
	\(\Gamma_n(N^2)\subseteq \Lambda_n(N)\subseteq H\).
	Since \(\Gamma_n(N^2)\) has finite index in \(G\), we conclude that
	\([G:H]\le [G:\Gamma_n(N^2)]<\infty\).
	This proves \((2)\Rightarrow(1)\).
	
	The final reformulation is just the contrapositive of the equivalence of
	\((1)\) and \((2)\).
\end{proof}

   In rank \(3\), one can prove a stronger structural statement for finitely generated amenable subgroups of exponential growth, namely that they virtually contain a subgroup \(U\cong \mathbb Z^2\) consisting entirely of transvections. We record this refinement separately in Appendix~\ref{app:rank3}
   , since it will not be used in the general argument.

    For the proof of Theorem~\ref{thm:SLn-no-subexponential-weight}, the only parabolic input needed is Theorem~\ref{thm:line-parabolic-rank-criterion}, together with the following rank-two reduction and Lemma~\ref{le:rank-exponential}. The latter is merely a generalization of one step in the proof of Proposition~\ref{prop:sl3-parabolic-no-slc-subexp}.

\begin{lemma}\label{lem:rank-two-reduction}
	Let \(U\cong \mathbb Z^d\) with \(d\ge 2\), and let \(N\le U\) be a subgroup of
	rank at most \(d-1\). Then there exists a primitive rank-two direct summand
	\(W\le U\) such that
	\[
	\operatorname{rank}_{\mathbb Z}(N\cap W)\le 1.
	\]
\end{lemma}

\begin{proof}
	Since \(\operatorname{rank}_{\mathbb Z}(N)\le d-1\), the \(\mathbb Q\)-subspace
	\(N_{\mathbb Q}:=N\otimes_{\mathbb Z}\mathbb Q\) is a proper subspace of
	\(U_{\mathbb Q}:=U\otimes_{\mathbb Z}\mathbb Q\). Choose
	\(y\in U\setminus N_{\mathbb Q}\), and write \(y=mx_1\) with \(m\ge 1\) and
	\(x_1\in U\) primitive. Since \(N_{\mathbb Q}\) is a \(\mathbb Q\)-subspace,
	\(x_1\notin N_{\mathbb Q}\). Extend \(x_1\) to a \(\mathbb Z\)-basis $x_1,x_2,\dots,x_d$
	of \(U\). Let $W:=\langle x_1,x_2\rangle \le U.$
	Then \(W\) is a primitive rank-two direct summand of \(U\).
	
	If \(\operatorname{rank}_{\mathbb Z}(N\cap W)=2\), then \(N\cap W\) has full
	rank in the rank-two free abelian group \(W\), hence finite index in \(W\).
	Therefore $(N\cap W)\otimes_{\mathbb Z}\mathbb Q = W\otimes_{\mathbb Z}\mathbb Q.$
	Since \(N\cap W\le N\), it follows that $W\otimes_{\mathbb Z}\mathbb Q \subseteq N_{\mathbb Q},$
	and in particular \(x_1\in N_{\mathbb Q}\), contradicting the choice of
	\(x_1\). Hence $\operatorname{rank}_{\mathbb Z}(N\cap W)\le 1.$
\end{proof}

\begin{lemma}\label{le:rank-exponential}
	Let \(V\cong \mathbb Z^2\), let \(A=\begin{pmatrix}2&1\\1&1\end{pmatrix}\),
	and let $K=V\rtimes_A \langle t\rangle$, where \(tvt^{-1}=Av\) for all \(v\in V\). Let \(L\le K\). If $\operatorname{rank}_{\mathbb Z}(L\cap V)\le 1,$
	then the Schreier graph of \(K/L\) has exponential growth.
\end{lemma}

\begin{proof}
	Write the group law on \(V\) additively. Thus every element of $K=V\rtimes_A \langle t\rangle$
	can be written uniquely as \(vt^m\) with \(v\in V\) and \(m\in \mathbb Z\), and
	$(vt^m)(wt^n)=(v+A^m w)t^{m+n}.$
	Set $M:=L\cap V.$
	Since \(V\cong \mathbb Z^2\) and \(\operatorname{rank}_{\mathbb Z}(M)\le 1\),
	there exists a primitive vector \(\xi\in V\) such that $M\subseteq \mathbb Z\xi.$
	Extend \(\xi\) to a \(\mathbb Z\)-basis \((\xi,\eta)\) of \(V\). Let $S:=\{\xi,\eta,t\}^{\pm1}.$
	This is a finite generating set of \(K\). It therefore suffices to prove that
	the Schreier graph \(\operatorname{Sch}(K/L,S)\) has exponential growth.
	
	Let $\lambda=\frac{3+\sqrt5}{2}>1$
	be the expanding eigenvalue of \(A\), and let \(v_+,v_-\in V\otimes_\mathbb Z
	\mathbb R\) be eigenvectors of \(A\) with eigenvalues \(\lambda\) and
	\(\lambda^{-1}\), respectively. Since \(A\in \mathrm{SL}_2(\mathbb Z)\) has
	irrational eigenvalues, neither eigenspace contains a nonzero vector of the
	lattice \(V\). Hence neither \(\xi\) nor \(\eta\) lies in an eigenspace. Write
	$\xi=\alpha_+v_+ + \alpha_-v_-,$ and  $\eta=\beta_+v_+ + \beta_-v_-,$
	with \(\alpha_\pm,\beta_\pm\neq 0\).
	
	Let \(\det\) denote the determinant form on \(V\) with respect to the basis
	\((\xi,\eta)\), so that \(\det(\xi,\eta)=1\). For \(j\ge 0\), define
	$a_j:=\det(A^j\xi,\eta).$
	Then
	\[
	a_j
	=
	\det(\alpha_+\lambda^j v_+ + \alpha_-\lambda^{-j}v_-,\,\eta)
	=
	c_+\lambda^j + c_-\lambda^{-j},
	\]
	where $c_+ = \alpha_+\det(v_+,\eta)\neq 0.$
	Hence \(a_j\sim c_+\lambda^j\) as \(j\to\infty\), and therefore
	\[
	\frac{\sum_{i=0}^{j-1}|a_i|}{|a_j|}
	\longrightarrow
	\frac{1}{\lambda-1}
	<1.
	\]
	Choose \(J\ge 1\) such that
	\begin{equation}\label{eq:aj}
	 	|a_j|>\sum_{i=0}^{j-1}|a_i|
	 \qquad\text{for all }j\ge J+1.
	\end{equation}
	
	Now fix \(m\ge 1\). For each
	\(\varepsilon=(\varepsilon_0,\dots,\varepsilon_{m-1})\in\{0,1\}^m\), set
	$v_\varepsilon:=\sum_{r=0}^{m-1}\varepsilon_r A^r\eta\in V$ and $
	k_\varepsilon:=v_\varepsilon t^{J+m}\in K.$
	Since $(vt)(wt)=(v+Aw)t^2$,
	we have $k_\varepsilon
	=
	\Bigl(\prod_{r=0}^{m-1} (\varepsilon_r\eta)\, t\Bigr)t^J,$
	where \(0\eta:=0\in V\). Therefore
	\begin{equation}\label{eq:l_S}
	    \ell_S(k_\varepsilon)\le 2m+J\le (J+2)m.
	\end{equation}

	We claim that the left cosets \(k_\varepsilon L\) are pairwise distinct.
	Suppose $k_\varepsilon L=k_{\varepsilon'}L$
	for some \(\varepsilon,\varepsilon'\in\{0,1\}^m\). Then
	\(k_{\varepsilon'}^{-1}k_\varepsilon\in L\). Writing \(N:=J+m\), we compute
	$(vt^N)^{-1}=-A^{-N}v\, t^{-N},$
	hence $k_{\varepsilon'}^{-1}k_\varepsilon
	=
	A^{-N}(v_\varepsilon-v_{\varepsilon'})\in V.$
	Since this element lies in \(L\cap V=M\subseteq \mathbb Z\xi\), it follows that
	$v_\varepsilon-v_{\varepsilon'}\in \mathbb Z\,A^N\xi.$
	Therefore $0=\det(A^N\xi,\,v_\varepsilon-v_{\varepsilon'}).$
	Using \(\det(Au,Av)=\det(u,v)\) (because \(\det A=1\)), we obtain
	\[
	0
	=
	\sum_{r=0}^{m-1}(\varepsilon_r-\varepsilon_r')
	\det(A^N\xi,\,A^r\eta)
	=
	\sum_{r=0}^{m-1}(\varepsilon_r-\varepsilon_r')\det(A^{N-r}\xi,\eta).
	\]
	Equivalently, $0=\sum_{j=J+1}^{J+m}\sigma_j a_j$
	for suitable \(\sigma_j\in\{-1,0,1\}\), not all zero if
	\(\varepsilon\neq\varepsilon'\). Let \(j_0\) be maximal with
	\(\sigma_{j_0}\neq 0\). Then by ~\eqref{eq:aj},
	\[
	\left|\sum_{j=J+1}^{J+m}\sigma_j a_j\right|
	\ge
	|a_{j_0}|-\sum_{i=0}^{j_0-1}|a_i|
	>0,
	\]
	a contradiction. Hence \(\varepsilon=\varepsilon'\), so the cosets
	\(k_\varepsilon L\) are pairwise distinct.
	
	By~\eqref{eq:l_S}, all these \(2^m\) distinct vertices lie in the ball of radius
	\((J+2)m\) around the base vertex \(L\) in \(\operatorname{Sch}(K/L,S)\). Thus
	\[
	\bigl|B_{\operatorname{Sch}(K/L,S)}((J+2)m)\bigr|\ge 2^m
	\qquad (m\ge 1).
	\]
	Therefore \(\operatorname{Sch}(K/L,S)\) has exponential growth, and hence so
	does the Schreier graph of \(K/L\).
\end{proof}

\begin{theorem}\label{thm:SLn-no-subexponential-weight}
	Let \(n\ge 3\), let \(G=\mathrm{SL}_n(\mathbb Z)\), and let \(H\le G\) be a
	subgroup of infinite index. Then \((G,H)\notin \mathbf{SLC}_{\mathrm{subexp}}\).
\end{theorem}

\begin{proof}
	Let \(P<G\) be the maximal parabolic subgroup stabilizing the primitive line
	\(\mathbb Z e_1\), and let \(U_P\) be its unipotent radical. Then
	\(U_P\cong \mathbb Z^{n-1}\). By
	Theorem~\ref{thm:line-parabolic-rank-criterion}, since \([G:H]=\infty\), there
	exists a conjugate \(P'=gPg^{-1}\) such that
	\[
	\operatorname{rank}_{\mathbb Z}(H\cap U_{P'})\le n-2.
	\]
	Set \(U:=U_{P'}\) and \(N:=H\cap U\).
	Then \(U\cong \mathbb Z^{n-1}\) and \(\operatorname{rank}_{\mathbb Z}(N)\le n-2\).
	
	By Lemma~\ref{lem:rank-two-reduction}, there exists a primitive rank-two direct
	summand \(W\le U\) such that
	\[
	\operatorname{rank}_{\mathbb Z}(N\cap W)\le 1.
	\]
	
	Choose \(g_0\in G\) such that \(g_0P'g_0^{-1}=P_1:=\operatorname{Stab}_G(\mathbb Z e_1)\).
	Then
	\[
	g_0Ug_0^{-1}=U_1=
	\left\{u(x):=I_n+\sum_{j=2}^n x_jE_{1j}\;\middle|\; x=(x_2,\dots,x_n)\in \mathbb Z^{n-1}\right\}.
	\]
	Under the identification \(u(x)\leftrightarrow x\), the subgroup
	\(g_0Wg_0^{-1}\) corresponds to a primitive rank-two direct summand
	\(W_0\le \mathbb Z^{n-1}\). Choose \(C\in \mathrm{SL}_{n-1}(\mathbb Z)\) such
	that \(C(W_0)\) is generated by the first two coordinate axes. Via the composed
	identification
	\[
	U \xrightarrow{\operatorname{Ad}(g_0)} U_1 \xrightarrow{\sim} \mathbb Z^{n-1}
	\xrightarrow{C} \mathbb Z^{n-1},
	\]
	choose a basis \(x_1,\dots,x_{n-1}\) of \(U\) such that \(W=\langle x_1,x_2\rangle\).
	
	For \(B\in \mathrm{SL}_{n-1}(\mathbb Z)\), let $t_B:=\operatorname{diag}(1,B^{-1})\in P_1.$
	Then \(t_Bu(x)t_B^{-1}=u(xB)\) for \(x\in \mathbb Z^{n-1}\). Taking
	\[
	B=C^{-1}\operatorname{diag}(A,I_{n-3})C,
	\qquad
	A=\begin{pmatrix}2&1\\1&1\end{pmatrix},
	\]
	with the obvious interpretation \(B=C^{-1}AC\) when \(n=3\), we obtain
	\[
	t:=g_0^{-1}t_Bg_0\in P'
	\]
	preserving \(W\) and acting on \(W\) by \(A\) with respect to the basis
	\(x_1,x_2\).
	
	Set
	\[
	K:=\langle W,t\rangle \le G,
	\qquad
	L:=K\cap H,
	\qquad
	M:=L\cap W=H\cap W.
	\]
	Since \(tWt^{-1}=W\), the subgroup \(W\) is normal in \(K\). Moreover,
	\(\langle t\rangle\cap U=\{e\}\), because \(t_B\) lies in the Levi subgroup of
	\(P_1\), which intersects \(U_1\) trivially. As \(W\le U\), it follows that $
	\langle t\rangle\cap W=\{e\}.$
	Therefore
	\[
	K=W\rtimes \langle t\rangle \cong \mathbb Z^2\rtimes_A \mathbb Z,
	\qquad
	\operatorname{rank}_{\mathbb Z}(M)\le 1.
	\]
	In particular, \(K\) is polycyclic, hence amenable.
	
	Since \(K=W\rtimes\langle t\rangle \cong \mathbb Z^2\rtimes_A\mathbb Z\) and
	\(\operatorname{rank}_{\mathbb Z}(L\cap W)\le 1\), Lemma~\ref{le:rank-exponential}
	shows that the Schreier graph of \(K/L\) has exponential growth.

	Assume now that \((G,H)\in \mathbf{SLC}_{\mathrm{subexp}}\), say
	\[
	\|f\|_{h,(G,H)} \le C_h\, \|fw\|_{(2,1),(G,H)}
	\qquad (f\in \mathbb C G),
	\]
	where $W(t):=\sup_{\ell(g)\le t} w(g)$
	has subexponential growth. By Lemma~\ref{lem:restriction-subgroup} $(K,L)\in \mathbf{SLC}_{\mathrm{subexp}}.$ Since \(K\) is amenable,
	\(L\) is co-amenable in \(K\). Remark~\ref{rem:subexp-analogue-thm12} then
	implies that the Schreier graph of \(K/L\) has subexponential growth, a
	contradiction. Therefore \((G,H)\notin \mathbf{SLC}_{\mathrm{subexp}}\).
\end{proof}

\begin{theorem}\label{thm:SLn-subexp-classification}
	Let \(n\ge 3\), let \(G=\mathrm{SL}_n(\mathbb Z)\), and let \(H\le G\) be a subgroup. The pair \((G,H)\in \mathbf{SLC}_{\mathrm{subexp}}\) if and only if \(H\) has finite index.
\end{theorem}

\begin{corollary}
		Let \(n\ge 3\), let \(G=\mathrm{SL}_n(\mathbb Z)\), and let \(H\le G\) be a subgroup. The pair \((G,H)\)has pair rapid decay  if and only if \(H\) has finite index.
\end{corollary}

\clearpage
\appendix

\section{The rank-three transvection argument}\label{app:rank3}

\begin{proposition}\label{prop:SL3-transvection}
	Let \(K\le \mathrm{SL}_3(\mathbb Z)\) be finitely generated, amenable, and of exponential growth.
	Then there exist a finite-index subgroup \(K'\le K\) and a subgroup \(U\le K'\) such that
	\[
	U\cong \mathbb Z^2,
	\qquad
	(u-I)^2=0\quad\text{for all }u\in U.
	\]
\end{proposition}

\begin{lemma}\label{lem:SL2-unipotent}
	If \(M\le \mathrm{SL}_2(\mathbb Z)\) consists of unipotent matrices, then there exists
	\(h\in \mathrm{SL}_2(\mathbb Z)\) such that $hMh^{-1}\le \mathrm{UT}_2(\mathbb Z).$
\end{lemma}

\begin{proof}
	If \(M=1\), there is nothing to prove. Assume \(M\neq 1\), and choose
	\(1\neq z\in M\). Since \(z\) is a nontrivial unipotent matrix in \(\mathrm{SL}_2(\mathbb Z)\),
	the space $L:=\ker(z-I)\subset \mathbb Q^2$
	is a one-dimensional rational subspace. Choose a primitive vector
	\(v\in L\cap \mathbb Z^2\), and choose \(h\in \mathrm{SL}_2(\mathbb Z)\)
	such that \(h(v)=e_1\). Then
	\[
	u:=hzh^{-1}=\begin{pmatrix}1&a\\0&1\end{pmatrix}
	\]
	for some \(a\in\mathbb Z\setminus\{0\}\).
	
	Now let \(m\in M\) be arbitrary, and write
	\[
	x:=hmh^{-1}=\begin{pmatrix}p&q\\ r&s\end{pmatrix}\in \mathrm{SL}_2(\mathbb Z).
	\]
	Since \(x\) is unipotent, we have $\operatorname{tr}(x)=p+s=2.$
	Also \(ux=h(zm)h^{-1}\in hMh^{-1}\), so \(ux\) is unipotent as well. Hence
	\[
	2=\operatorname{tr}(ux)=\operatorname{tr}\!\begin{pmatrix}p+ar&q+as\\ r&s\end{pmatrix}
	=p+s+ar=2+ar.
	\]
	Therefore \(ar=0\), and since \(a\neq 0\), we get \(r=0\). Thus \(x\) is upper triangular.
	
	Finally, \(x\in \mathrm{SL}_2(\mathbb Z)\) is unipotent, so its diagonal entries are both \(1\). Hence
	\[
	x=\begin{pmatrix}1&q\\0&1\end{pmatrix}\in \mathrm{UT}_2(\mathbb Z).
	\]
	Since \(m\in M\) was arbitrary, we conclude that $hMh^{-1}\le \mathrm{UT}_2(\mathbb Z).$
\end{proof}

\begin{lemma}\label{lem:SL3-integral-triangularization}
	If \(N\le \mathrm{SL}_3(\mathbb Z)\) is finitely generated and consists of unipotent matrices, then there exists
	\(g\in \mathrm{SL}_3(\mathbb Z)\) such that $gNg^{-1}\le \mathrm{UT}_3(\mathbb Z).$
\end{lemma}

\begin{proof}
	If \(N=1\), there is nothing to prove. Assume \(N\neq 1\).
	Since \(N\) is a finitely generated unipotent linear group, it is nilpotent, hence \(Z(N)\neq 1\).
	Choose \(1\neq z\in Z(N)\), and set $W:=\ker(z-I)\subset \mathbb Q^3.$
	Then \(W\) is a nonzero proper rational \(N\)-invariant subspace. Indeed, if \(n\in N\) and \(w\in W\), then $z(nw)=n(zw)=nw.$ We distinguish two cases.
	
	\smallskip
	
	\noindent\emph{Case 1: \(\dim W=1\).}
	Let \(v_1\in W\cap \mathbb Z^3\) be primitive. Since \(N\) acts unipotently on the one-dimensional space
	\(\mathbb Qv_1\), it fixes \(v_1\) pointwise. Thus \(N\) acts on the quotient
	\(\mathbb Z^3/\mathbb Zv_1\cong \mathbb Z^2\) by unipotent matrices. 
		Let \(\pi:\mathbb Z^3\to \mathbb Z^3/\mathbb Zv_1\cong \mathbb Z^2\) be the quotient map.
	By Lemma~\ref{lem:SL2-unipotent}, the induced action of \(N\) on \(\mathbb Z^3/\mathbb Zv_1\)
	preserves a rank-one direct summand. Let \(\bar L_1\subset \mathbb Z^3/\mathbb Zv_1\) be such a summand, and set $L_2:=\pi^{-1}(\bar L_1).$
	Then \(L_2\) is \(N\)-invariant, \(\mathbb Zv_1<L_2<\mathbb Z^3\), and \(\operatorname{rank}(L_2)=2\).
	Choosing a basis adapted to the flag $0<\mathbb Zv_1<L_2<\mathbb Z^3$
	gives the claim.

	\smallskip
	
	\noindent\emph{Case 2: \(\dim W=2\).}
	Let \(L:=W\cap \mathbb Z^3\), a rank-two \(N\)-invariant lattice.
	The induced action of \(N\) on \(L\) is by unipotent matrices in \(\mathrm{SL}_2(\mathbb Z)\).
	By Lemma~\ref{lem:SL2-unipotent}, there exists a primitive vector \(v_1\in L\) fixed by all elements of \(N\).
	Now apply the argument of Case 1 to the induced action on the quotient \(\mathbb Z^3/\mathbb Zv_1\):
	again we obtain an \(N\)-invariant complete lattice flag $0<\mathbb Zv_1<L_2<\mathbb Z^3.$
	Choosing a basis adapted to this flag gives $gNg^{-1}\le \mathrm{UT}_3(\mathbb Z)$
	for some \(g\in \mathrm{SL}_3(\mathbb Z)\).
\end{proof}

\begin{proof}[Proof of Proposition~\ref{prop:SL3-transvection}]
	Passing to finite index does not change whether a finitely generated group has exponential growth.
	Hence we may replace \(K\) by a finite-index subgroup whenever convenient.
	
	By the Tits alternative \cite{tits1972}, a finitely generated linear group is either virtually solvable or contains a nonabelian free subgroup. Since \(K\) is amenable, it cannot contain \(F_2\), and hence \(K\) is virtually solvable. After passing to a finite-index subgroup, we may therefore assume that \(K\) is solvable. Since \(K\le \mathrm{SL}_3(\mathbb Z)\), a classical theorem of Mal'cev \cite{ceccherini2021} implies that \(K\) is polycyclic. Moreover, as \(K\) is linear over a field of characteristic \(0\), Selberg's lemma \cite{alperin1987,selberg1960} shows that \(K\) has a torsion-free subgroup of finite index. Thus, after passing to finite index once more, we may assume that \(K\) is torsion-free polycyclic. Since \(K\) has exponential growth, Theorem 4.3(1) of Wolf \cite{wolf1968} implies that \(K\) is not virtually nilpotent.

	Let \(G\) be the Zariski closure of \(K\) in \(\mathrm{SL}_3(\mathbb C)\), and replace \(K\) by
	\(K\cap G^\circ\), still of finite index. Then \(G^\circ\) is a connected solvable algebraic group.
	By the Lie--Kolchin theorem \cite{szamuely2017}, there exists \(x\in \mathrm{GL}_3(\mathbb C)\) such that $xG^\circ x^{-1}\le B_3(\mathbb C),$
	where \(B_3\) denotes the upper triangular Borel subgroup. Set $N:=[K,K].$
	Then $
	xNx^{-1}\subset [B_3(\mathbb C),B_3(\mathbb C)]\subset U_3(\mathbb C),$
	so every element of \(N\) is unipotent. Since \(N\le K\) and \(K\) is torsion-free polycyclic,
	\(N\) is a finitely generated torsion-free nilpotent group.
	
	We claim that the Hirsch rank \(h(N)\) satisfies \(h(N)\ge 2\).
	If \(h(N)=0\), then \(N=1\), hence \(K\) is abelian, contradicting exponential growth.
	If \(h(N)=1\), then \(N\cong \mathbb Z\), and conjugation induces a homomorphism
	$K\to \mathrm{Aut}(N)\cong \{\pm 1\}.$
	Its kernel \(K_0\) has finite index in \(K\) and centralizes \(N\). Since
	\([K_0,K_0]\le [K,K]=N\), the quotient \(K_0/N\) is abelian, so \(K_0\) is nilpotent of class at most \(2\).
	Hence \(K_0\), and therefore \(K\), has polynomial growth, contradiction. Thus \(h(N)\ge 2\).
	
	By Lemma~\ref{lem:SL3-integral-triangularization}, after conjugating by an element of \(\mathrm{SL}_3(\mathbb Z)\),
	we may assume that $N\le \mathrm{UT}_3(\mathbb Z).$
	Write
	\[
	u(a,b,c):=
	\begin{pmatrix}
		1 & a & b\\
		0 & 1 & c\\
		0 & 0 & 1
	\end{pmatrix},
	\qquad a,b,c\in \mathbb Z.
	\]
	A direct computation gives $(u(a,b,c)-I)^2=ac\,E_{13}.$
	Hence $(u-I)^2=0$ if and only if $ac=0$.
	Since \(h(N)\ge 2\) and \(h(\mathrm{UT}_3(\mathbb Z))=3\), there are only two possibilities:
	\(h(N)=2\) or \(h(N)=3\).
	
	\smallskip
	
	\noindent\emph{Case 1: \(h(N)=3\).}
	Consider the homomorphism
	\[
	\pi_c:\mathrm{UT}_3(\mathbb Z)\to \mathbb Z,
	\qquad
	\pi_c(u(a,b,c))=c.
	\]
	If \(\pi_c(N)=0\), then \(N\subset \{c=0\}\cong \mathbb Z^2\), contradicting \(h(N)=3\).
	Therefore \(\pi_c(N)\cong \mathbb Z\), and $U:=\ker(\pi_c|_N)$
	has Hirsch rank \(2\). Since \(U\subset \{c=0\}\), every \(u\in U\) satisfies \((u-I)^2=0\).
	Moreover \(U\) is torsion-free abelian of rank \(2\), hence \(U\cong \mathbb Z^2\).
	
	\smallskip
	
	\noindent\emph{Case 2: \(h(N)=2\).}
	Then \(N\) is abelian, because a torsion-free nilpotent group of Hirsch rank \(2\) is abelian.
	In \(\mathrm{UT}_3(\mathbb Z)\) one has
	\[
	[u(a,b,c),u(a',b',c')]=u(0,ac'-a'c,0).
	\]
	Thus the abelianity of \(N\) implies that all pairs \((a,c)\) arising from elements of \(N\)
	lie on a single rank-one subgroup of \(\mathbb Z^2\).
	
	If this subgroup is one of the coordinate axes, then either \(a=0\) for all elements of \(N\),
	or \(c=0\) for all elements of \(N\). In both cases \((u-I)^2=0\) for every \(u\in N\), so we may simply take $U:=N\cong \mathbb Z^2.$
	
	Assume now that we are in the regular case: there exists a primitive pair \((p,q)\in \mathbb Z^2\) with
	\(pq\neq 0\) such that every element of \(N\) has the form \(u(mp,b,mq)\).
	Then the subgroup
	\[
	C:=\{n\in N:(n-I)^2=0\}
	\]
	is exactly
	\[
	C=\{u(0,b,0):b\in \mathbb Z\}\cap N\cong \mathbb Z.
	\]
	Since the condition \((n-I)^2=0\) is conjugacy-invariant and \(N=[K,K]\) is characteristic in \(K\),
	the subgroup \(C\) is normal in \(K\).
	
	Now conjugation yields
	\[
	\rho:K\to \mathrm{Aut}(N)\cong \mathrm{GL}_2(\mathbb Z).
	\]
	Because \(C\) is a \(K\)-invariant rank-one subgroup of \(N\), the image \(\rho(K)\) preserves a rank-one lattice in
	\(\mathbb Z^2\). hence, after choosing a basis of \(N\),
	\[
	\rho(K)\subset
	\left\{
	\begin{pmatrix}
		\pm 1 & *\\
		0 & \pm 1
	\end{pmatrix}
	\right\}.
	\]
	Therefore the induced actions on \(C\) and on \(N/C\cong \mathbb Z\) both have finite image.
	Let \(K_0\le K\) be the finite-index subgroup acting trivially on both \(C\) and \(N/C\).
	Then for every \(k\in K_0\) and \(n\in N\), $knk^{-1}\equiv n \pmod C,$
	so \([k,n]\in C\). Since \(K_0\) centralizes \(C\), we obtain
	\[
	[K_0,N]\subset C\subset Z(K_0).
	\]
	On the other hand,
	\[
	[K_0,K_0]\subset [K,K]=N.
	\]
	Hence
	\[
	\gamma_3(K_0)=[K_0,[K_0,K_0]]\subset [K_0,N]\subset Z(K_0),
	\]
	so \(\gamma_4(K_0)=1\), where \(\gamma_i(K_0)\) denotes the lower central series of \(K_0\). Thus \(K_0\) is nilpotent of class at most \(3\), and therefore has polynomial growth,
	contradicting the exponential growth of \(K\). This contradiction excludes the regular case.
	
	So in all cases we obtain a subgroup \(U\le N\le K\) with
	\[
	U\cong \mathbb Z^2,
	\qquad
	(u-I)^2=0\quad(\forall u\in U).
	\]
	Finally, if we performed an integral conjugation at the triangularization step, we conjugate \(U\) back.
	This preserves both the isomorphism type and the condition \((u-I)^2=0\).
\end{proof}

\newpage
\section*{Acknowledgement}

The author is sincerely grateful to Professor Qin Wang for her helpful comments and valuable suggestions, which have improved the manuscript. The author also thanks Professor Indira Chatterji for correspondence regarding the work in \cite{ChatterjiZarka2024v1}.

\bibliographystyle{plain}
\bibliography{pairrapiddecay.bib}

\bigskip

\address{Jvbin Yao \endgraf
	Research Center for Operator Algebras, School of Mathematical Sciences, East China Normal University, Shanghai 200241, PR China
}

\textit{E-mail address}: \texttt{52285500009@stu.ecnu.edu.cn}




\end{document}